%% file: main.tex
\documentclass[hidelinks,onefignum,onetabnum]{siamart220329}

\input{ex_shared}

\ifpdf
\hypersetup{
	pdftitle={Divergence-free decoupled finite element methods for incompressible flow problems},
	pdfauthor={Volker John, Xu Li, and Christian Merdon}
}
\fi

\usepackage{bm}
\usepackage{amssymb,version}
\usepackage{cases}
\usepackage{color}
\usepackage{verbatim}
\usepackage{fancyvrb}
\usepackage{graphics}
\usepackage{epsfig}
\usepackage{epstopdf}
\usepackage{enumerate}
\usepackage{lipsum}
\usepackage{multirow}
\usepackage{mathrsfs}
\usepackage{graphicx}
\usepackage{subfig}
\usepackage{svg}
\usepackage{todonotes}
\usepackage{mathabx}
\usepackage{booktabs}
\usepackage{hyperref} 

\allowdisplaybreaks

\newcommand{\vecb}[1]{\boldsymbol{#1}}
\newcommand{\RTzerozero}{\widecheck{\boldsymbol{RT}}}
\renewcommand{\(}{\left(}
\renewcommand{\)}{\right)}
\newcommand{\jump}[1]{ [\![ {#1} ]\!]}

\newcommand{\ds}{\, \mathrm{d}\boldsymbol s}

\begin{document}
	\maketitle
	\numberwithin{equation}{section}
	
	\begin{abstract}
    Incompressible flows are modeled by a coupled system of partial differential equations for
    velocity and pressure, Starting from a divergence-free mixed method proposed in 
    [John, Li, Merdon and Rui, Math. Models Methods Appl. Sci. 34(05):919--949, 2024], 
    this paper proposes $\vecb{H}(\mathrm{div})$-conforming finite element methods
    which decouple the velocity and pressure by constructing
    divergence-free basis functions. Algorithmic issues like the computation of this basis 
    and the imposition of non-homogeneous Dirichlet boundary conditions are discussed. 
    Numerical studies at two- and three-dimensional Stokes problems compare the efficiency of the 
    proposed methods with methods from the above mentioned paper. 
\end{abstract}


\input{intro.tex}

\input{div_free_method.tex}

\input{scheme.tex}

\input{non_hom_bc.tex}
\input{numerical_examples.tex}

\input{summary.tex}

\bibliographystyle{siamplain}
\bibliography{lit_doi_deleted}

\end{document}

%% file: ex_shared.tex
\usepackage{lipsum}
\usepackage{amsfonts}
\usepackage{graphicx}
\usepackage{graphbox}
\usepackage{epstopdf}
\usepackage{algorithmic}
\ifpdf
  \DeclareGraphicsExtensions{.eps,.pdf,.png,.jpg}
\else
  \DeclareGraphicsExtensions{.eps}
\fi

\newsiamremark{remark}{Remark}
\newsiamremark{hypothesis}{Hypothesis}
\crefname{hypothesis}{Hypothesis}{Hypotheses}
\newsiamthm{claim}{Claim}

\headers{Divergence-free decoupled FEM for flow}{V. John, X. Li, C. Merdon}

\title{Divergence-free decoupled finite element methods for incompressible flow problems\thanks{\funding{Xu Li was supported by NSFC (No. 12501511). Christian Merdon gratefully acknowledges the funding by
the German Science Foundation (DFG) within the project “ME 4819/2-1”.
}}}

\author{ Volker John \thanks{Weierstrass Institute for Applied Analysis and Stochastics (WIAS), Anton-Wilhelm-Amo Str.
39, 10117 Berlin, Germany and Freie Universit\"at Berlin, Department of Mathematics and
Computer Science, Arnimallee 6, 14195 Berlin, Germany (\email{john@wias-berlin.de}).}
        \and Xu Li \thanks{School of Mathematics, Shandong University, Jinan 250100, China (\email{xulisdu@sdu.edu.cn}).}
        \and Christian Merdon \thanks{ Weierstrass Institute for Applied Analysis and Stochastics (WIAS), Anton-Wilhelm-Amo Str.
39, 10117 Berlin, Germany (\email{merdon@wias-berlin.de}).}
}

\usepackage[]{hyperref}

\definecolor{darkblue}{rgb}{0,0,.7}
\hypersetup{colorlinks=true,linkcolor=darkblue,citecolor=darkblue}


%% file: intro.tex
\section{Introduction}

This paper is concerned with decoupled finite element methods for incompressible flow problems in
a polyhedral domain $\Omega\subset \mathbb{R}^d$, $d \in \{2,3\}$, with Lipschitz boundary.
For the sake of concentrating on the main goal of this paper, the Stokes equations are considered 
\begin{equation}\label{eq:generaleq}
    \begin{array}{rcll}
        -\nu \Delta \vecb{u} +\nabla p & = & \vecb{f} & \mbox{in } \Omega,\\
        \mathrm{div} (\vecb{u}) & = & 0 &  \mbox{in } \Omega,\\
        \vecb{u} & = & \vecb{0} & \mbox{on }\partial \Omega,
    \end{array}
\end{equation}
where the unknowns are the velocity $\vecb{u}$ and the pressure $p$, 
$\nu$ denotes the constant
viscosity, and $\vecb{f}$ is the unit external body force.

The objective of this paper is to investigate a decoupled version of the 
Raviart--Thomas enriched Scott--Vogelius element method proposed in 
\cite{JLMR2022}. Scott--Vogelius elements are a class of well known
and simple divergence-free elements (i.e., the discrete velocity solution 
is weakly divergence-free, which means that the $L^2(\Omega)$ norm of its divergence vanishes, cf. \cite{SV1985,Arnold1992,Zhang:2005}), 
whose structure is exactly 
$\vecb{P}_k\times \vecb{P}_{k-1}^\mathrm{disc}$, that is, the velocity 
space are classical $k$-th order Lagrange elements, while the pressure 
space consists of
discontinuous piecewise polynomials of degree no more than $k-1$. 
As a class of divergence-free methods, it has many fascinating properties
such as pressure-robustness (velocity error estimates are independent of the 
pressure, cf. \cite{JLMNR:2017}) and convection-robustness in the case of the  Navier--Stokes equations 
(the constants in velocity error estimates do not 
depend on the Reynolds number, cf. \cite{Schroeder_2018,GJN21,ABBGLM2021}). 
However, the inf-sup stability of the classical Scott--Vogelius has  severe
requirements on the mesh and the polynomial degree (e.g., see \cite{Joh16}) especially
in three dimensions (3D). This drawback was overcome in \cite{JLMR2022} by enriching 
the velocity space with some suitably chosen Raviart--Thomas bubbles.

It is well known that mixed methods for the Stokes equations result in a 
(non-positive definite) saddle point problem, 
while the decoupling of velocity and pressure only 
requires to solve several symmetric positive definite systems. Due to this 
reason, there is extensive literature studying decoupled methods such 
as the projection methods \cite{Shen1992,GMS2006,GS2003SINUM}, 
consistent-splitting methods \cite{GS2003_jcp,LLP2007,HS2023}, and 
artificial compressibility methods \cite{DLM2017,GP2015,GP2019}. 
These methods are efficient and the structure of their finite element spaces
is usually very standard. However, these methods usually bypass the discretization of 
the continuity equation and thus they do not conserve mass exactly, i.e., the divergence of the computed velocity approximation 
usually does not vanish. 
Moreover, the stability analysis of these methods is much more technical than those of mixed methods and 
in some cases it is still an open problem. Another strategy consists in constructing divergence-free basis
functions directly such as in \cite{Wang2009} 
(see also the discretely divergence-free basis methods \cite{YeHall1997,MU2018}). However, the 3D basis functions in 
\cite{Wang2009} are not linearly independent. In addition, to the best of the 
authors' knowledge, a robust strategy to deal with non-homogeneous
boundary conditions for these classes of methods in 3D is an open problem.
In \cite{FU2019}, by reformulating the decoupled problem as 
a mixed Darcy system, the explicit construction of 
divergence-free basis functions is avoided, 
thereby circumventing the aforementioned difficulties. 
However, this approach requires the viscous and convective terms to be treated explicitly, 
and, unless hybrid-mixed elements are employed together with static condensation, 
one still needs to solve a saddle point problem.

On the first glance it seems to be very attractive to solve a problem for the velocity 
only, instead of a saddle point problem for velocity and pressure, because the saddle point 
character is removed and the number of degrees of freedom is reduced. This was our main motivation for 
developing divergence-free decoupled methods.
The point of departure is the class of methods from \cite{JLMR2022}. 
First, a novel symmetric variant of these methods is proposed. Then, to construct a method with 
a divergence-free basis, a slight modification 
of the enrichment space is introduced. An appropriate decomposition 
of the space of discretely divergence-free functions is derived, which uses for the 3D situation a 
result and algorithm from \cite{Robert2002}. The functions from this space are even weakly divergence-free. This 
decomposition offers a straightforward approach to define and implement weakly divergence-free basis functions.
The imposition of non-homogeneous Dirichlet boundary conditions requires additional efforts. Depending on the 
concrete problem, different approaches can be pursued. In particular, a method is presented that works 
for arbitrary domains in two and three dimensions. The computation of a discrete approximation of the 
pressure is also discussed. 
Finally, the accuracy and efficiency of the proposed decoupled methods are studied by means of numerical 
comparisons with several methods proposed in \cite{JLMR2022}.
These methods differ by the type of system (decoupled vs. saddle point), the sparsity (number of non-zero entries) and sparsity patterns
of the arising matrices, and the overhead that is needed to compute the final solution. 

This paper is organized as follows. In Section~\ref{sec:method_ori} we briefly introduce the 
Raviart--Thomas enriched Scott--Vogelius element method from \cite{JLMR2022}
and a symmetric variant. The decoupled
finite element method is proposed in Section~\ref{sec:scheme}. Two strategies to deal with non-homogeneous
boundary conditions are discussed in Section~\ref{sec:nonhomo_BC}, and a pressure reconstruction procedure, which is inspired by 
\cite{GS2003_jcp,LLP2007}, is given in Section~\ref{sec:pressure}. The numerical studies are presented in Section~\ref{sec:experiments}.

%% file: div_free_method.tex
\section{Notation and divergence-free mixed method from \cite{JLMR2022}}
\label{sec:method_ori}

Standard notation for Lebesgue and Sobolev spaces is used. Vector-valued function spaces and 
functions are written with bold face symbols. 
The symbol $(\cdot,\cdot)$ denotes 
the inner product in $L^2(\Omega)$ or $\vecb{L}^2(\Omega)$.

The methods proposed in this paper are inspired from the methods developed in \cite{JLMR2022}.
This section provides a brief description of the latter methods, it proposes a novel symmetric variant of
these methods, and it introduces the finite element 
spaces used in this paper. 

Denote $\vecb{V}:=\vecb{H}_0^1(\Omega)$, $Q:=L_0^2(\Omega)$, and let $\vecb{f} \in \vecb{L}^2(\Omega)$.  
A variational formulation of the Stokes equations
with homogeneous velocity boundary conditions reads: Find $(\vecb{u},p)\in \vecb{V}\times Q$ such that
\begin{equation*}
    \begin{array}{rcll}
        (\nu \nabla \vecb{u},\nabla \vecb{v}) - (\mathrm{div} (\vecb{v}),p) & = & (\vecb{f},\vecb{v}) & \text{for all } \vecb{v}\in \vecb{V}, \\
        (\mathrm{div} (\vecb{u}),q)                                         & = & 0                   & \text{for all } q\in Q.
    \end{array}
\end{equation*}

Let $\{\mathcal{T}_h\}$ be a family of admissible and  shape-regular partitions of $\Omega$. 
Let $\mathcal T \in \{\mathcal{T}_h\}$. A mesh cell of $\mathcal T$ is denoted with
$T$ and the set of 
facets is denoted with $\mathcal{F}$. 
The symbols  $h_T$ and $h_F$ are used to denote the diameter of $T\in \mathcal{T}$
and $F\in\mathcal{F}$,
respectively. Let $h_\mathcal{T}$ be a piecewise constant function whose restriction on $T$ equals $h_T$ and set $h:=\max_{T\in\mathcal T}{h_T}$.

The space of polynomials of degree $k$, $k\ge 0$, is denoted by 
$P_k(T)$.
Following the notation in \cite{JLMR2022}, the following finite element spaces are defined:
\begin{align*}
    {P}_k(\mathcal{T})                             & :=\left\{q_h\in H^1(\Omega)\ :\ q_h|_T\in P_k(T) \text{ for all } T\in \mathcal{T}\right\} \quad \text{for } k > 0,                                                                                              \\
    {{P}}_k^{\mathrm{disc}}(\mathcal{T})           & :=\left\{q_h\in L^2(\Omega)\ :\ q_h|_T\in P_k(T) \text{ for all } T\in \mathcal{T}\right\},                                                                                               \\
    \widetilde{P}_{k}^{\mathrm{disc}}(\mathcal{T}) & :=\left\{q_h\in {P}_{k}^{\mathrm{disc}}(\mathcal{T})\ :\ (q_h, 1)_T=0 \text{ for all } T\in\mathcal{T}\right\},                                                                           \\
    \vecb{RT}_{k}(T)                               & := \left\lbrace \vecb{v} \in \vecb{L}^2 (T) 
    \ :\ \exists \vecb{p} \in \vecb{P}_k(T), q \in {P}_k(T), \ \vecb{v}|_T(\vecb{x}) = \vecb{p}(\vecb{x}) + q(\vecb{x}) \vecb{x} \right\rbrace, \\
    \vecb{RT}_{k}(\mathcal{T})                     & := \left\lbrace \vecb{v} \in \vecb{H}(\mathrm{div},\Omega)\ :\ \forall T \in \mathcal{T} \, \vecb{v}|_T \in \vecb{RT}_{k}(T) \right\rbrace,
\end{align*}
where $\vecb{H}(\mathrm{div},\Omega)$ is the space of functions from $\vecb{L}^2(\Omega)$ whose divergence is contained in $L^2(\Omega)$. 
The subspace of interior Raviart--Thomas cell functions reads
\begin{align*}
    \vecb{RT}_k^{\mathrm{int}}(\mathcal{T}):=\left\lbrace\vecb{v}\in\vecb{RT}_k(\mathcal{T})\ :\ \vecb{v}\cdot\vecb{n}_T|_{\partial T}=0 \text{ for all } T\in\mathcal{T}\right\rbrace.
\end{align*}
This space can be decomposed into
\begin{align*}
    \vecb{RT}_k^{\mathrm{int}}(\mathcal{T})=\vecb{RT}_{k,0}^{\mathrm{int}}(\mathcal{T})\oplus \widetilde{\vecb{RT}}_k^{\mathrm{int}}(\mathcal{T}),
\end{align*}
where $\vecb{RT}_{k,0}^{\mathrm{int}}(\mathcal{T})$ is the subspace of divergence-free Raviart--Thomas cell functions and
$\widetilde{\vecb{RT}}_k^{\mathrm{int}}(\mathcal{T})$ is an arbitrary complement space to the former one. We
require that all local spaces $\widetilde{\vecb{RT}}_k^{\mathrm{int}}(T)$ of
$\widetilde{\vecb{RT}}_k^{\mathrm{int}}(\mathcal{T})$ come from the same reference space by Piola's transformation, e.g., see \cite[\S~2.1]{BBF:book:2013}.

The operator
\[
\mathrm{div}|_{\widetilde{\vecb{RT}}_k^{\mathrm{int}}(\mathcal{T})}\ : \
    \widetilde{\vecb{RT}}_k^{\mathrm{int}}(\mathcal{T})\rightarrow
    \widetilde{P}_{k}^{\mathrm{disc}}(\mathcal{T})
\]
is a bijective operator, see \cite{JLMR2022} for an explanation of this property.
Consequently there exists an inverse operator
\begin{equation}\label{eq:mathcalR}
\mathcal{R}_k\ : \ \widetilde{P}_{k}^{\mathrm{disc}}(\mathcal{T})
    \rightarrow
    \widetilde{\vecb{RT}}_k^{\mathrm{int}}(\mathcal{T})
\end{equation}
of $\mathrm{div}|_{\widetilde{\vecb{RT}}_k^{\mathrm{int}}(\mathcal{T})}$ such that
for any $q_h\in \widetilde{P}_{k}^{\mathrm{disc}}(\mathcal{T})$,
it holds $\mathrm{div}(\mathcal{R}_kq_h)=q_h$. We extend the definition of
$\mathcal{R}_{k}$ to $Q$ by $\mathcal{R}_kq:=
   \mathcal{R}_k( \pi_{\widetilde{P}_{k}^{\mathrm{disc}}(\mathcal{T})}q)$ 
   such that $\mathrm{div}(\mathcal{R}_kq)=\pi_{\widetilde{P}_{k}^{\mathrm{disc}}(\mathcal{T})}q$ 
   for all $q\in Q$,
where
$\pi_{\mathcal{S}}$ is the $L^2(\Omega)$ projection to a space $\mathcal{S}$.

Next, we recall the method developed and analyzed in \cite{JLMR2022}. For this purpose, define
\begin{align*}
    \vecb{V}_h^\mathrm{ct}:=\vecb{P}_{k}(\mathcal{T})\cap \vecb{V},\quad
    Q_h:=P_{k-1}^\mathrm{disc}(\mathcal{T})\cap Q.
\end{align*}
The pair of spaces $\vecb{V}_h^\mathrm{ct}\times Q_h$ is usually referred to as the Scott--Vogelius
pair \cite{SV:1983}. It is well known that in general Scott--Vogelius pairs do not satisfy a discrete 
inf-sup condition. 
A popular way to achieve an inf-sup stable discretization consists in enriching $\vecb{V}_h^\mathrm{ct}$ with some
suitable bubble functions, such as in the Bernardi--Raugel method \cite{BR1985} for
$k=1$. Instead of using $\vecb{H}^1$-conforming functions, it is advocated in \cite{JLMR2022} to utilize
$\vecb{H}(\mathrm{div})$-conforming Raviart--Thomas functions for enriching 
$\vecb{V}_h^\mathrm{ct}$. This approach does not require any integrals on facets and it leads in 
fact to a simple to implement method where the discrete velocity solution is weakly
divergence-free. 
The fundamental principle for the choice
of the enrichment space $\vecb{V}_h^\mathrm{R}$ is summarized in the following.

For arbitrary $k$, there must be a subspace $\widehat{Q}_h\subseteq Q_h$ so that
$\vecb{V}_h^\mathrm{ct}\times \widehat{Q}_h$ is inf-sup stable, e.g., 
$\widehat{Q}_h=\{0\}$ is always a trivial choice for all $k$.
For higher $k$, $k\geq d$, some classical inf-sup stable pairs on general shape-regular grids imply that $\vecb{V}_h^\mathrm{ct}\times (P_{k-d}^\mathrm{disc}(\mathcal{T})\cap Q)$
is inf-sup stable, cf. \cite[pp.~132--144]{GiraultRaviart:1986}.
Hence any $\widehat{Q}_h\subseteq (P_{k-d}^\mathrm{disc}(\mathcal{T})\cap Q)$
is an admissible choice in this situation. In some special cases $\widehat{Q}_h$ can be even equal to $Q_h$, 
like in the case of
Scott--Vogelius pairs on barycentrically refined grids for $k\geq d$, \cite{Arnold1992,Zhang:2005}.
Define $Q_h^0:=P_0^\mathrm{disc}(\mathcal{T})\cap Q$. It is required in \cite{JLMR2022} that
\begin{equation}\label{eq:qh_hat_q0}
    Q_h^0 \subseteq \widehat{Q}_h \quad \text { for } k\geq d.
\end{equation}
The finite element pressure space $Q_h$ is split by $L^2(\Omega)$-orthogonality as
\begin{align*}
    Q_h=\widehat{Q}_h\oplus_{L^2} \widehat{Q}_h^\perp.
\end{align*}
By \eqref{eq:qh_hat_q0}, one has $\widehat{Q}_h^\perp\subseteq \widetilde{P}_{k-1}^{\mathrm{disc}}(\mathcal{T})$ 
for $k\geq d$.
Then, the principle for choosing $\vecb{V}_h^\mathrm{R}$ consists in finding $\vecb{V}_h^\mathrm{R}\subseteq \vecb{RT}_0(\mathcal{T})\oplus\widetilde{\vecb{RT}}_{k-1}^{\mathrm{int}}(\mathcal{T})$
such that
either
\[
    \mathrm{div} : \vecb{V}_h^\mathrm{R} \rightarrow \widehat{Q}_h^{\perp}
    \quad \text{is bijective}\text{ for } k\geq d,
\]
or
\[
    \mathrm{div} : \vecb{V}_h^\mathrm{R} \rightarrow \widehat{Q}_h^{\perp}
    \quad \text{is surjective} \text{ for } k<d.
\]
For more details, including some explicit constructions of  $\vecb{V}_h^\mathrm{R}$,
the reader is referred to \cite{JLMR2022}.

Define the continuous space $\vecb{V}(h) :=\vecb{V}\times \vecb{V}_h^\mathrm{R}$, the 
finite element space $\vecb{V}_h := \vecb{V}_h^\mathrm{ct}\times \vecb{V}_h^\mathrm{R}$, and 
the bilinear forms
\[
\begin{array}{rcllrcl}
a &: & \vecb{V}\times\vecb{V}\to \mathbb{R} \quad &\mbox{ by }\quad 
&a\(\vecb{u}^\mathrm{ct},\vecb{v}^\mathrm{ct}\)&:=&\(\nabla\vecb{u}^\mathrm{ct},\nabla\vecb{v}^\mathrm{ct}\),\\[0.5ex]
b &:& \vecb{V}(h)\times Q\to \mathbb{R} \quad &\mbox{ by }\quad 
&b\(\vecb{v},q\)&:=&-\(\mathrm{div}\(\vecb{v}^\mathrm{ct}+\vecb{v}^\mathrm{R}\),q\),
\end{array}
\]
for all $\vecb{u}=(\vecb{u}^\mathrm{ct},\vecb{u}^\mathrm{R}),
    \vecb{v}=(\vecb{v}^\mathrm{ct},\vecb{v}^\mathrm{R})\in \vecb{V}(h)$, and
$q\in Q$.
With an abuse of notation, for any $(\vecb{v},q)\in \vecb{H}(\mathrm{div},\Omega)\times Q$,
we define $b(\vecb{v},q):=-(\mathrm{div}(\vecb{v}),q)$.

The finite element method proposed in \cite{JLMR2022} reads as follows: Find $(\vecb{u}_h,p_h) \in \vecb{V}_{h}\times Q_h$ such that
\begin{equation}\label{eq:fullscheme}
    \begin{aligned}
        \nu a_h(\vecb{u}_h, \vecb{v}_h) + b(\vecb{v}_h,p_h)
         & = \left(\vecb{f}, \vecb{v}_h^\mathrm{ct}+\vecb{v}_h^\mathrm{R}\right)
         &                                                            & \text{for all } \vecb{v}_h=\left(\vecb{v}_h^\mathrm{ct},\vecb{v}_h^\mathrm{R}\right) \in \vecb{V}_{h}, \\
        b(\vecb{u}_h, q_h)
         & = 0
         &                                                            & \text{for all } q_h \in Q_{h},
    \end{aligned}
\end{equation}
with
\begin{equation}\label{eqn:def_ah}
    a_h\left(\vecb{u}_h,\vecb{v}_h\right)
    := a\left(\vecb{u}_h^{\mathrm{ct}}, \vecb{v}_h^{\mathrm{ct}}\right)
    - \left(\Delta_\text{pw} \vecb{u}_h^{\mathrm{ct}}, \vecb{v}_h^{\mathrm{R}}\right)
    + \delta \left(\Delta_\text{pw} \vecb{v}_h^{\mathrm{ct}}, \vecb{u}_h^{\mathrm{R}}\right)
    + a_h^\mathrm{S}\left(\vecb{u}_h^\mathrm{R}, \vecb{v}_h^\mathrm{R}\right),
\end{equation}
where $\Delta_\text{pw}$ is the piecewise Laplacian and $\delta=\pm 1$ selects the skew-symmetric
variant of \cite{JLMR2022} (for $\delta = 1$) or the new symmetric variant (for $\delta = -1$).
The stabilizing term is of the form
\begin{equation}\label{eqn:RTk_stabilization}
    a_h^\mathrm{S}\left(\vecb{u}_h^\mathrm{R}, \vecb{v}_h^\mathrm{R}\right):=
    a_h^\mathrm{D}\left(\vecb{u}_h^\mathrm{RT_0}, \vecb{v}_h^\mathrm{RT_0}\right)+ 
    \frac{(1-\delta)\alpha}{2}\left(\mathrm{div}\left(\widetilde{\vecb{u}}_h^\mathrm{R}\right),
    \mathrm{div}\left(\widetilde{\vecb{v}}_h^\mathrm{R}\right)\right).
\end{equation}
Here, $\vecb{v}_h^{\mathrm{RT}_0}$
denotes the lowest order Raviart--Thomas part of $\vecb{v}_h^\mathrm{R}\in \vecb{V}_h^\mathrm{R}$
that only appears and needs to be stabilized in the case $k < d$,
while $\widetilde{\vecb{v}}_h^\mathrm{R}:=\vecb{v}_h^\mathrm{R}-\vecb{v}_h^\mathrm{RT_0}$
denotes the higher-order Raviart--Thomas enrichment of some
$\vecb{v}_h^\mathrm{R}\in \vecb{V}_h^\mathrm{R}$
that only needs to be stabilized in the symmetric case $\delta = -1$.
For the stabilization
$a_h^\mathrm{D}(\cdot,\cdot)$ three spectrally equivalent choices
were suggested in \cite{JLMR2022,li2021low}. One of them, which is used in the numerical studies,  is
\begin{equation}\label{eqn:RT0_stabilization}
    a_h^\mathrm{D}\left(\vecb{u}_h^{\mathrm{RT}_0}, \vecb{v}_h^{\mathrm{RT}_0}\right)
    := \alpha_0 \sum_{F\in\mathcal{F}^0}\mathrm{dof}_F\left(\vecb{u}_h^{\mathrm{RT}_0}\right) \mathrm{dof}_F\left(\vecb{v}_h^{\mathrm{RT}_0}\right) \, \left(\mathrm{div} \vecb{\psi}_F, \mathrm{div} \vecb{\psi}_F\right),
\end{equation}
where $\mathrm{dof}_F: \vecb{H}^1(\Omega)+\vecb{RT}_{k-1}(\mathcal{T})\rightarrow \mathbb{R}$ evaluates the normal flux of the argument on the face $F$
and
the functions $\{\vecb{\psi}_F\in \vecb{RT}_0(\mathcal{T}):
    F\in \mathcal{F}\}$ form the standard basis
of $\vecb{RT}_{0}(\mathcal{T})$ such that
\begin{equation*}
    \vecb{v}^{\mathrm{RT}_0}_h = \sum_{F\in\mathcal{F}}\mathrm{dof}_F(\vecb{v}^{\mathrm{RT}_0}_h)\vecb{\psi}_F
    = \sum_{F\in\mathcal{F}}\mathrm{dof}_F(\vecb{v}^\mathrm{R}_h)\vecb{\psi}_F
    \quad \forall \vecb{v}^\mathrm{R}_h \in \vecb{V}_h^{\mathrm{R}},
\end{equation*}
and the factors $\alpha$ and $\alpha_0$ are parameters that allow to scale the
stabilization and in \cite{JLM2024} it was found that $\alpha_0 \approx 1$
is a good choice.

If $\delta=1$, then $a_h(\cdot,\cdot)$ coincides with the bilinear form from the 
methods studied in \cite{JLMR2022}. If $\delta=-1$, then $a_h(\cdot,\cdot)$ is symmetric and
positive definite whenever $\alpha$ and $\alpha_0$ are sufficiently large. 
The newly introduced term is a stabilization of the 
$\widetilde{\vecb{RT}}_{k-1}^\mathrm{int}(\mathcal{T})$ part that is of magnitude
$(h_\mathcal{T}^{-2}\widetilde{\vecb{u}}_h^\mathrm{R},
    \widetilde{\vecb{v}}_h^\mathrm{R})$. This statement can be proved using the 
inequality 
\begin{equation*}
        \left\|\vecb{v}_h\right\|_{L^2(T)}\leq C h_{T} \left\|\mathrm{div}\left(\vecb{v}_h\right)\right\|_{L^2(T)}\quad \text{for all } T\in \mathcal{T}, 
\end{equation*}
which is valid for all $\vecb{v}_h \in \widetilde{\vecb{RT}}_{k-1}^{\mathrm{int}}(\mathcal{T})$, 
see \cite[Lemma~5.2]{JLMR2022}.
Method \eqref{eq:fullscheme} will be called the full scheme.

Next, define
\begin{displaymath}
    \widehat{\vecb{V}}_{h} := \left\{\vecb{v}_{h}=\left(\vecb{v}_h^{\mathrm{ct}},\vecb{v}_h^{\mathrm{R}}\right)\in
    \vecb{V}_{h}\ :\  \mathrm{div}\left(\vecb{v}_h^{\mathrm{ct}}
    +\vecb{v}_h^{\mathrm{R}}\right)\in \widehat{Q}_h+Q_h^0\right\},
\end{displaymath}
and
\begin{align}\label{eqn:def_VhO}
    \vecb{V}_{h,0} := & \left\lbrace \vecb{v}_h = \left(\vecb{v}_h^{\mathrm{ct}}, \vecb{v}_h^{\mathrm{R}}\right) \in \vecb{V}_h\ :\ b(\vecb{v}_h,q_h) = 0
    \text { for all } q_h\in Q_h \right\rbrace                                                                                                                                              \\
    =                 & \left\lbrace \vecb{v}_h = \left(\vecb{v}_h^{\mathrm{ct}}, \vecb{v}_h^{\mathrm{R}}\right) \in \vecb{V}_h\ :\ \mathrm{div}\left(\vecb{v}_h^{\mathrm{ct}} + \vecb{v}_h^{\mathrm{R}}\right) = 0
    \right\rbrace\subseteq \widehat{\vecb{V}}_{h}.\nonumber
\end{align}
Note that   $\widehat{Q}_h+Q_h^0=\widehat{Q}_h$ for $k\geq d$ because of \eqref{eq:qh_hat_q0}.
Using the inverse of the divergence operator defined in \eqref{eq:mathcalR}, the operator 
\[
\mathcal{R}^\perp\ :\ \vecb{V}_h^\mathrm{ct}\to \vecb{V}_h^\mathrm{R},  \quad
\mathcal{R}^\perp \vecb{v}_h^\mathrm{ct}:=\mathcal{R}_{k-1}
    \pi_{\widehat{Q}_h^\perp}
    \mathrm{div}\left(\vecb{v}_h^\mathrm{ct}\right),
\]    
is defined. Hence it is 
$\mathrm{div}(\mathcal{R}^\perp \vecb{v}_h^\mathrm{ct})=
    \pi_{\widehat{Q}_h^\perp\cap\widetilde{P}_{k-1}^{\mathrm{disc}}(\mathcal{T})}
    \mathrm{div}(\vecb{v}_h^\mathrm{ct})$ and this allows to represent $\widehat{\vecb{V}}_h$, for $k\geq d$, via
\begin{equation}\label{eq:reduced_hatVh_structure}
        \widehat{\vecb{V}}_h=\left\{\vecb{v}_h=\left(\vecb{v}_h^{\mathrm{ct}},-\mathcal{R}^\perp\vecb{v}_h^{\mathrm{ct}}\right)\ :\  \vecb{v}_h^{\mathrm{ct}}\in \vecb{V}_h^{\mathrm{ct}}\right\},
\end{equation}
see \cite[Lemma~6.1]{JLMR2022}.

For deriving a discrete problem that contains only the  velocity, one first observes 
that 
the velocity solution of \eqref{eq:fullscheme} is contained in $\widehat{\vecb{V}}_{h}$,
so that it is sufficient for the continuity constraint to consider test functions from 
$\widehat{Q}_h+Q_h^0$. Thus, for $k\geq d$, problem \eqref{eq:fullscheme} is
equivalent to finding $\vecb{u}_h \in \widehat{\vecb{V}}_{h}$ and $\hat p_h \in \widehat{Q}_h+Q_h^0$,
such that
\begin{equation}\label{eq:galerkin_reducedform0}
    \begin{aligned}
        \nu a_h(\vecb{u}_h, \vecb{v}_h) + b(\vecb{v}_h,\hat{p}_h)
         & = \left(\vecb{f},  \vecb{v}_h^\mathrm{ct}+\vecb{v}_h^\mathrm{R}\right)
         &                                                             & \text{for all } \vecb{v}_h \in \widehat{\vecb{V}}_{h}, \\
        b(\vecb{u}_h, q_h)
         & = 0
         &                                                             & \text{for all } q_h \in \widehat{Q}_h+Q_h^0,
    \end{aligned}
\end{equation}
and, considering only the velocity, of seeking $\vecb{u}_h\in \vecb{V}_{h,0}$ such that
\begin{equation}\label{eq:reduced_scheme}
    \nu a_h(\vecb{u}_h, \vecb{v}_h)
    = \left(\vecb{f}, \vecb{v}_h^\mathrm{ct}+\vecb{v}_h^\mathrm{R}\right)
    \quad \text{for all } \vecb{v}_h \in \vecb{V}_{h,0}.
\end{equation}
Using \eqref{eq:reduced_hatVh_structure}, problem \eqref{eq:galerkin_reducedform0} is also equivalent to
computing $\vecb{u}_h^{\mathrm{ct}} \in {\vecb{V}}_{h}^{\mathrm{ct}}$ and $\hat p_h \in \widehat{Q}_h+Q_h^0$ such that
\begin{equation}\label{eq:galerkin_reducedform}
    \begin{aligned}
        \nu a_h\left(\left(\vecb{u}_h^{\mathrm{ct}},-\mathcal{R}^\perp\vecb{u}_h^{\mathrm{ct}}\right), \left(\vecb{v}_h^{\mathrm{ct}},-\mathcal{R}^\perp\vecb{v}_h^{\mathrm{ct}}\right)\right) - \left(\mathrm{div}\left(\vecb{v}_h^{\mathrm{ct}}\right),\hat{p}_h\right)
         & = \left(\vecb{f}, \vecb{v}_h^{\mathrm{ct}}-\mathcal{R}^\perp\vecb{v}_h^{\mathrm{ct}}\right)
         &                                                                                  & \forall\  \vecb{v}_h^{\mathrm{ct}}\in {\vecb{V}}_{h}^{\mathrm{ct}}, \\
        \left(\mathrm{div}\left(\vecb{u}_h^{\mathrm{ct}}\right), q_h\right)
         & = 0
         &                                                                                  & \forall\  q_h \in \widehat{Q}_h,
    \end{aligned}
\end{equation}
for $k\geq d$. Since $\mathcal{R}^\perp$ is an elementwise computable operator, see 
\cite[Section~6]{JLMR2022}, system \eqref{eq:galerkin_reducedform} can be implemented easily
and solved efficiently.  Choosing $\widehat{Q}_h$
as $Q_h^0$ results in a $\vecb{P}_k\times P_0$ system, which will be called reduced scheme.
For $k<d$ one can similarly obtain a
$\vecb{P}_k\oplus\vecb{RT}_0\times P_0$ system, see \cite[Section~6.2]{JLMR2022} for more
details. Therein a condensation method of the $\vecb{RT}_0$ part is also discussed.

%% file: scheme.tex
\section{Finite element scheme for the velocity based on a divergence-free basis}
\label{sec:scheme}

The aim of the current paper is to construct appropriate pairs $\vecb{V}_h\times Q_h$ in a slightly different
framework such that the basis functions of $\vecb{V}_{h,0}$ are easily computable. In this way,
a decoupled computation of velocity and pressure is possible, where the velocity
equation is given in \eqref{eq:reduced_scheme}. This approach is particularly appealing if only the velocity 
solution is needed. But even if in addition the pressure is of interest,  it is not longer necessary 
to solve a discrete problem of saddle point character. 

\subsection{Derivation of the scheme}
The new approach starts with a modification of the enrichment space
of the full scheme \eqref{eq:fullscheme}. Let
$\RTzerozero_0:=\vecb{RT}_0(\mathcal{T})\cap \vecb{H}_0(\mathrm{div},\Omega),$
where $\vecb{H}_0(\mathrm{div},\Omega)$ is the subspace of functions from $\vecb{H}(\mathrm{div},\Omega)$ with 
zero normal trace on $\partial\Omega$. 
Now, the original enrichment space $\vecb{V}_h^\mathrm{R}$ is replaced with
\begin{align*}
    \vecb{V}_h^\mathrm{MR}:=\vecb{V}_h^\mathrm{R}+\RTzerozero_0,
\end{align*}
such that $\vecb{V}_h^\mathrm{MR}=\vecb{V}_h^\mathrm{R}$ for $k<d$ (see \cite{JLMR2022}) and
$\vecb{V}_h^\mathrm{MR}=\vecb{V}_h^\mathrm{R}\oplus\RTzerozero_0$ for $k\geq d$.
By abuse of notation, we 
define $\vecb{V}_h:=\vecb{V}_h^\mathrm{ct}\times \vecb{V}_h^\mathrm{MR}$
and $\vecb{V}(h):=\vecb{V}\times \vecb{V}_h^\mathrm{MR}$.

For the moment, the new scheme can be written in form \eqref{eq:fullscheme}, with the modification mentioned above: Find $(\vecb{u}_h,p_h) \in \vecb{V}_{h}\times Q_h$, such that
\begin{equation}\label{eq:fullscheme_new}
    \begin{aligned}
        \nu a_h(\vecb{u}_h, \vecb{v}_h) + b(\vecb{v}_h,p_h)
         & = \left(\vecb{f}, \vecb{v}_h^\mathrm{ct}+\vecb{v}_h^\mathrm{R}\right)
         &                                                            & \text{for all } \vecb{v}_h=\left(\vecb{v}_h^\mathrm{ct},\vecb{v}_h^\mathrm{R}\right) \in \vecb{V}_{h}, \\
        b(\vecb{u}_h, q_h)
         & = 0
         &                                                            & \text{for all } q_h \in Q_{h}.
    \end{aligned}
\end{equation}

Next, a set of the computable basis functions
of the (new) space $\vecb{V}_{h,0}$ is derived and then the decoupling of velocity and 
pressure computation is performed. To this end, 
define
\begin{align*}
    \vecb{RT}_0^0:=\left\{\vecb{v}\in\vecb{RT}_0(\mathcal{T})\ :\ 
    \nabla\cdot\vecb{v}=0\right\},\quad
    \RTzerozero_0^0:=\left\{\vecb{v}\in\RTzerozero_0\ :\
    \nabla\cdot\vecb{v}=0\right\}.
\end{align*}

\begin{remark}[Basis for $\vecb{RT}_0^0$]
    \label{rem:basisRTdivfree}
In the two-dimensional situation, it is known, e.g., see \cite{Arnold1997} and \cite[Theorem 3.4]{Wang2009}, that
\begin{align*}
    \vecb{RT}_0^0=\mathrm{curl}\left(P_1(\mathcal{T})\cap H^1(\Omega)/\mathbb{R} \right)
    \quad \text{and} \quad
    \RTzerozero_0^0=\mathrm{curl}\({P_1(\mathcal{T})\cap H_0^1(\Omega)}\).
\end{align*}
Moreover, the curl of the nodal basis of ${P_1(\mathcal{T})\cap H_0^1(\Omega)}$
forms a basis for $\RTzerozero_0^0$.
In three dimensions, things become much more complicated due to the large kernel of the curl operator.
In this case, 
$\vecb{RT}_0^0$ and $\RTzerozero_0^0$ are
two subspaces of $\mathrm{curl} \vecb{N}_0$,
with $\vecb{N}_0$ being the lowest order N\'{e}d\'{e}lec space, see \cite{Robert2002}. In
\cite{Robert2002}, the basis functions of $\vecb{N}_0$
and some tools from graph theory are utilized for constructing linearly independent basis functions of
$\vecb{RT}_0^0$ and $\RTzerozero_0^0$, which are also employed
in our implementation of the proposed methods.
\end{remark}

The following lemma shows a decomposition of the space of discretely divergence-free functions
given by
\begin{align*}
    \vecb{V}_{h,0} := \bigl\lbrace \vecb{v}_h \in \vecb{V}_h : b(\vecb{v}_h,q_h) = 0 \text{ for all } q_h \in Q_h \bigr\rbrace.
\end{align*}

\begin{lemma}[Characterization of the space of discretely divergence-free functions]\label{lem:char_disc_divfree}
It holds
    \begin{equation*}
        \vecb{V}_{h,0}=\vecb{V}_{h,0}^\mathrm{ct} 
        \oplus \left(\vecb{0}\times\RTzerozero_0^0\right),
    \end{equation*}
with
\begin{equation}\label{eq:V_h0_st}
\vecb{V}_{h,0}^\mathrm{ct}:=\mathrm{span}\left\{\left(\vecb{v}_h^\mathrm{ct},
        -\mathcal{R}_{k-1}\mathrm{div}\left(\vecb{v}_h^\mathrm{ct}\right)-\Pi_h^{\mathrm{RT}_0}\vecb{v}_h^\mathrm{ct}\right)
\        :\ \vecb{v}_h^\mathrm{ct}\in \vecb{V}_h^\mathrm{ct}\right\}
\subset \vecb{V}_h
\end{equation}
and $\Pi_h^{\mathrm{RT}_0}$ being the standard lowest-order Raviart--Thomas interpolation, i.e.,
\begin{equation*}
    \Pi_h^{\mathrm{RT}_0} \vecb{v}^{\mathrm{ct}}_h := \sum_{F\in\mathcal{F}}\mathrm{dof}_F(\vecb{v}_h^\mathrm{ct})\vecb{\psi}_F
    \quad \forall \vecb{v}^\mathrm{ct}_h \in \vecb{V}_h^{\mathrm{ct}}.
\end{equation*}
\end{lemma}

\begin{proof}
First it will be proved that $\vecb{V}_{h,0}^\mathrm{ct}\cap (\vecb{0}\times\RTzerozero_0^0)=\{\vecb{0}\}$.
    For any $\vecb{v}_h=(\vecb{v}_h^\mathrm{ct},\vecb{v}_h^\mathrm{R})\in \vecb{V}_{h,0}^\mathrm{ct}\cap (\vecb{0}\times\RTzerozero_0^0)$,
    we have $\vecb{v}_h^\mathrm{ct}=\vecb{0}$ and hence also
    $\vecb{v}_h^\mathrm{R} = -\mathcal{R}_{k-1}\mathrm{div}(\vecb{v}_h^\mathrm{ct})-\Pi_h^{\mathrm{RT}_0}\vecb{v}_h^\mathrm{ct}=\vecb{0}$.

The next step consists in showing that $\vecb{V}_{h,0}^\mathrm{ct}\oplus
        (\vecb{0}\times\RTzerozero_0^0)\subseteq \vecb{V}_{h,0}$.
    Let $\vecb{v}_h^\mathrm{ct} \in \vecb{V}_h^\mathrm{ct}$, then the definition of the 
    extension of the inverse divergence operator \eqref{eq:mathcalR}  yields 
\[    
\mathrm{div}\left(\mathcal{R}_{k-1}\mathrm{div}\left(\vecb{v}_h^\mathrm{ct}\right)\right)=
        \pi_{\widetilde{P}_{k-1}^{\mathrm{disc}}(\mathcal{T})}\mathrm{div}\left(\vecb{v}_h^\mathrm{ct}\right)
\]
and for the projection to $\vecb{RT}_0(\mathcal{T})$ it is known that 
\[
\mathrm{div}\left(\Pi_h^{\mathrm{RT}_0}\vecb{v}_h^\mathrm{ct}\right)=
        \pi_{Q_h^0}\mathrm{div}\left(\vecb{v}_h^\mathrm{ct}\right).
\]
Since $\mathrm{div}(\vecb{v}_h^\mathrm{ct}) \in P_{k-1}^\mathrm{disc}(\mathcal{T})\cap Q = \widetilde{P}_{k-1}^{\mathrm{disc}}(\mathcal{T}) \oplus_{L^2} Q_h^0$, 
it follows that 
    $\mathrm{div}(\vecb{v}_h^\mathrm{ct}
        -\mathcal{R}_{k-1}\mathrm{div}(\vecb{v}_h^\mathrm{ct})
        -\Pi_h^{\mathrm{RT}_0}\vecb{v}_h^\mathrm{ct})=0$,
which demonstrates that $\vecb{V}_{h,0}^\mathrm{ct}\oplus
        (\vecb{0}\times\RTzerozero_0^0)\subseteq \vecb{V}_{h,0}$.

It remains to prove that $\vecb{V}_{h,0}\subseteq
        \vecb{V}_{h,0}^\mathrm{ct}\oplus
        (\vecb{0}\times\RTzerozero_0^0)$.
Consider an arbitrary $\vecb{v}_h=(\vecb{v}_h^\mathrm{ct},\vecb{v}_h^\mathrm{R})\in \vecb{V}_{h,0}$. Since it can be split as
\begin{align*}
    \vecb{v}_h=\left(\vecb{v}_h^\mathrm{ct},-\mathcal{R}_{k-1}\mathrm{div}(\vecb{v}_h^\mathrm{ct})
    -\Pi_h^{\mathrm{RT}_0}\vecb{v}_h^\mathrm{ct}\right)+\left(\vecb{0},\vecb{v}_h^\mathrm{R}+\mathcal{R}_{k-1}\mathrm{div}(\vecb{v}_h^\mathrm{ct})
    +\Pi_h^{\mathrm{RT}_0}\vecb{v}_h^\mathrm{ct}\right),
\end{align*}
it suffices to show that
\begin{equation}\label{eq:lemma1_0}
\vecb{v}_h^\mathrm{R}+\mathcal{R}_{k-1}\mathrm{div}\left(\vecb{v}_h^\mathrm{ct}\right)
        +\Pi_h^{\mathrm{RT}_0}\vecb{v}_h^\mathrm{ct}\in \RTzerozero_0^0.
\end{equation}
By construction, the functions from $\vecb{V}_{h,0}$ are even weakly divergence-free, since a property of the form \eqref{eqn:def_VhO} is valid also with the modified enrichment space.
Using $\mathrm{div}(\vecb{v}_h^\mathrm{R})  =  -\mathrm{div}(\vecb{v}_h^\mathrm{ct})$
and $\mathrm{div}(\vecb{v}_h^\mathrm{ct}) \in {P}_{k-1}^{\mathrm{disc}}(\mathcal{T})$
gives   $\mathrm{div}(\vecb{v}_h^\mathrm{R}) \in {P}_{k-1}^{\mathrm{disc}}(\mathcal{T})$
and
\begin{align*}
\mathrm{div}\left(\vecb{v}_h^\mathrm{R}+
        \mathcal{R}_{k-1}\mathrm{div}\left(\vecb{v}_h^\mathrm{ct}\right)+\Pi_h^{\mathrm{RT}_0}\vecb{v}_h^\mathrm{ct}\right) =  
        \mathrm{div}\left(\vecb{v}_h^\mathrm{R}\right)-
        \pi_{\widetilde{P}_{k-1}^{\mathrm{disc}}(\mathcal{T})}
        \mathrm{div}\left(\vecb{v}_h^\mathrm{R}\right)- \pi_{P_0^\mathrm{disc}(\mathcal{T})}\mathrm{div}\left(\vecb{v}_h^\mathrm{R}\right) = 0.
\end{align*}
In addition, it can be seen from this equation that  $\mathrm{div} (\vecb{v}_h^\mathrm{R}+
        \mathcal{R}_{k-1}\mathrm{div}(\vecb{v}_h^\mathrm{ct})) \in   Q_h^0
$ (implying that the $\widetilde{\vecb{RT}}_{k-1}^{\mathrm{int}}(\mathcal{T})$ part of $\vecb{v}_h^\mathrm{R}+
        \mathcal{R}_{k-1}\mathrm{div}(\vecb{v}_h^\mathrm{ct})$ vanishes), thus $\vecb{v}_h^\mathrm{R}+
        \mathcal{R}_{k-1}\mathrm{div}(\vecb{v}_h^\mathrm{ct}) \in \RTzerozero_0$ and hence the left-hand side of \eqref{eq:lemma1_0} is also a lowest order Raviart--Thomas function. Consequently \eqref{eq:lemma1_0} is shown. 
\end{proof}

Given the conforming velocity space $\vecb{V}_h^\mathrm{ct}$, then the representation of $\vecb{V}_{h,0}$ in Lemma~\ref{lem:char_disc_divfree} offers
a straightforward way to implement this space. The operators $\mathcal{R}_{k-1}$ and $\Pi_h^{\mathrm{RT}_0}$ are
just cellwise and define a linear operator $\Pi^\text{R} : \vecb{V}_h^\text{ct} \rightarrow \vecb{V}_h^R$
given by
\begin{align*}
    \Pi^\text{R} := -\mathcal{R}_{k-1}\mathrm{div}(\vecb{v}_h^\mathrm{ct})-\Pi_h^{\mathrm{RT}_0}\vecb{v}_h^\mathrm{ct}
\end{align*}
with a sparse matrix representation $R$.
Altogether, this leads to the pressure-free formulation, which has the same velocity solution as 
\eqref{eq:fullscheme_new}:
seek $(\vecb{u}_h^\mathrm{ct}, \vecb{u}_h^\text{RT0}) \in \vecb{V}_h^\text{ct} \times \RTzerozero^0_0$
such that
\begin{align}\label{eq:pressure_free_scheme}
  \nu  a_h \left((\vecb{u}_h^\text{ct}, \Pi^\text{R} \vecb{u}_h^\text{ct} + \vecb{u}_h^\text{RT0}), (\vecb{v}_h^\text{ct}, \Pi^\text{R} \vecb{v}_h^\text{ct} + \vecb{v}_h^\text{RT0})\right)
    = \left(\vecb{f}, \vecb{v}_h^\text{ct} + \Pi^\text{R} \vecb{v}_h^\text{ct} + \vecb{v}_h^\text{RT0} \right)
\end{align}
for all $(\vecb{v}_h^\text{ct}, \vecb{v}_h^\text{RT0}) \in \vecb{V}_h^\text{ct} \times \RTzerozero^0_0$.
A basis for $\RTzerozero^0_0$ is discussed in Remark~\ref{rem:basisRTdivfree}
and we also consider the representation matrix $S$ of the change of basis
from $\RTzerozero^0_0$ to $\vecb{RT}_0$.

Algebraically a linear system of the form
\begin{align}\label{eqn:linear_system}
    \begin{pmatrix}
        A^\text{ct,ct} + A^\text{ct,R} R + R^T A^\text{R,ct} + R^T A^\text{R,R} R & A^\text{ct,R} S + R^T A^\text{R,R} S\\
        S^T A^\text{R,ct} + S^T A^\text{R,R} R & S^T A^\text{R,R} S
    \end{pmatrix}
    \begin{pmatrix}
        \vecb{x}^\text{ct}\\
        \vecb{x}^{\RTzerozero^0_0}
    \end{pmatrix}
    = \begin{pmatrix}
    \vecb{b}^\text{ct} + R^T \vecb{b}^\text{R}\\
    S^T \vecb{b}^{\mathrm{RT}_0}
    \end{pmatrix}
\end{align}
needs to be solved. Here, the $A$ blocks represent the corresponding
terms in \eqref{eqn:def_ah} that also need to be implemented for the original method
\eqref{eq:fullscheme_new}. The vectors $\vecb{b}^\text{ct}, \vecb{b}^\text{R}$,
and $\vecb{b}^{\mathrm{RT}_0}$ correspond to the
evaluation of the right-hand side $(\vecb{f}, \bullet)$ in the basis functions of $\vecb{V}_h^\text{ct}$, $\vecb{V}_h^\text{R}$,
and $\vecb{RT}_0$, respectively. Notice that, depending on $\delta$, it holds either
symmetry or skew-symmetry of $A$ in the sense that $A^\text{ct,R} = -\delta (A^\text{R,ct})^T$.

\subsection{Error estimate} 

The relationship between the cases $\delta=-1$ and $\delta=1$ in \eqref{eq:fullscheme_new} 
(or \eqref{eq:pressure_free_scheme}) 
is very similar to the one between symmetric interior penalty Galerkin (SIPG) and 
nonsymmetric interior penalty Galerkin (NIPG), cf. \cite{ABCM2002}. 
Apart from the symmetry property, the primary difference between SIPG and NIPG is that, 
in SIPG, the penalty parameter must be chosen sufficiently large to guarantee stability (coercivity).
On the other hand, the coercivity, boundedness, and consistency properties of $a_h$ with $\delta=1$ 
has been thoroughly analyzed in \cite{JLMR2022}, although the velocity space therein is 
slightly different from the one in \eqref{eq:fullscheme_new} when $k\geq d$. Therefore 
we only list some main results here and their proofs are omitted.

Define the norm $|||\bullet |||_\star$ on $\vecb{V} \times \vecb{V}_h^\mathrm{R}$ as
\begin{align*}
   ||| \vecb{v} |||^2_\star := a_h(\vecb{v}, \vecb{v}) +\|h_\mathcal{T}\Delta_\text{pw} \vecb{v}^{\mathrm{ct}}\|^2+\left\|\mathrm{div}(\widetilde{\vecb{v}}^\mathrm{R})\right\|^2,
\end{align*}
for all $\vecb{v}=:(\vecb{v}^\mathrm{ct}, \vecb{v}^\mathrm{R})=:(\vecb{v}^\mathrm{ct}, \vecb{v}^{\mathrm{RT}_0}+\widetilde{\vecb{v}}^\mathrm{R})\in \vecb{V} \times \vecb{V}_h^\mathrm{R}$.

\begin{lemma} [Properties of $a_h$]
    Assume that the penalty parameter $\alpha$ in $a_h$ is sufficiently large when $\delta=-1$. 
    Then, the bilinear form $a_h$ is coercive, bounded, and consistent in the sense that 
    \begin{align*}
    \hspace{3cm} a_h(\vecb{v}_h,\vecb{v}_h)& \gtrsim |||\vecb{v}_h|||_\star^2 && \text{for all } \vecb{v}_h\in \vecb{V}_{h,0},\hspace{3cm}\\
    a_h(\vecb{u},\vecb{v}) & \lesssim |||\vecb{u}|||_\star|||\vecb{v}|||_\star && \text{for all } \vecb{u}, \vecb{v}\in \vecb{V} \times \vecb{V}_h^\mathrm{R},
    \end{align*}
    and  
    \begin{align*}
    \nu a_h((\vecb{u},\vecb{0}), \vecb{v}_h) = (\vecb{f},\vecb{v}_h^\mathrm{ct}+\vecb{v}_h^\mathrm{R}) \quad \text{for all } \vecb{v}_h \in \vecb{V}_{h,0},
    \end{align*}
    respectively, where $\vecb{u}\in \vecb{V}$ is the velocity solution of \eqref{eq:generaleq}.
\end{lemma}

In addition, since the velocity space used in our paper includes the one in \cite{JLMR2022}, 
an inf-sup stability which is similar to \cite[Theorem 4.1]{JLMR2022} is satisfied, i.e., 
\begin{align*}
    \sup_{\vecb{v}_h \in \vecb{V}_h} \frac{b(\vecb{v}_h,q_h)}{||| \vecb{v}_h |||_\star} \geq \beta \|q_h\|
    \quad \text{for all } q_h \in Q_h,
\end{align*}
where $\beta>0$ is a constant independent of $h$.

The above properties, together with the approximation properties of $\vecb{V}_h^\mathrm{ct}$, e.g., see \cite{BrennerScott:2008}, 
yield the following error estimates, derived in the same way as in \cite[Section~5]{JLMR2022}.

\begin{theorem} \label{thm:velo_error}
    Let $(\vecb{u},p)$ and $(\vecb{u}_h,p_h)$ be the solutions of \eqref{eq:generaleq} and \eqref{eq:pressure_free_scheme}, 
    respectively, and consider a family of shape-regular triangulations $\{\mathcal T_h\}_{h>0}$. Assume that $\alpha$ is sufficiently large when $\delta=-1$ and 
    $(\vecb{u},p)\in \vecb{H}^{k+1}(\Omega)\times H^k(\Omega)$, then
    \[
    |||(\vecb{u},\vecb{0})-\vecb{u}_h|||_\star  \leq C h^k |\vecb{u}|_{\vecb{H}^{k+1}(\Omega)},
    \]
    where $C>0$ is independent of $h$ and $\nu$.
\end{theorem}

%% file: non_hom_bc.tex
\section{Dealing with non-homogeneous boundary conditions}
\label{sec:nonhomo_BC}

The incorporation of non-homoge\-neous Dirichlet boundary conditions in the proposed divergence-free 
decoupled methods is not straightforward. This section presents a simple approach, which is restricted 
to two-dimensional problems, and a general strategy suitable for two- and three-dimensional problems. 

The discrete solution $\vecb{u}_h$ in the novel method can be split into three parts, compare the 
characterization of $\vecb{V}_{h,0}$  in Lemma~\ref{lem:char_disc_divfree}: The $\vecb{P}_k$ part
(consider the first component of $\vecb{V}_{h,0}^\mathrm{ct}$), a
divergence-free $\vecb{RT}_0$ part (second part of $\vecb{V}_{h,0}$), and a function in $\vecb{RT}_{k-1}(\mathcal{T})$ (second component of $\vecb{V}_{h,0}^\mathrm{ct}$), whose divergence is exactly the additive inverse of 
the divergence of the $\vecb{P}_k$ part.
Once the $\vecb{P}_k$ part is given, the third part of $\vecb{u}_h$ is also fixed, compare the definition 
\eqref{eq:V_h0_st} of $\vecb{V}_{h,0}^\mathrm{ct}$. 
When dealing 
with non-homogeneous boundary conditions, the $\vecb{P}_k$ boundary part  can be defined by the usual interpolations/projections, and then the third part  is 
determined by it. Denoting the sum of these parts by $\vecb{u}_h^\mathrm{ct+}$, then 
it remains to determine a contribution  $\check{\vecb{u}}_h^\mathrm{R}$  from the divergence-free $\vecb{RT}_0$ part such that
\begin{align*}
    \int_F\left(\vecb{u}_h^\mathrm{ct+}+\check{\vecb{u}}_h^\mathrm{R}\right)\cdot \vecb{n}\ \ds = \int_F\vecb{u} \cdot \vecb{n}\ \ds \quad\text{for all } F\subset\partial\Omega.
\end{align*}
This condition determines the degrees of freedom of $\check{\vecb{u}}_h^\mathrm{R}$ on each facet at the boundary. 

Altogether, the task of imposing non-homogeneous Dirichlet boundary conditions is reduced to the following general problem: For any function $g$ on $\partial\Omega$, 
find an inexpensive strategy to get a function $\vecb{z}_h\in \vecb{RT}_0^0$ such that $\int_F \vecb{z}_h\cdot\vecb{n}\ \ds=
\int_F g\ \ds$ for all $F\subset\partial\Omega$. This function can be used to impose the boundary values of the divergence-free $\vecb{RT}_0$ part.
Having decomposed in this way $\check{\vecb{u}}_h^\mathrm{R} = \check{\check{\vecb{u}}}_h^\mathrm{R}+\vecb{z}_h$, where 
$\check{\check{\vecb{u}}}_h^\mathrm{R}\in \RTzerozero^0_0$ takes homogeneous boundary conditions, 
leads to a change of the right-hand side in \eqref{eqn:linear_system} to
\begin{align*}
    \begin{pmatrix}
        \vecb{b}^\text{ct} + R^T \vecb{b}^\text{R} - A^\text{ct,R} \vecb{x}_z\\
        S^T \vecb{b}^{\mathrm{RT}_0} - S^T A^\text{R,R} \vecb{x}_z
    \end{pmatrix},
\end{align*}
where $\vecb{x}_z$ are the coefficients of $\vecb{z}_h$ w.r.t.\ the $\vecb{RT}_0$ basis (hence a transform
of the basis via $S$ does not appear).
After having solved the modified system \eqref{eqn:linear_system}, one needs
to add $\vecb{z}_h$ to its solution.

Below two strategies for the construction of $\vecb{z}_h$ are discussed.

\subsection{A simple strategy in two dimensions}\label{sec:nonhomo_BC_2d}
A simple strategy to deal with non-homogeneous boundary conditions in two dimensions was proposed in \cite[Sect.~3.3]{Wang2009}. Note that for any 
$\vecb{z}_h=\mathrm{curl} \phi_h$ with $\phi_h\in P_1(\mathcal{T})$, the tangential derivative
satisfies $\partial\phi_h / \partial \tau =\vecb{z}_h\cdot\vecb{n}$, and thus
\begin{align*}
    \phi_h(\vecb{x})=\phi_h(\vecb{x}_0)+\int_{\vecb{x}_0}^{\vecb{x}} \vecb{z}_h \cdot \vecb{n} \ \ds, \quad \vecb{x} \in \partial \Omega,
\end{align*}
where both $\vecb{x}_0\in \partial\Omega$ and $\phi_h(\vecb{x}_0)$ are arbitrary. 
Thus, fixing a vertex $\vecb{x}_0$ at some component of $\partial\Omega$ and a value $\phi_h(\vecb{x}_0)$ 
determines the value of $\phi_h$ at all other vertices at this boundary component. After having computed 
in this way values at all vertices on each component of $\partial\Omega$, one can assign some value to each 
vertex in $\Omega$, e.g., zero, defining in this way $\phi_h$ and then $\vecb{z}_h$ can be computed. 

\subsection{A robust strategy for both two- and three-dimensional spaces}\label{sec:nonhomo_BC_3d}
For any function $g$ on $\partial\Omega$, consider the following problem: Find $(\vecb{z}_h,r_h)\in \vecb{RT}_0(\mathcal{T})\times Q_h^0$ 
such that 
\begin{align}\label{eq:findingzh}
    \int_F \vecb{z}_h\cdot\vecb{n}\ \ds&=\int_F g\ \ds
    && \text{for all }  F\subset\partial\Omega,\\
    D_h(\vecb{z}_h,\vecb{v}_h)-(r_h,\mathrm{div}(\vecb{v}_h))+(\mathrm{div}(\vecb{z}_h),q_h)&=0
    && \text{for all } (\vecb{v}_h,q_h)\in \RTzerozero_0\times Q_h^0,
\end{align}
where 
\begin{align*}
    D_h(\vecb{z}_h, \vecb{v}_h)
    := \sum_{F\in\mathcal{F}}\mathrm{dof}_F(\vecb{z}_h) \mathrm{dof}_F(\vecb{v}_h) \, \(\vecb{\psi}_F, \vecb{\psi}_F\).
\end{align*}
In \cite{li2021low} it is proved that $D_h(\bullet, \bullet)$ and $(\bullet, \bullet)$ are spectrally equivalent on $\vecb{RT}_0(\mathcal{T})$, which,
together with the inf-sup stability of $\RTzerozero_0\times Q_h^0$ (cf. \cite{BBF:book:2013}), guarantees the unique solvability 
of the Darcy-type problem \eqref{eq:findingzh}. Notice that the matrix from $D_h$ is diagonal
with positive entries on its diagonal. Hence, one obtains a symmetric positive definite (s.p.d.) Schur
complement matrix and one could compute $r_h$ by using a solver for s.p.d. systems.
And since $D_h$ is a diagonal matrix, one can obtain thereafter $\vecb{z}_h$ locally from $r_h$. Clearly we have $\vecb{z}_h\in \vecb{RT}_0^0$.

Computing the boundary values by solving \eqref{eq:findingzh} requires
the solution of a global Darcy-type problem, which is not efficient. 
A much more efficient approach is solving a similar problem on 
a much smaller subdomain near the boundary. 
Define 
\begin{align*}
    \mathcal{T}_\partial=\left\{T\in\mathcal{T}: \partial T\cap \partial\Omega\neq \emptyset\right\},
\end{align*}
that is, $\mathcal{T}_\partial$ is the union of all vertex patches (see, e.g., \cite{LLMS2017}) of the boundary vertices and let $\Omega_\partial=\cup_{T\in \mathcal{T}_\partial} T$. 
Then, the following much smaller system also provides a solution with the desired properties: Find $(\vecb{z}_h,r_h)\in \vecb{RT}_0(\mathcal{T}_\partial)\times (P_0^\mathrm{disc}(\mathcal{T}_\partial)\cap L_0^2(\Omega_\partial))$ 
such that
\begin{align}
    \int_F \vecb{z}_h\cdot\vecb{n}\ \ds& =\int_F g\ \ds
    && \text{for all } F\subset\partial\Omega,\label{eq:patchboundary}\\
    \int_F \vecb{z}_h\cdot\vecb{n}\ \ds&=0
    && \text{for all } F\subset\partial\Omega_\partial\setminus\partial\Omega,\label{eq:patchboundary2}\\
    D_h(\vecb{z}_h,\vecb{v}_h)-(r_h,\mathrm{div}(\vecb{v}_h))+(\mathrm{div}(\vecb{z}_h),q_h)&=0
    && \text{for all } (\vecb{v}_h,q_h)\in (\vecb{RT}_0(\mathcal{T}_\partial)\cap \vecb{H}_0(\mathrm{div},\Omega_\partial))\label{eq:findingzh-small}\\
    &&& \hspace{3cm} \times (P_0^\mathrm{disc}(\mathcal{T}_\partial)\cap L_0^2(\Omega_\partial)).\nonumber
\end{align}
If $\Omega$ is simply connected, $\Omega_\partial$ must be a connected domain. Then the inf-sup stability of
the pair $(\vecb{RT}_0(\mathcal{T}_\partial)\cap \vecb{H}_0(\mathrm{div},\Omega_\partial))\times (P_0^\mathrm{disc}(\mathcal{T}_\partial)\cap L_0^2(\Omega_\partial))$
can be guaranteed, too. Thus \eqref{eq:patchboundary}--\eqref{eq:findingzh-small} is a well-posed problem. If $\Omega$ is multiply connected,
$\Omega_\partial$ can be disconnected (consisting of several connected components). In this case one just needs to solve several independent systems like \eqref{eq:patchboundary}--\eqref{eq:findingzh-small}.

\section{Computing the discrete pressure}
\label{sec:pressure}
In many situations the pressure solution is also important. Here we discuss one strategy to approximate the pressure, which is inspired by 
a class of consistent splitting schemes proposed in \cite{GS2003_jcp,LLP2007}. It seeks a pressure solution 
in $H^1(\Omega)/\mathbb{R}$. It relies
on the identity
$\Delta \vecb{u} = \nabla(\nabla\cdot\vecb{u}) - \nabla\times(\nabla\times\vecb{u})$
which allows to replace the Laplacian in \eqref{eq:generaleq}
by the curl-curl operator. Then, testing the momentum
with $\nabla q$ for any $q \in H^1(\Omega)$ yields
\begin{equation}\label{eq:pres}
    (\nabla p,\nabla q) = -(\nu\nabla\times\nabla\times\vecb{u},\nabla q)+(\vecb{f},\nabla q)=\nu\int_{\partial\Omega} (\nabla\times\vecb{u}) \cdot (\vecb{n}\times \nabla q )\ \ds+(\vecb{f},\nabla q).
\end{equation}
This leads to the following pressure reconstruction scheme:
Find $p_h \in P_{\ell}(\mathcal{T}) \cap H^1(\Omega) \cap Q$ with $\ell\geq 1$ such that
\begin{equation}\label{eq:pres_reco}
    (\nabla p_h, \nabla q_h) = (\vecb{f}, \nabla q_h) + \nu\int_{\partial\Omega} (\nabla\times\vecb{u}_h^\mathrm{ct}) \cdot (\vecb{n}\times \nabla q_h)\ \ds
    \qquad \text{for all } q_h \in P_{\ell}(\mathcal{T}) \cap H^1(\Omega) \cap Q,
\end{equation} 
which is a pure Neumann problem with inhomogeneous boundary conditions. 
The solution can be computed by utilizing a solver for s.p.d. systems.
Notice that the finite element pressure reconstruction computed with \eqref{eq:pres_reco}
substitutes the finite element pressure which is present in \eqref{eq:fullscheme_new} and which then disappears in the derivation 
of the method with divergence-free velocity basis. For the sake of simplifying the notation in the presentation 
of the numerical results, we use the notation $p_h$ also for the reconstructed discrete pressure.
    
\begin{theorem}\label{thm:cont_pres_error}
Under the assumption that a quasi-uniform family of triangulations $\{\mathcal{T}_h\}_{h>0}$ is considered, 
the error of $p_h$ computed by \eqref{eq:pres_reco} can be bounded by
\begin{equation}\label{eq:est_press}
  \| \nabla(p - p_h) \|_{L^2(\Omega)}
  \leq \inf_{q_h \in P_{\ell}(\mathcal{T}_h) \cap H^1(\Omega) \cap Q} \| \nabla(p - q_h) \|_{L^2(\Omega)}   
  + \nu C h^{k-1} \| \vecb{u} \|_{\vecb{H}^{k+1}(\Omega)},
\end{equation}
where $C>0$ is independent of $h$ and $\nu$.
\end{theorem}

\begin{proof}
Let $\Pi_h p \in P_{\ell}(\mathcal{T}_h) \cap H^1(\Omega) \cap Q$ be the Riesz projection
(best approximation with respect to $H^1(\Omega)$) of $p$. Using this projection, the error equation 
obtained by subtracting \eqref{eq:pres_reco} from \eqref{eq:pres} with test function 
$q_h = \Pi_h p - p_h$ gives 
    \begin{equation}\label{eq:pres_error_identity}
    \| \nabla(\Pi_h p - p_h) \|_{L^2(\Omega)}^2
     = \(\nabla (p - p_h), \nabla (\Pi_h p - p_h)\)
     = \nu\int_{\partial\Omega} \(\nabla\times\( \vecb{u} - \vecb{u}_h^\mathrm{ct}\)\) \cdot \(\vecb{n}\times \nabla \(\Pi p - p_h\)\)\ \ds.
    \end{equation}
    Consider the following auxiliary problem: Find $e_{p,\mathrm{S}}\in Q$ such that 
    \begin{align*} 
        (\nabla e_{p,\mathrm{S}}, \nabla q)=\int_{\partial\Omega} \(\nabla\times\( \vecb{u} - \vecb{u}_h^\mathrm{ct}\)\) \cdot \(\vecb{n}\times \nabla q\)\ \ds
    \end{align*}
    for all $q\in Q$. 
    Then \eqref{eq:pres_error_identity} implies 
    \begin{align*} 
        \| \nabla(\Pi_h p - p_h) \|_{L^2(\Omega)}^2 = \nu (\nabla e_{p,\mathrm{S}}, \nabla (\Pi_h p - p_h))
        \leq \nu \|\nabla e_{p,\mathrm{S}}\|_{L^2(\Omega)}\|\nabla (\Pi_h p - p_h)\|_{L^2(\Omega)},
    \end{align*} 
    and further 
    \begin{align*} 
        \| \nabla(\Pi_h p - p_h) \|_{L^2(\Omega)} \leq \nu \|\nabla e_{p,\mathrm{S}}\|_{L^2(\Omega)}.
    \end{align*}
    Indeed, $e_{p,\mathrm{S}}$ is the Stokes pressure corresponding to $\vecb{u} - \vecb{u}_h^\mathrm{ct}$ 
    and it satisfies the estimate, see \cite{LLP2007,LS2025},
    \begin{eqnarray}\label{eq:stokes_pres_estimate}
        \|\nabla e_{p,\mathrm{S}}\|_{L^2(\Omega)}^2 & \leq & C_\varepsilon \|\nabla(\vecb{u} - \vecb{u}_h^\mathrm{ct})\|_{L^2(\Omega)}^2
        + \left(\frac{1}{2}+\varepsilon\right) \|\Delta_\mathrm{pw} (\vecb{u} - \vecb{u}_h^\mathrm{ct})\|_{L^2(\Omega)}^2 \nonumber\\
         && + C \sum_{F\in\mathcal{F}^0} h_F^{-1} \int_F |\jump{\nabla(\vecb{u} - \vecb{u}_h^\mathrm{ct})}|^2\ \ds, 
    \end{eqnarray}
    where $\varepsilon$ can be arbitrarily small, $C_\varepsilon>0$ depends on $\varepsilon^{-1}$, and
    $\jump{\cdot}$ denotes the jump operator on facets whose definition can be found in \cite{LS2025}.
    Under the quasi-uniformity assumption, 
    it follows from \cite[Theorem 1.6.6]{BrennerScott:2008}
    and Young's inequality that 
    \begin{align}\label{eq:trace_estimate}
    \sum_{F\in\mathcal{F}^0} h_F^{-1} \int_F |\jump{\nabla(\vecb{u} - \vecb{u}_h^\mathrm{ct})}|^2\ \ds \leq 
    C \left( h^{-2}  \|\nabla(\vecb{u} - \vecb{u}_h^\mathrm{ct})\|_{L^2(\Omega)}^2 + 
         \sum_{T\in\mathcal{T}_h} \|\nabla(\vecb{u} - \vecb{u}_h^\mathrm{ct})\|_{H^1(T)}^2\right).
    \end{align}
    Let $\mathcal{I}\vecb{u}$ be an interpolation of $\vecb{u}$ satisfying \cite[Eq.~(4.4.5)]{BrennerScott:2008}. 
    By successively applying the triangle inequality, the inverse inequality, 
    and again the triangle inequality, one obtains
    \begin{align*} 
        \|\nabla(\vecb{u} - \vecb{u}_h^\mathrm{ct})\|_{H^1(T)}
        &\leq \|\nabla(\vecb{u} - \mathcal{I}\vecb{u})\|_{H^1(T)}+
        \|\nabla(\mathcal{I}\vecb{u} - \vecb{u}_h^\mathrm{ct})\|_{H^1(T)}\\
        &\leq \|\nabla(\vecb{u} - \mathcal{I}\vecb{u})\|_{H^1(T)}+
        C h^{-1} \|\nabla(\mathcal{I}\vecb{u} - \vecb{u}_h^\mathrm{ct})\|_{L^2(T)}\\
        &\leq \|\nabla(\vecb{u} - \mathcal{I}\vecb{u})\|_{H^1(T)}
        +C h^{-1}\|\nabla(\vecb{u} - \mathcal{I}\vecb{u})\|_{L^2(T)}
        +C h^{-1} \|\nabla(\vecb{u} - \vecb{u}_h^\mathrm{ct})\|_{L^2(T)},
    \end{align*}
    which, together with Theorem~\ref{thm:velo_error}, implies 
    \begin{align*} 
        \sum_{T\in\mathcal{T}_h} \|\nabla(\vecb{u} - \vecb{u}_h^\mathrm{ct})\|_{H^1(T)}^2\leq 
        C h^{2k-2} \| \vecb{u} \|_{\vecb{H}^{k+1}(\Omega)}^2, 
    \end{align*} 
    and further, by utilizing \eqref{eq:stokes_pres_estimate} and \eqref{eq:trace_estimate}, one obtains 
    \begin{align*} 
        \|\nabla e_{p,\mathrm{S}}\|_{L^2(\Omega)}^2 \leq C h^{2k-2} \| \vecb{u} \|_{\vecb{H}^{k+1}(\Omega)}^2.
    \end{align*}
    The Pythagorean identity
    \[
    \| \nabla(p - p_h) \|_{L^2(\Omega)}^2 = \| \nabla(p - \Pi_h p) \|_{L^2(\Omega)}^2 + \| \nabla(\Pi_h p - p_h) \|_{L^2(\Omega)}^2.
    \]
    concludes the proof.
\end{proof}

Theorem~\ref{thm:cont_pres_error} shows that the pressure error is also optimal if one chooses $\ell\leq k-1$.

Notice that the usual norm of the error analysis for the pressure in mixed finite element methods is the 
$L^2(\Omega)$ norm and not the $L^2(\Omega)$ norm of the gradient as in Theorem~\ref{thm:cont_pres_error}. The proof 
in mixed methods uses the discrete inf-sup condition for the bilinear form that couples the
velocity and pressure finite element spaces. However, there is no such bilinear form in the decoupled method
as a consequence of using divergence-free velocity basis functions. {We think that it is possible
to derive an estimate for the $L^2(\Omega)$ error of the reconstructed pressure utilizing the Aubin--Nitsche trick and 
and expect convergence of one order higher then in \eqref{eq:est_press} for sufficiently smooth data.
However, this is outside the scope of this paper.}

%% file: numerical_examples.tex
\section{Numerical studies}
\label{sec:experiments}

The objectives of this section are twofold. First, the expectations concerning the convergence
of the decoupled methods should be supported, e.g., compare Theorem~\ref{thm:velo_error}, where the 
velocity approximation was computed by solving \eqref{eq:pressure_free_scheme} and the pressure 
approximation as described in Section~\ref{sec:pressure}.
And second, the efficiency 
of these methods shall be compared with those of the full mixed method from \cite{JLMR2022}, see \eqref{eq:fullscheme}, 
and the reduced mixed method from \cite{JLMR2022}, see \eqref{eq:galerkin_reducedform}.
Using a decoupled approach seems to be on the first glance more efficient than having to solve 
a saddle point problem, which was one of our main motivation for constructing the new methods. However, 
also other aspects possess an impact on the efficiency, like the sparsity pattern of the arising 
matrices and the number of non-zero entries (nnz). The considered methods differ in these aspects. 
In addition, the decoupled methods and the reduced mixed method need additional overhead to 
compute the final solution. For the decoupled method, this overhead consists of constructing the 
divergence-free basis, of incorporating non-homogeneous Dirichlet boundary conditions, and of
computing the pressure (if needed). 

All methods have been implemented in the code 
\cite{ExtendableFEMjl}. For solving the linear systems of equations, the 
sparse direct solver Pardiso \cite{SGFSpardiso, PardisoJL} 
was used, where 
we set {\tt iparm 8} to the value 10, {\tt iparm 10} to the value 6 for the two-dimensional simulations and 
the value 8 for the three-dimensional ones,
{\tt iparm 11} to the value 1, and {\tt iparm 13} to 0 for the polynomial degree $k=1$ and for the skew-symmetric
case in three dimensions, and otherwise to the value 1.
For the special case $k=d=3$, we changed
{\tt iparm 8} to the value 4 and {\tt iparm 10} to the value 11.

The stabilization parameter $\alpha$ in \eqref{eqn:RTk_stabilization} is set to $\alpha = 100$
in $d=2$ dimensions and to $\alpha = 200$ in $d=3$ dimensions. For the special case $k=d=3$ the stabilization
parameter is set $\alpha = 300$.
The stabilization parameter $\alpha_0 = 1$ is used in \eqref{eqn:RT0_stabilization}
for the lowest order schemes $k=1$
according to the experience in \cite{JLM2024} and for the skew-symmetric case ($\delta = 1$) in $d=3$
dimensions.
In the higher order ($k > 1$) and symmetric case ($\delta = -1$), $\alpha_0 = \alpha$ is used.

Notice that for polynomial degree $k=1$ the terms with $\delta$ vanish for all methods. We present 
the corresponding results together with the symmetric variants of the methods. 

\begin{figure}[t!]
\centerline{    \includegraphics[width=0.3\textwidth]{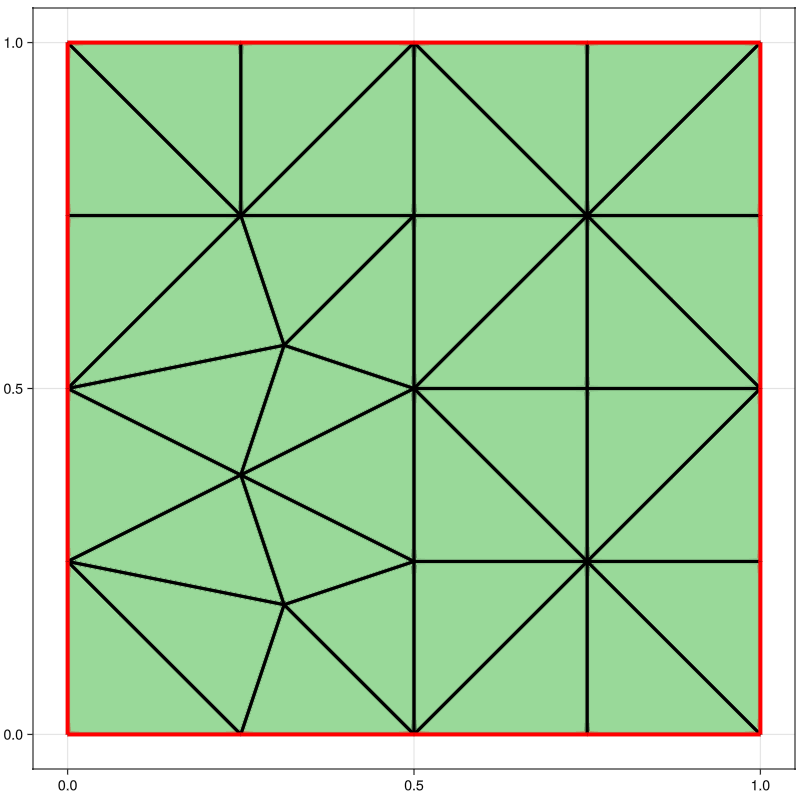}\hspace*{1em}
    \includegraphics[width=0.3\textwidth]{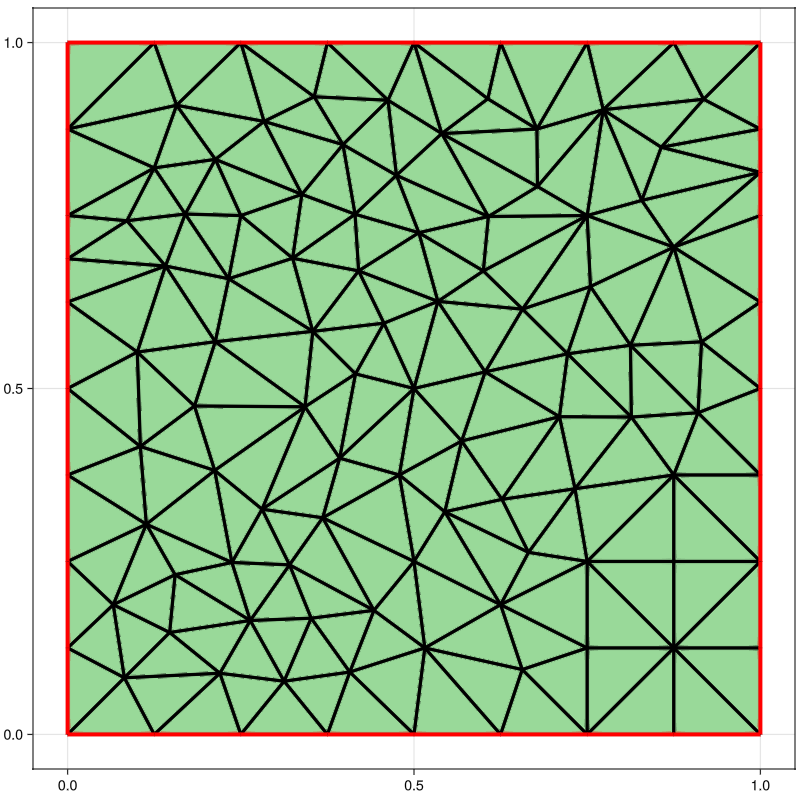}\hspace*{1em}
    \includegraphics[width=0.3\textwidth]{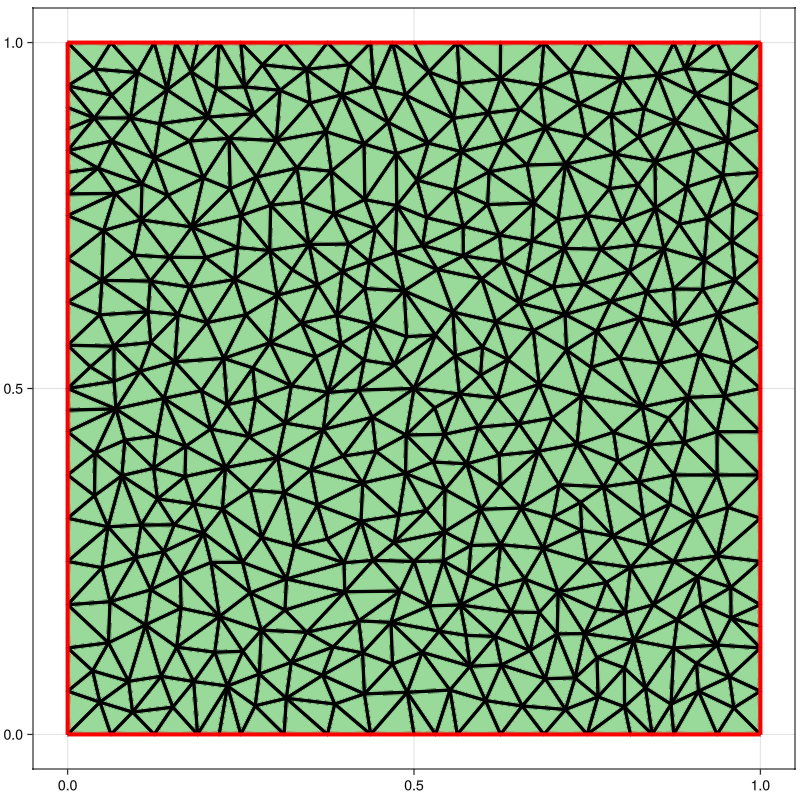}}
    \caption{Example~\ref{ex:2d}. The first three computational grids.} \label{fig:2d_grids}
\end{figure}

\subsection{A two-dimensional  example}\label{ex:2d}

This example is given by $\Omega=(0,1)^2$ and the right-hand side and Dirichlet boundary conditions are 
set such that 
\[
    \vecb{u}(x,y)  := \begin{pmatrix} - \partial_y \xi\\  \partial_x \xi\end{pmatrix}
    \quad \text{where} \quad \xi = -\sin(2\pi x) \cos(2 \pi y), \quad
	p  := \frac{\cos(4 \pi x) - \cos(4 \pi y)}4
\]
satisfy the Stokes problem for $\nu=1$ or $\nu=10^{-6}$, respectively.

Simulations were performed on a sequence of unstructured grids with decreasing mesh width, where
the first three grids are depicted in Figure~\ref{fig:2d_grids}, and for different polynomial degrees, up to $k=4$.
The non-homogeneous Dirichlet
boundary conditions were imposed for the decoupled methods as described in Section~\ref{sec:nonhomo_BC_2d}.
Information on the number of degrees of freedom and the non-zero entries for the different methods are 
provided in Tables~\ref{tab:2d_dof_nnz_sym} and \ref{tab:2d_dof_nnz}. 
In case of the decoupled schemes, the numbers only include the degrees of freedom or nonzero entries
for the velocity problem, the decoupled (and smaller) pressure problem is not included.
As it was a motivation of the construction
of the decoupled methods with divergence-free basis, the number of degrees of freedom is smallest for this method. The corresponding 
numbers for the reduced methods are usually a bit larger and the full mixed methods 
have often around twice as many degrees of freedom than the corresponding decoupled methods. 
The situation is somewhat different with respect to the number of non-zero matrix entries, since these 
numbers are often similar for the decoupled methods and the corresponding reduced mixed methods.
The most non-zero entries are usually required by the full mixed methods. 
 
 \begin{table}
    \caption{\label{tab:2d_dof_nnz_sym}Example~\ref{ex:2d}. Number of  degrees of freedom (ndofs) and number of non-zero sparse matrix entries (nnz)
    of the full system for the decoupled method with divergence-free basis (dfb), the reduced mixed method (red), and the full mixed method (full)
    for different orders $k$ and $\delta = 1$ (skew-symmetric case).}
    \footnotesize
    \begin{center}
    \input{tables/2d/table_order=2_levels=5.txt}

    \input{tables/2d/table_order=3_levels=5.txt}

    \input{tables/2d/table_order=4_levels=4.txt}

    \end{center}
\end{table}

\begin{table}
    \caption{\label{tab:2d_dof_nnz}Example~\ref{ex:2d}. Number of degrees of freedom (ndofs) and number of non-zero sparse matrix entries (nnz)
    of the full system
    for the decoupled method with divergence-free basis (dfb), the reduced mixed method (red), and the full mixed method (full)
    for different orders $k$ and $\delta = -1$ (symmetric case).}
    \footnotesize
    \begin{center}
    \input{tables/2d/table_order=1_levels=7_symmetric.txt}

    \input{tables/2d/table_order=2_levels=5_symmetric.txt}

    \input{tables/2d/table_order=3_levels=5_symmetric.txt}

    \input{tables/2d/table_order=4_levels=4_symmetric.txt}

    \end{center}
\end{table}

\begin{figure}[t!]
\centerline{
    \includegraphics[width=0.3\textwidth]{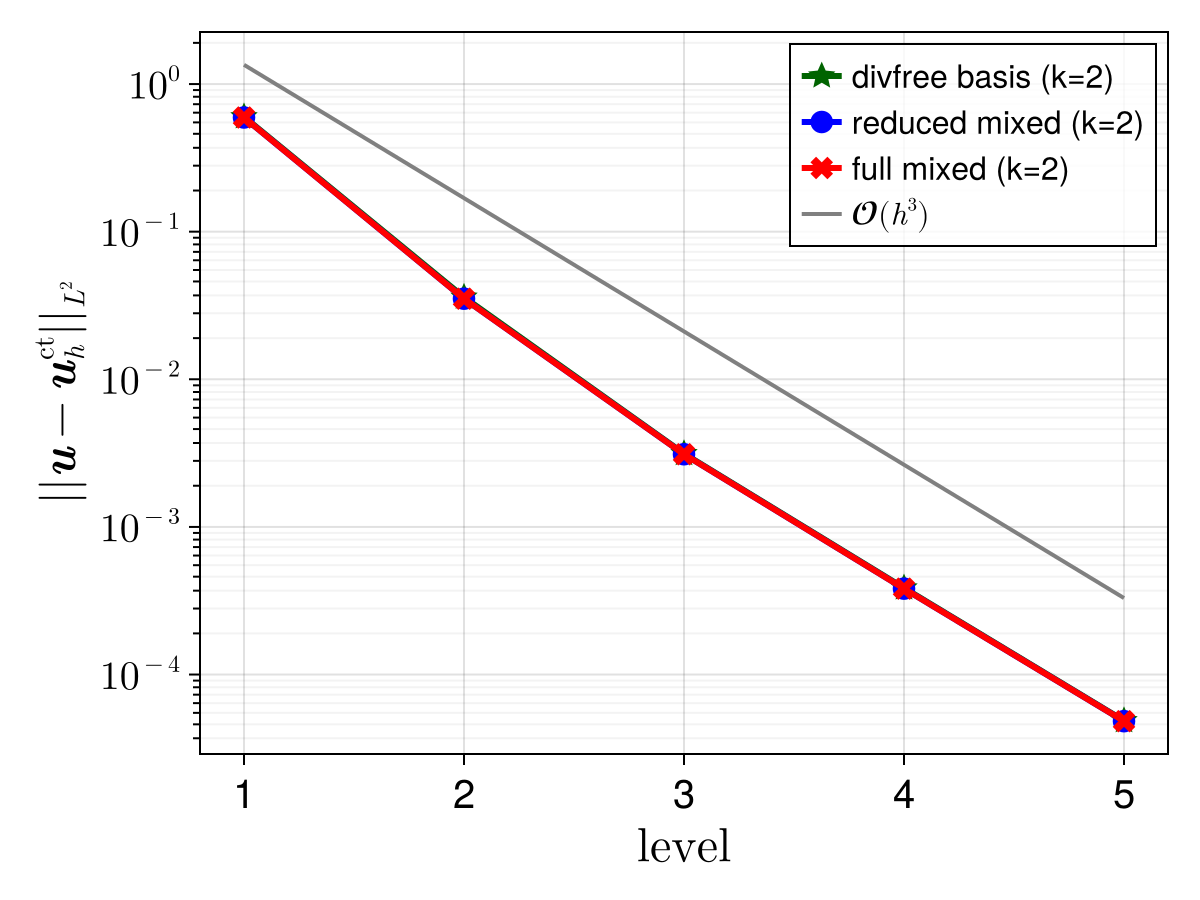}
    \includegraphics[width=0.3\textwidth]{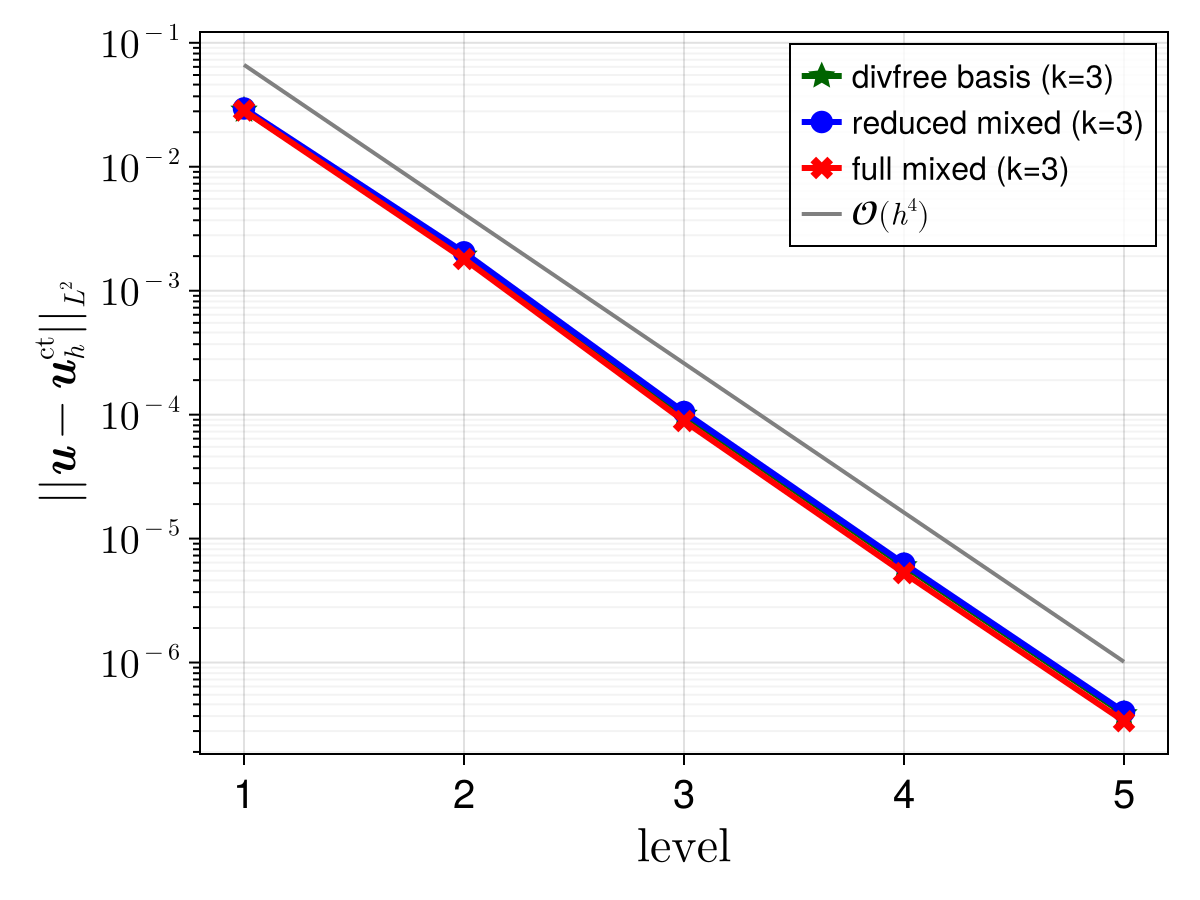}
    \includegraphics[width=0.3\textwidth]{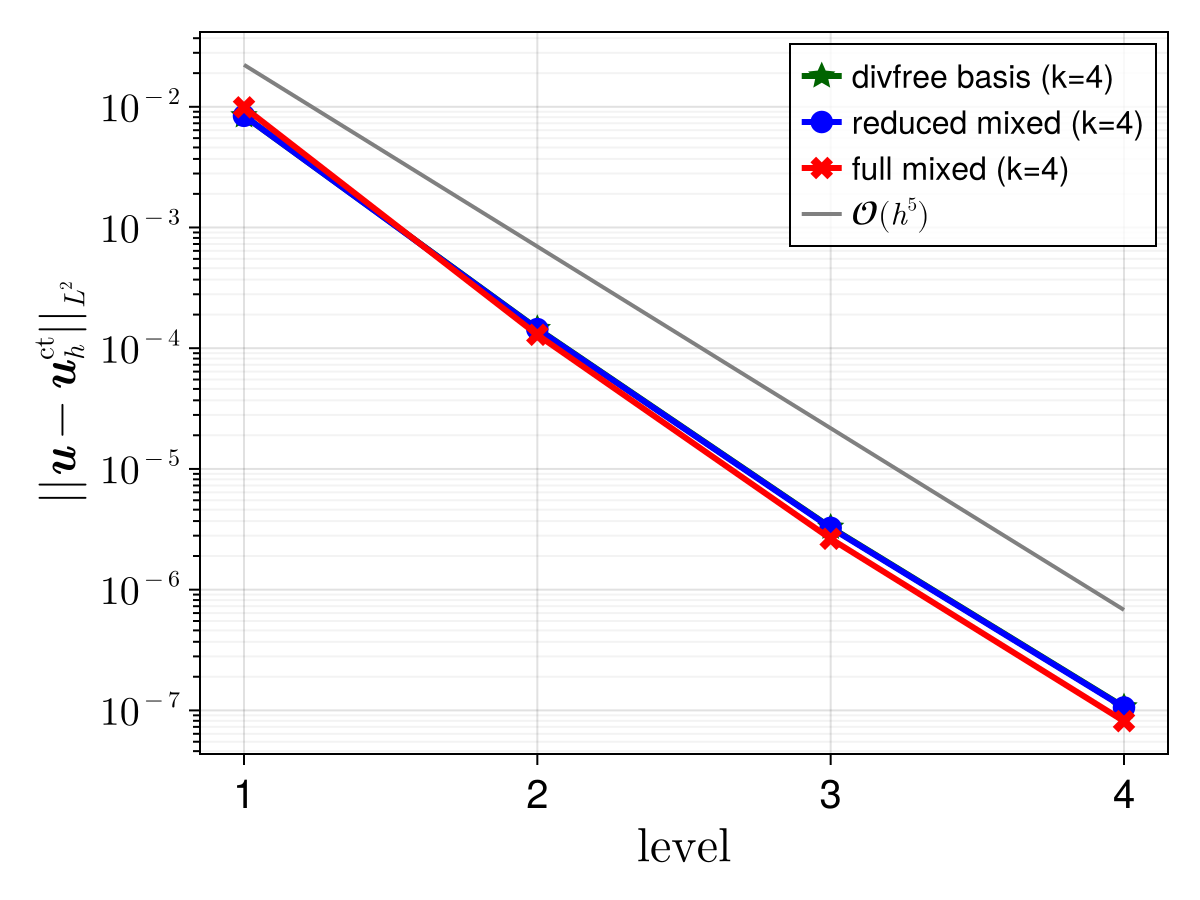}}
\centerline{
    \includegraphics[width=0.3\textwidth]{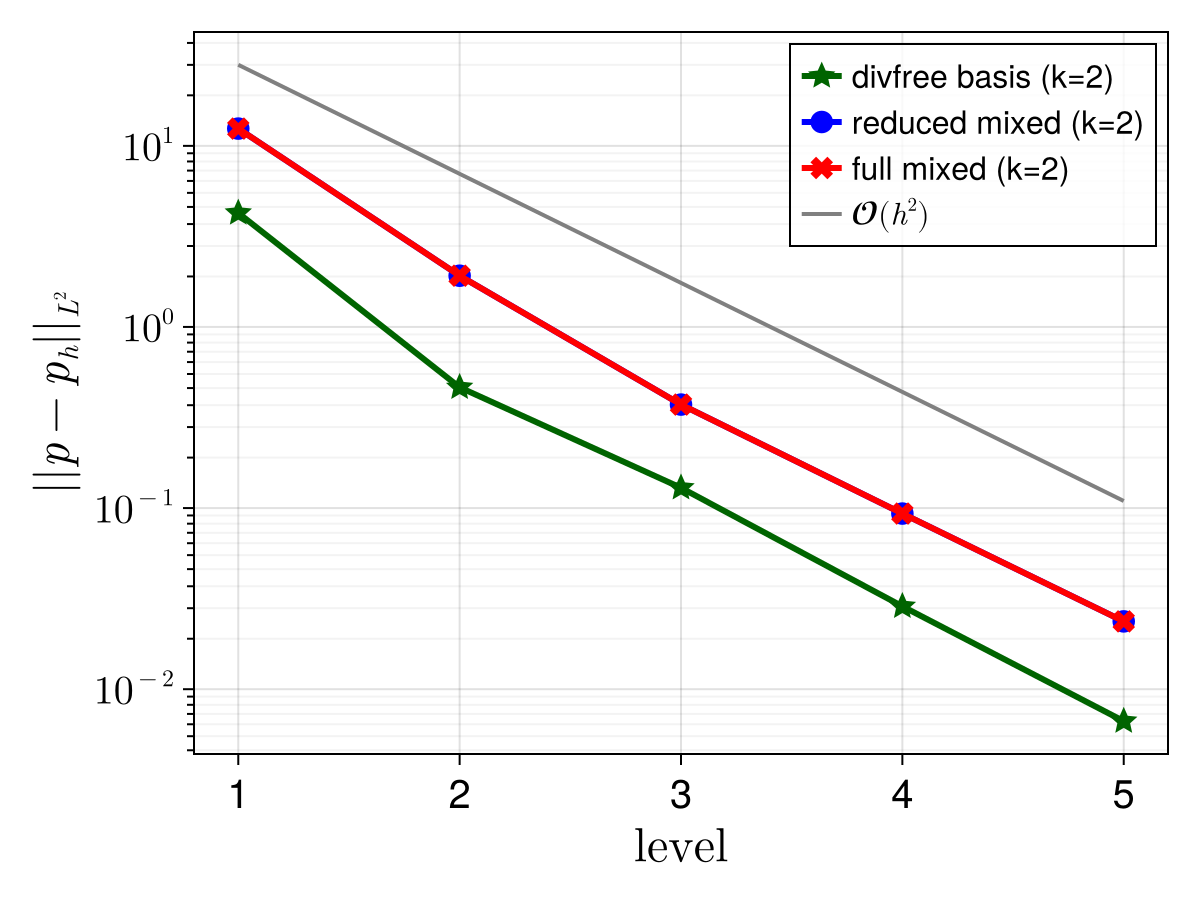}
    \includegraphics[width=0.3\textwidth]{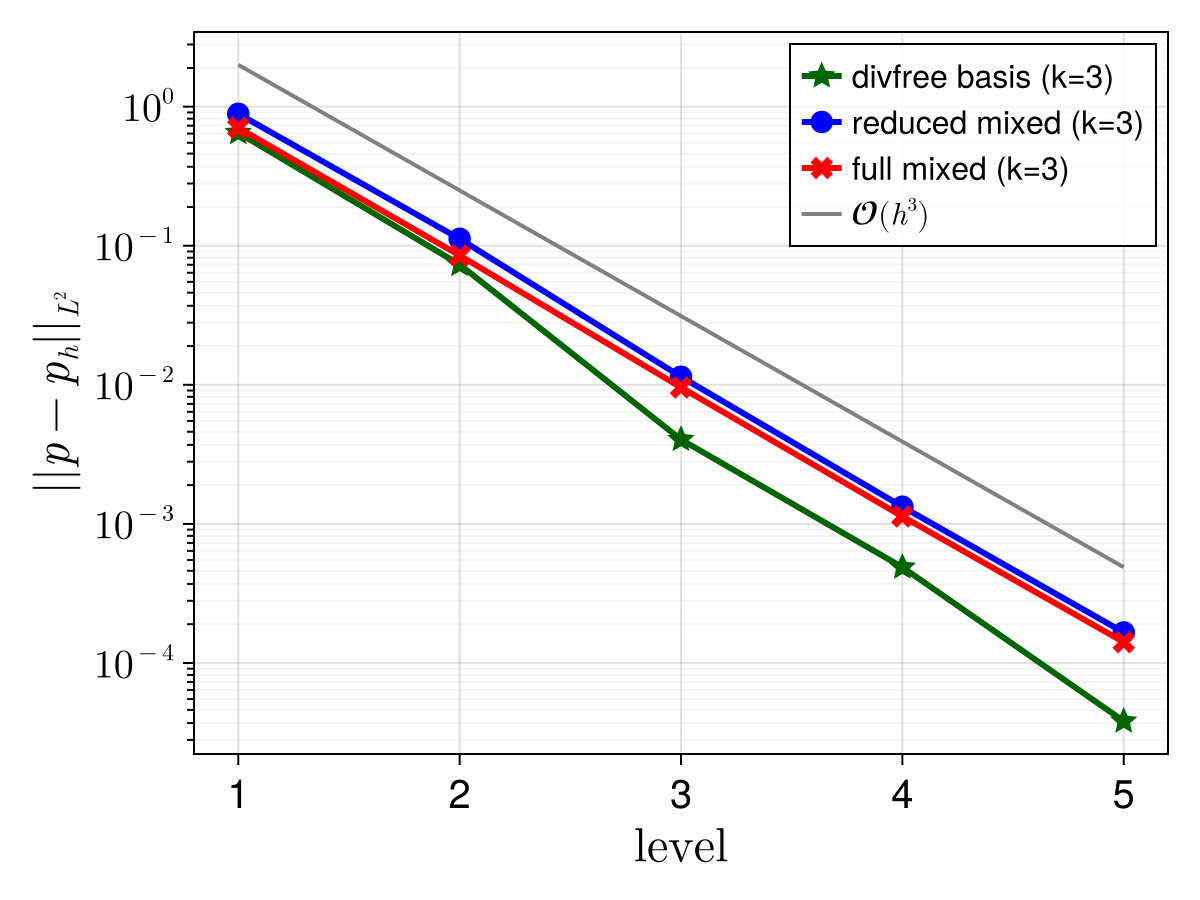}
    \includegraphics[width=0.3\textwidth]{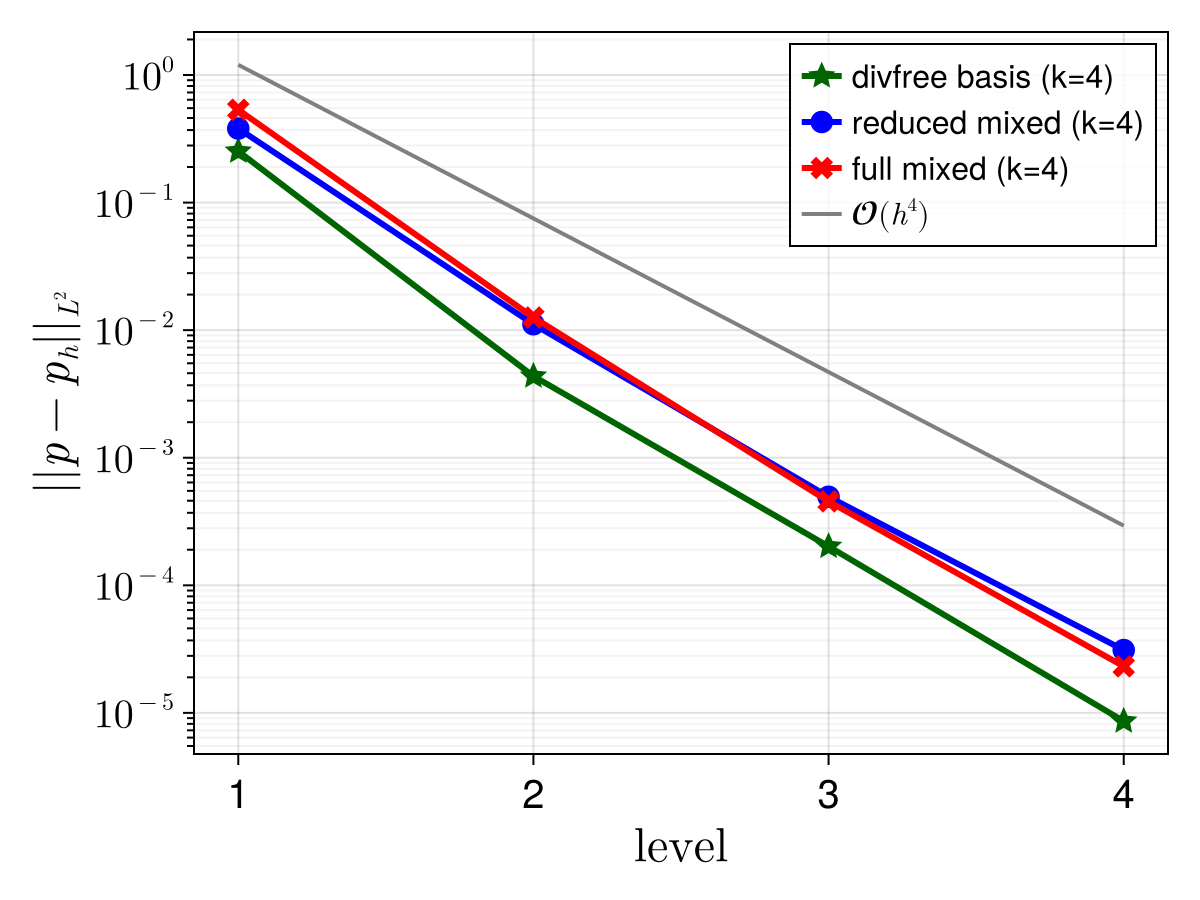}}
    \caption{Example~\ref{ex:2d} with $\nu = 1$. $L^2(\Omega)$ velocity error (top) and pressure error (bottom) for the polynomial degree $k=2$ (left), $k=3$ (center), and $k=4$ (right), $\delta = 1$ (skew-symmetric case).} \label{fig:2d_conv_nonsym}
\end{figure}

\begin{figure}[t!]
\centerline{
    \includegraphics[width=0.3\textwidth]{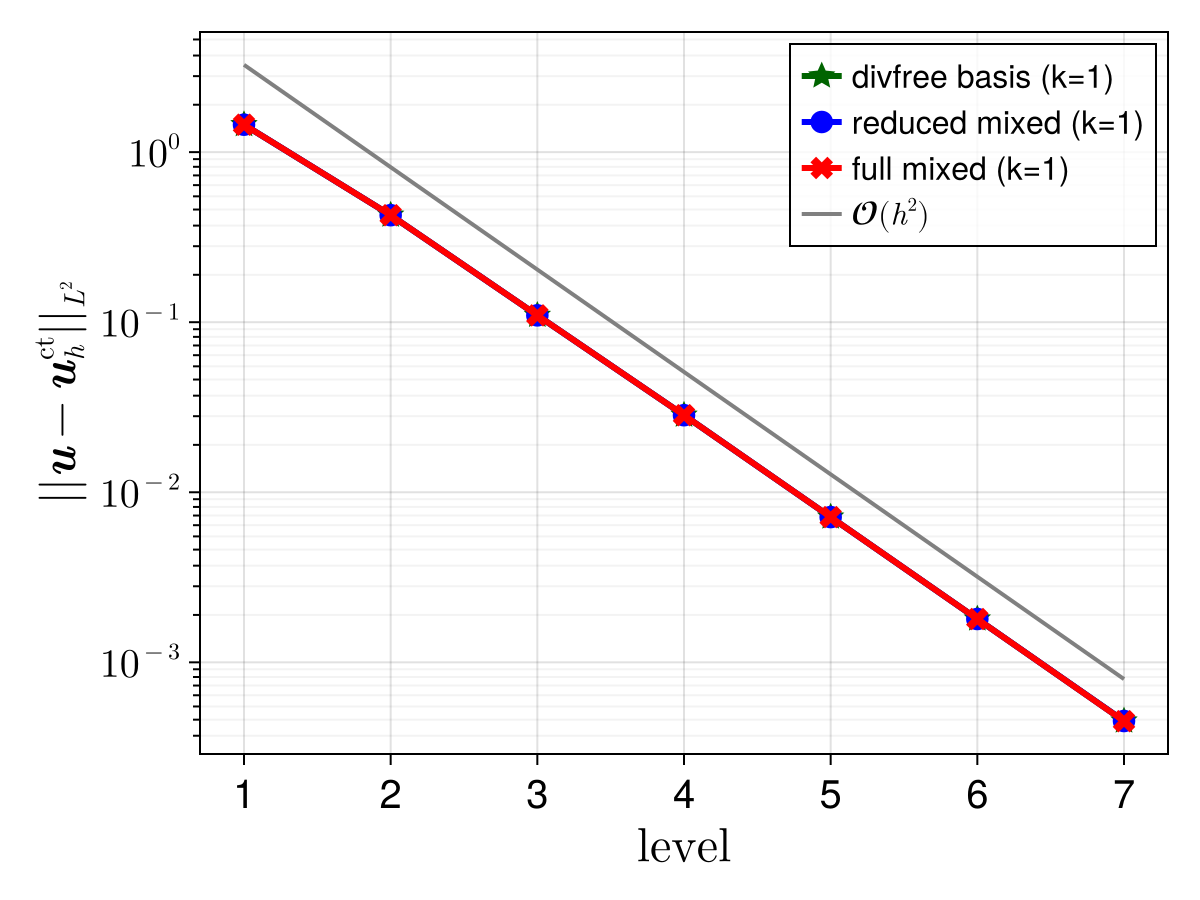}
    \includegraphics[width=0.3\textwidth]{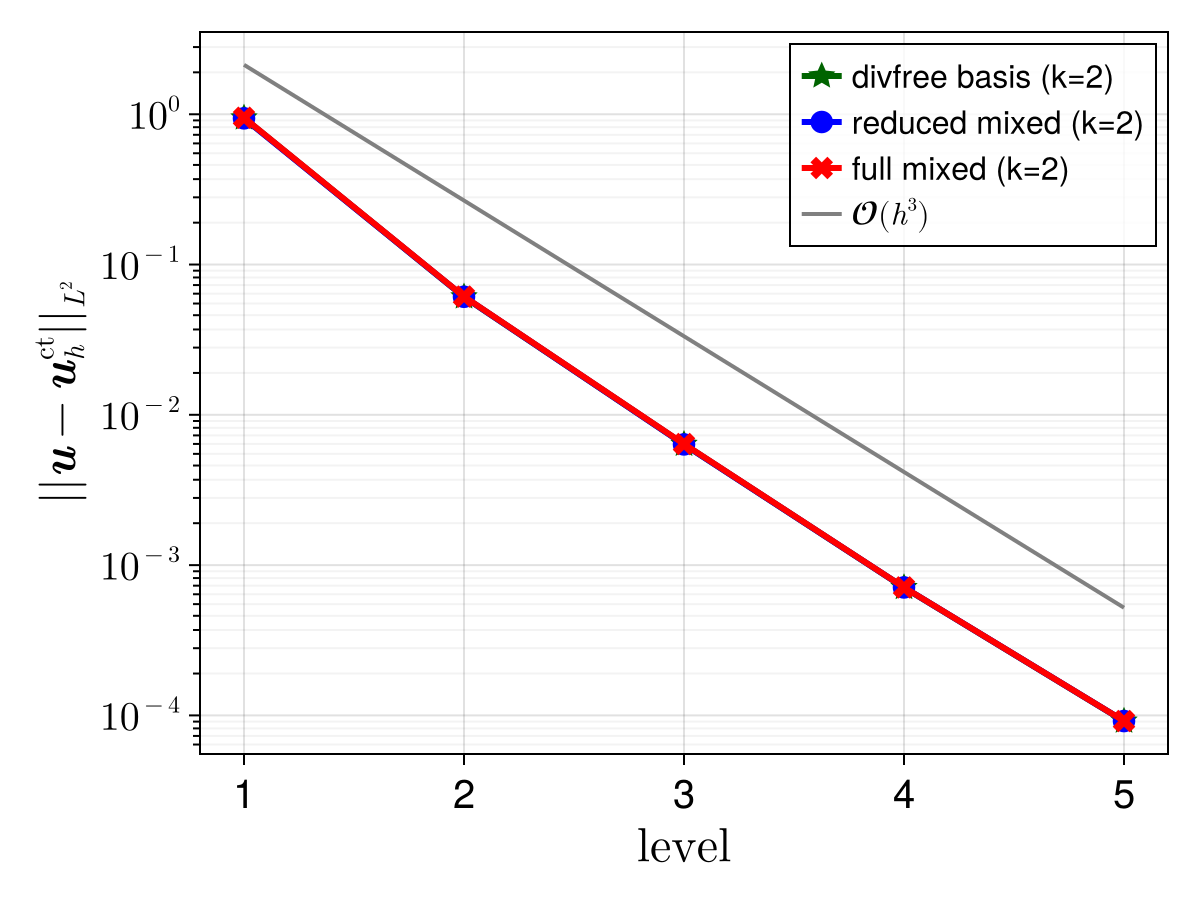}
    \includegraphics[width=0.3\textwidth]{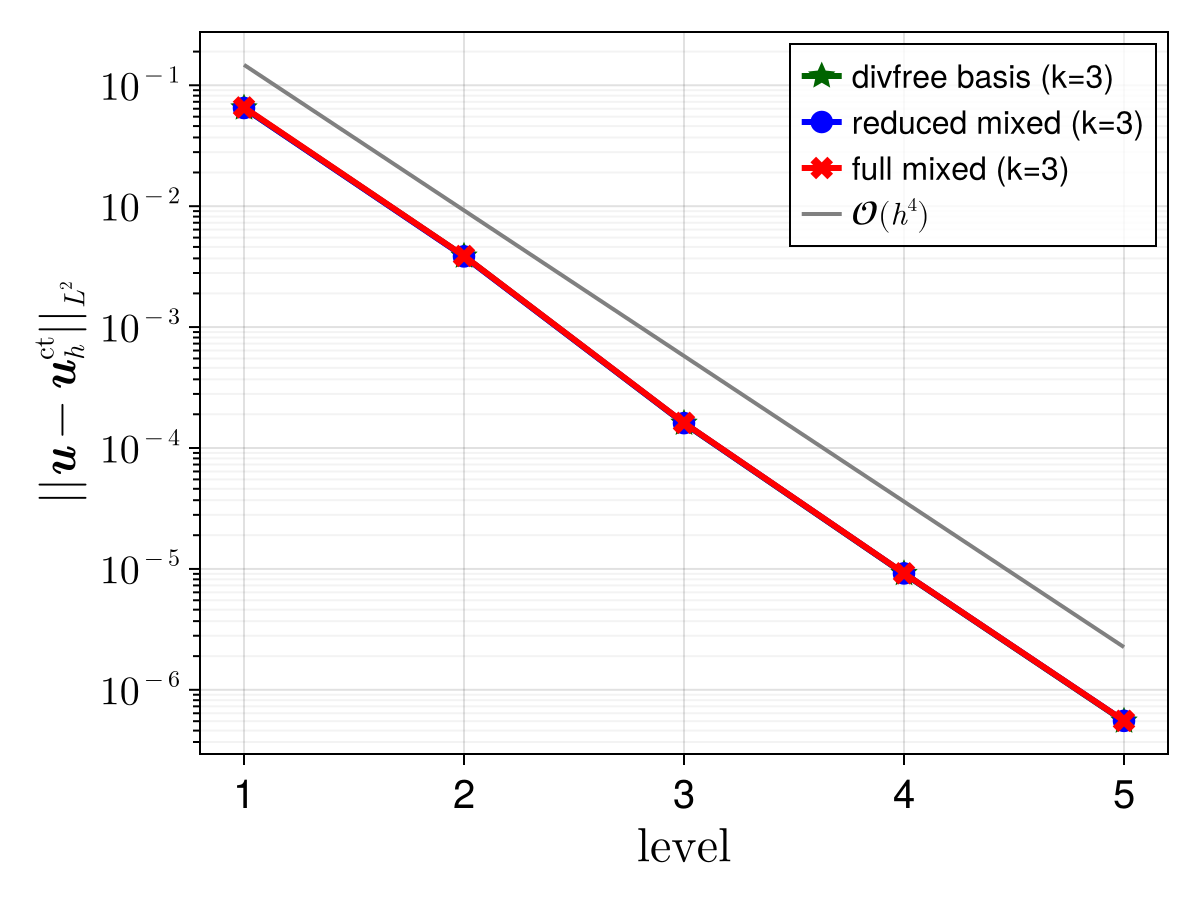}}
\centerline{
    \includegraphics[width=0.3\textwidth]{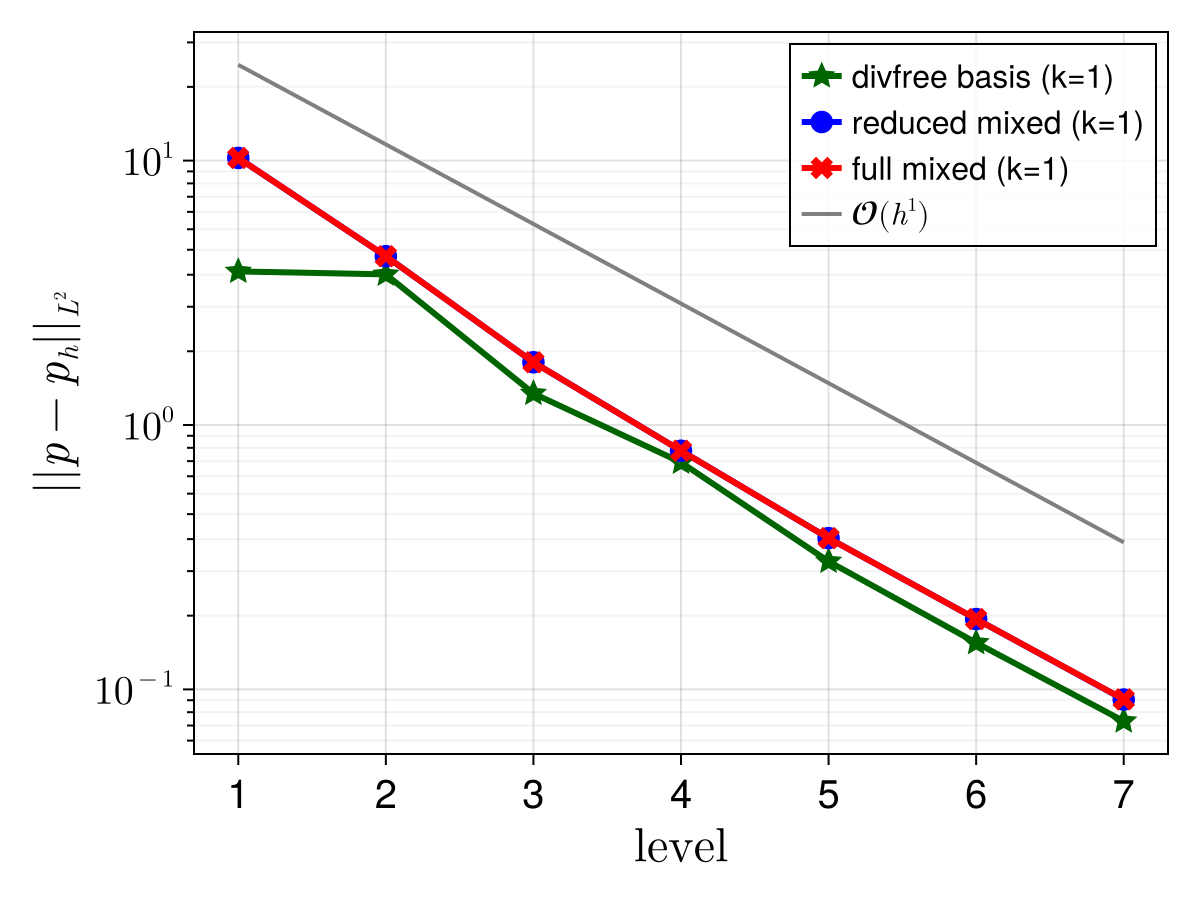}
    \includegraphics[width=0.3\textwidth]{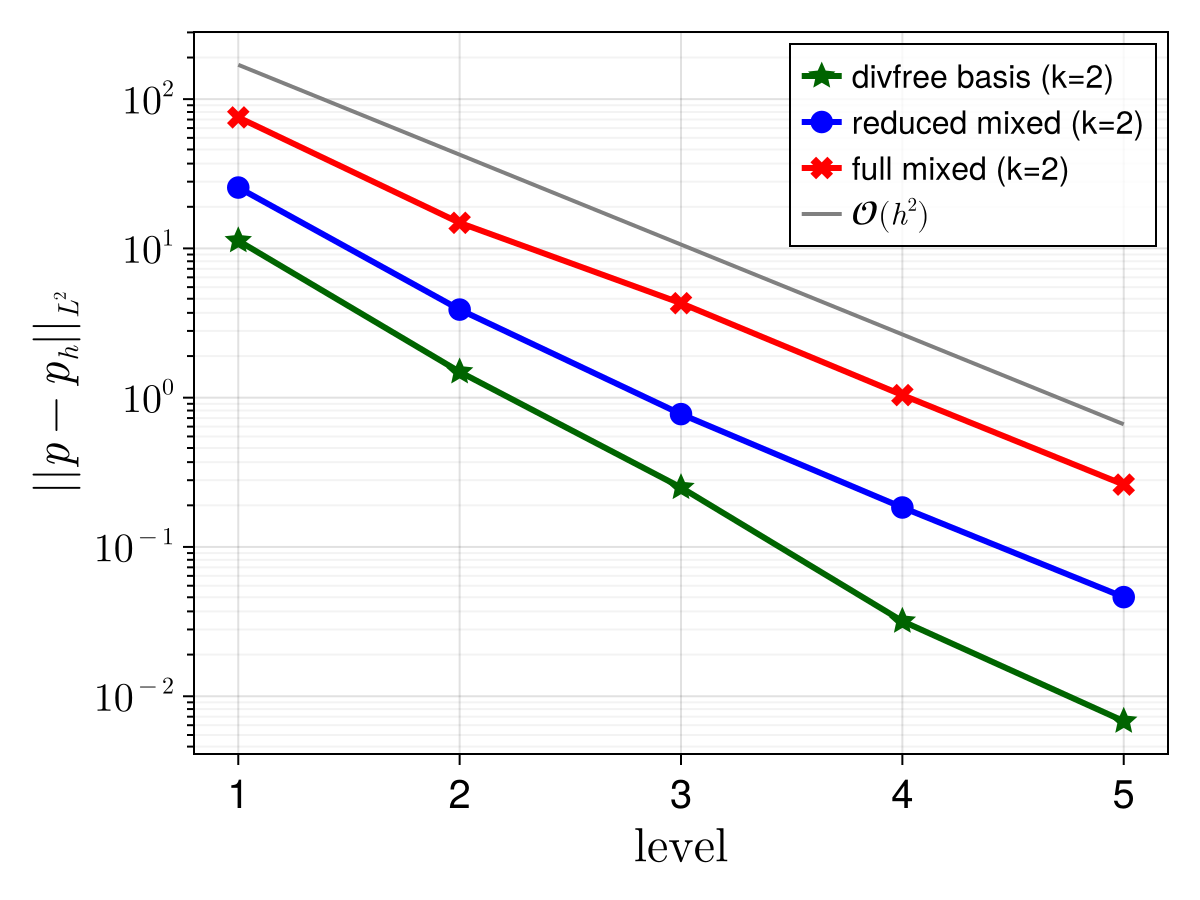}
    \includegraphics[width=0.3\textwidth]{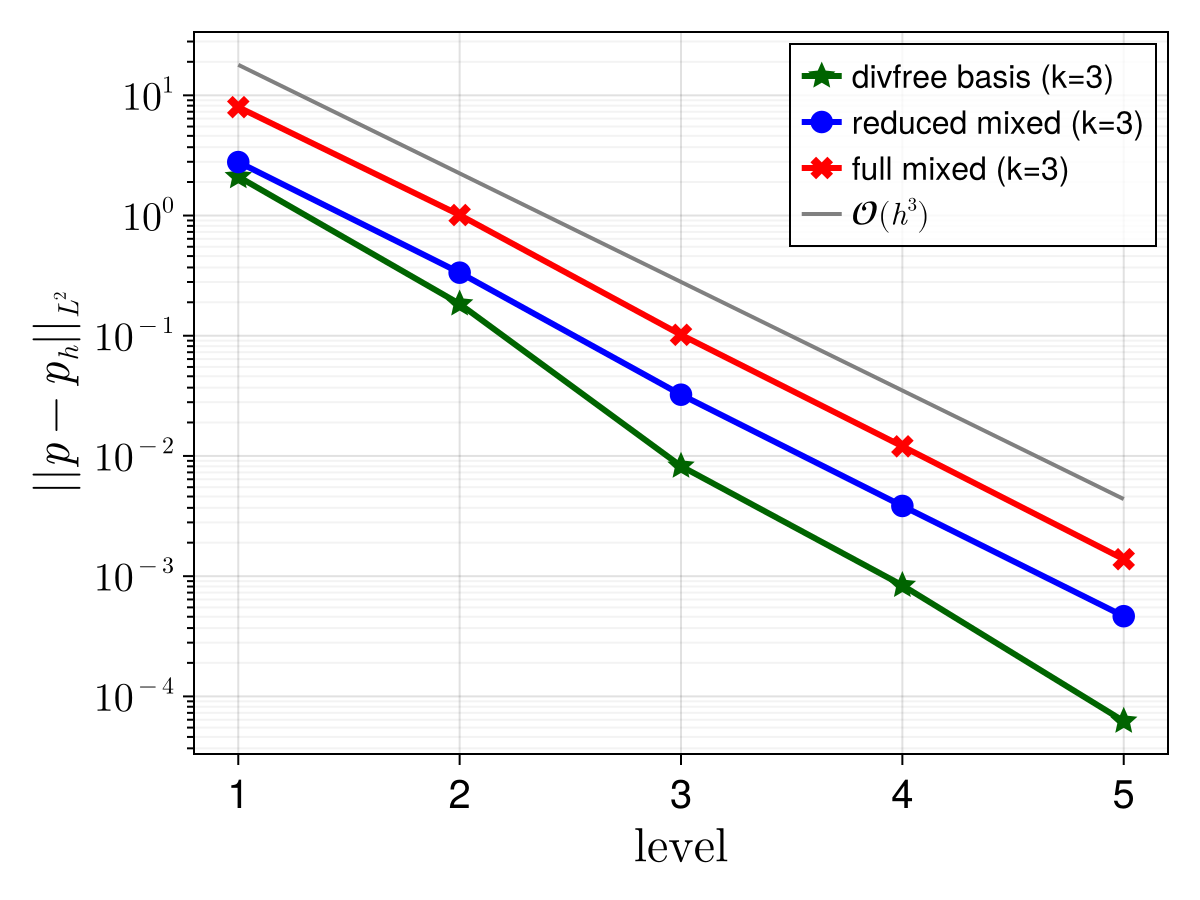}}
    \caption{Example~\ref{ex:2d} with $\nu = 1$. $L^2(\Omega)$ velocity error (top) and pressure error (bottom) for the polynomial degree $k=1$ (left), $k=2$ (center), and $k=3$ (right), $\delta=-1$ (symmetric case).} \label{fig:2d_conv_sym}
\end{figure}

\begin{figure}[t!]
\centerline{ 
    \includegraphics[width=0.3\textwidth]{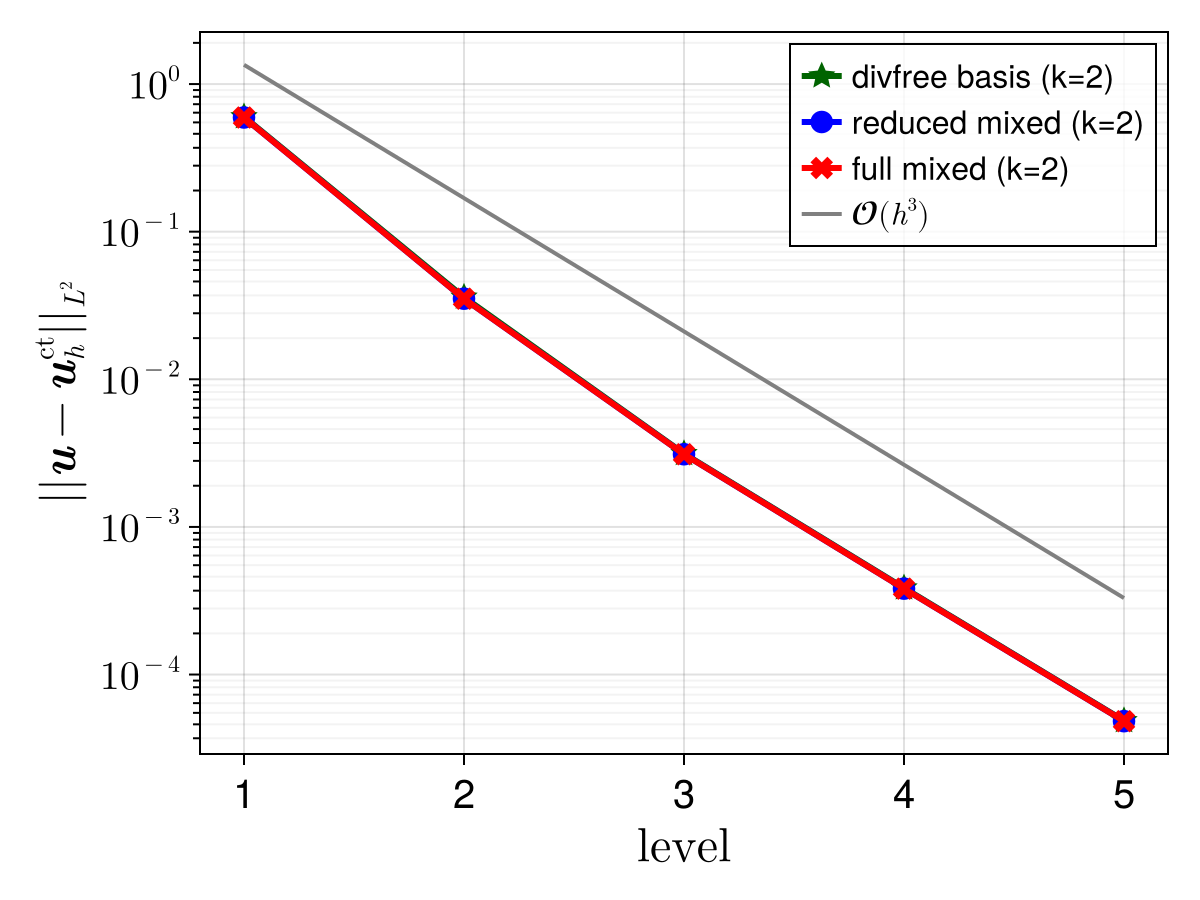}
    \includegraphics[width=0.3\textwidth]{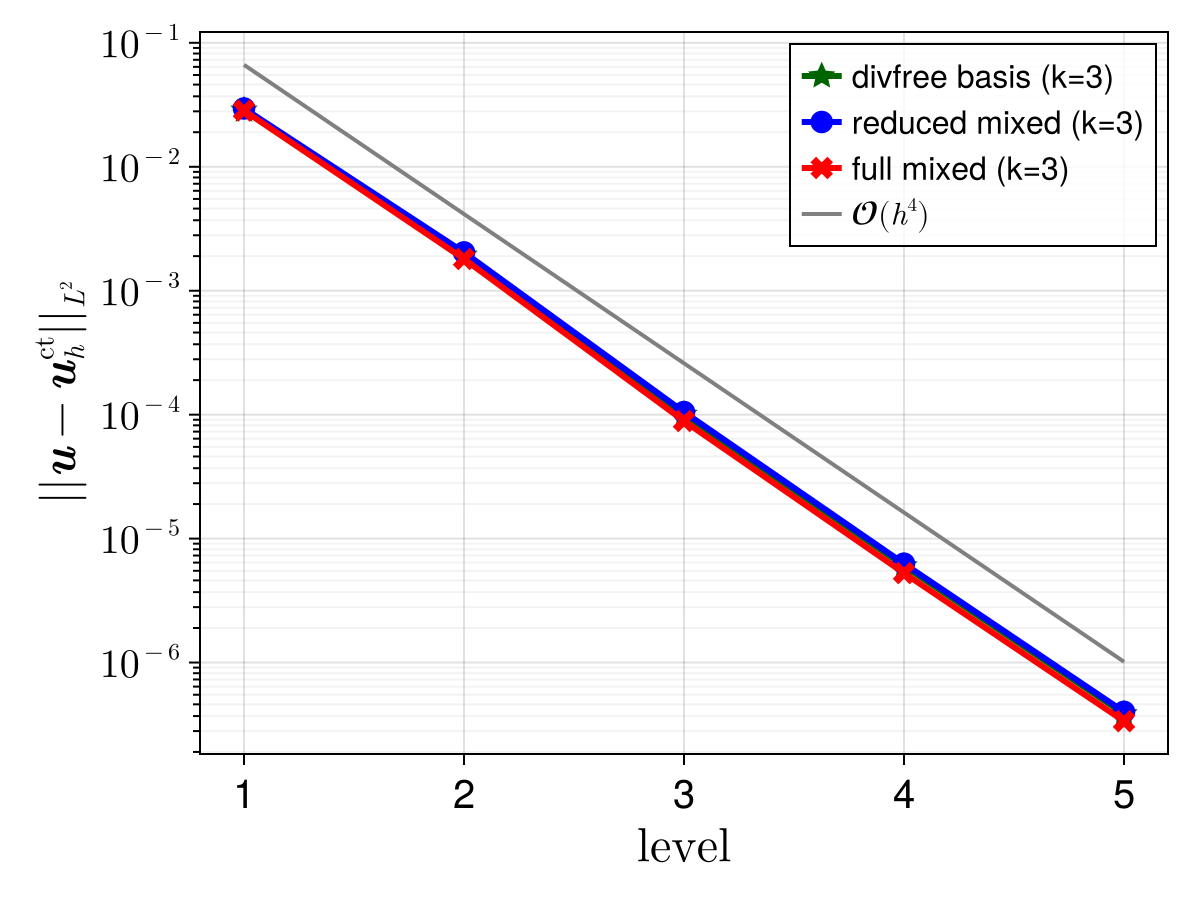}
    \includegraphics[width=0.3\textwidth]{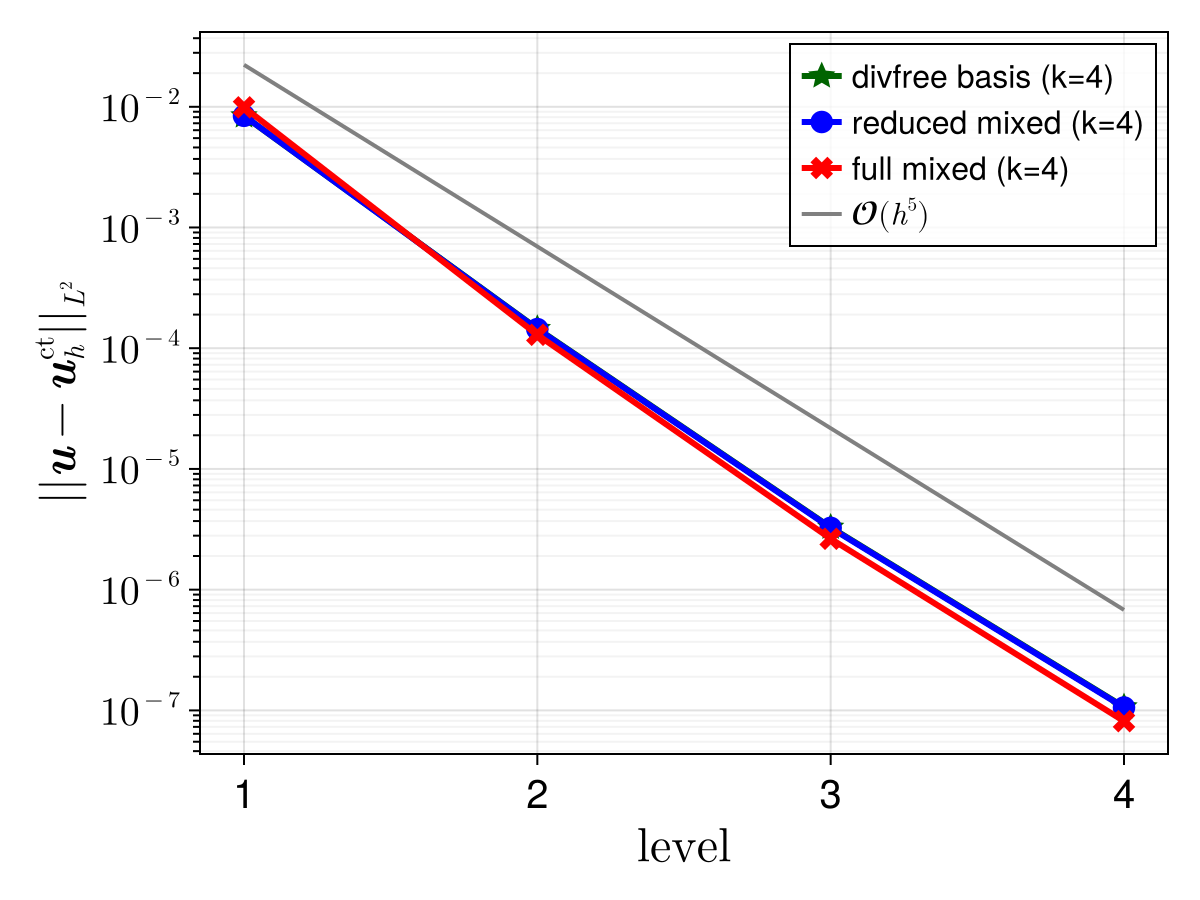}}
\centerline{
    \includegraphics[width=0.3\textwidth]{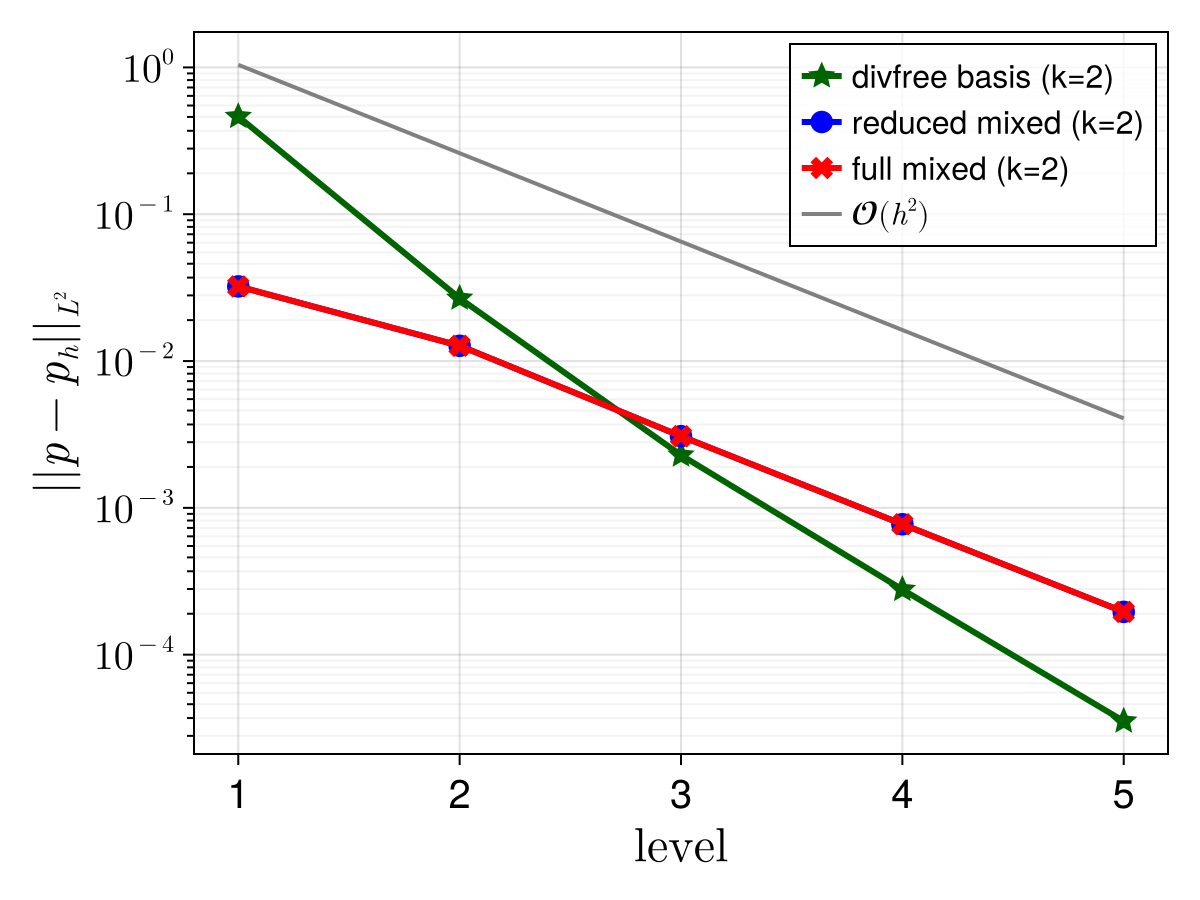}
    \includegraphics[width=0.3\textwidth]{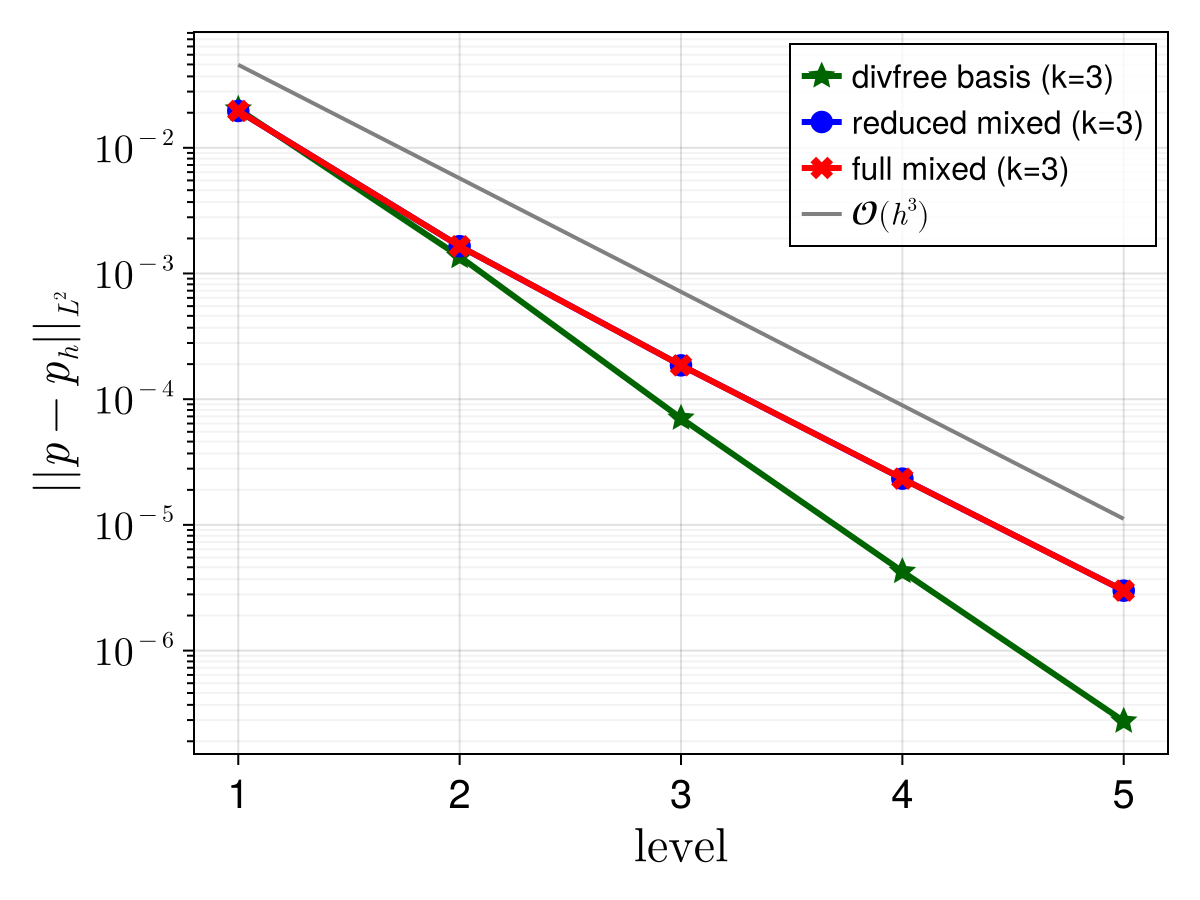}
    \includegraphics[width=0.3\textwidth]{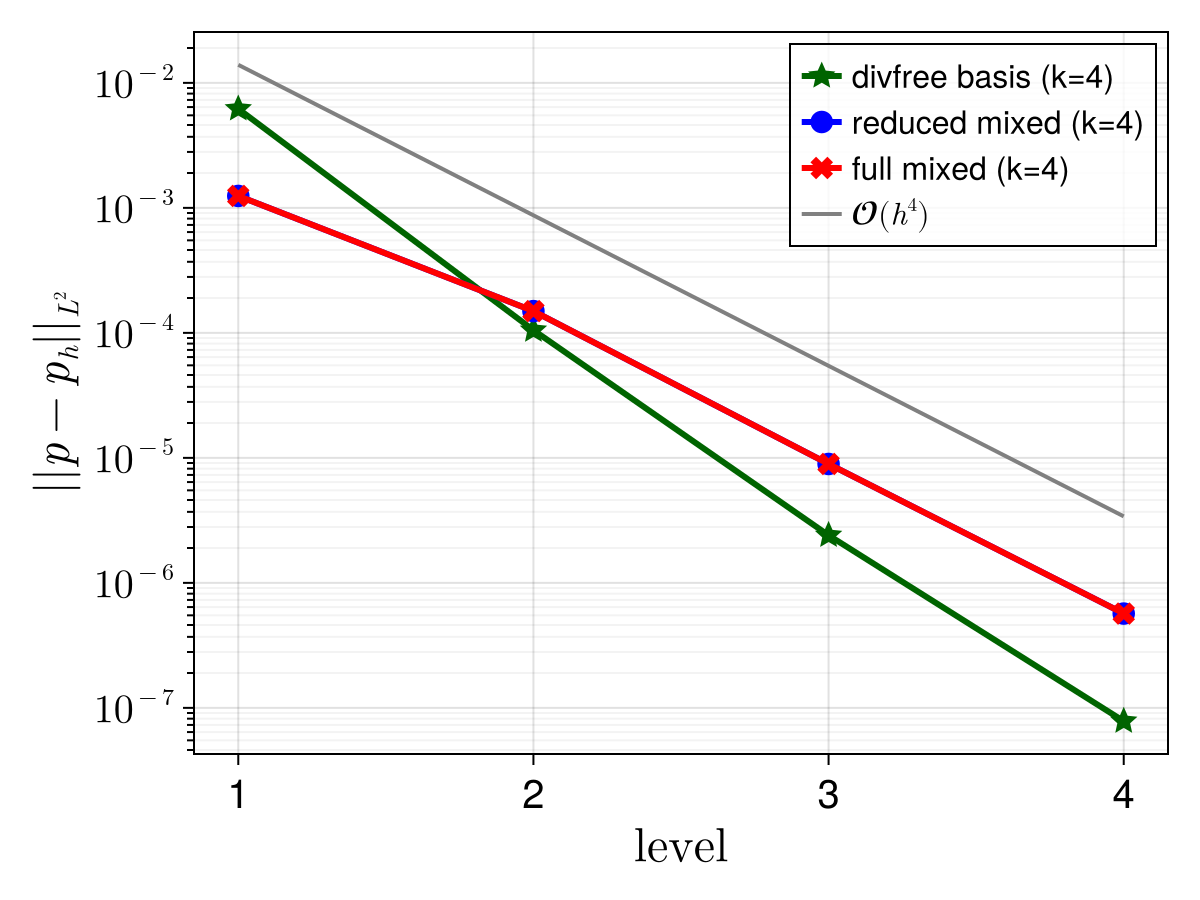}}
    \caption{Example~\ref{ex:2d} with $\nu = 10^{-6}$. $L^2(\Omega)$ velocity error (top) and pressure error (bottom) for the polynomial degree $k=1$ (left), $k=2$ (center), and $k=3$ (right), $\delta = 1$ (skew-symmetric case).} \label{fig:2d_conv_nonsym_small_nu}
\end{figure}

\begin{figure}[t!]
\centerline{
    \includegraphics[width=0.3\textwidth]{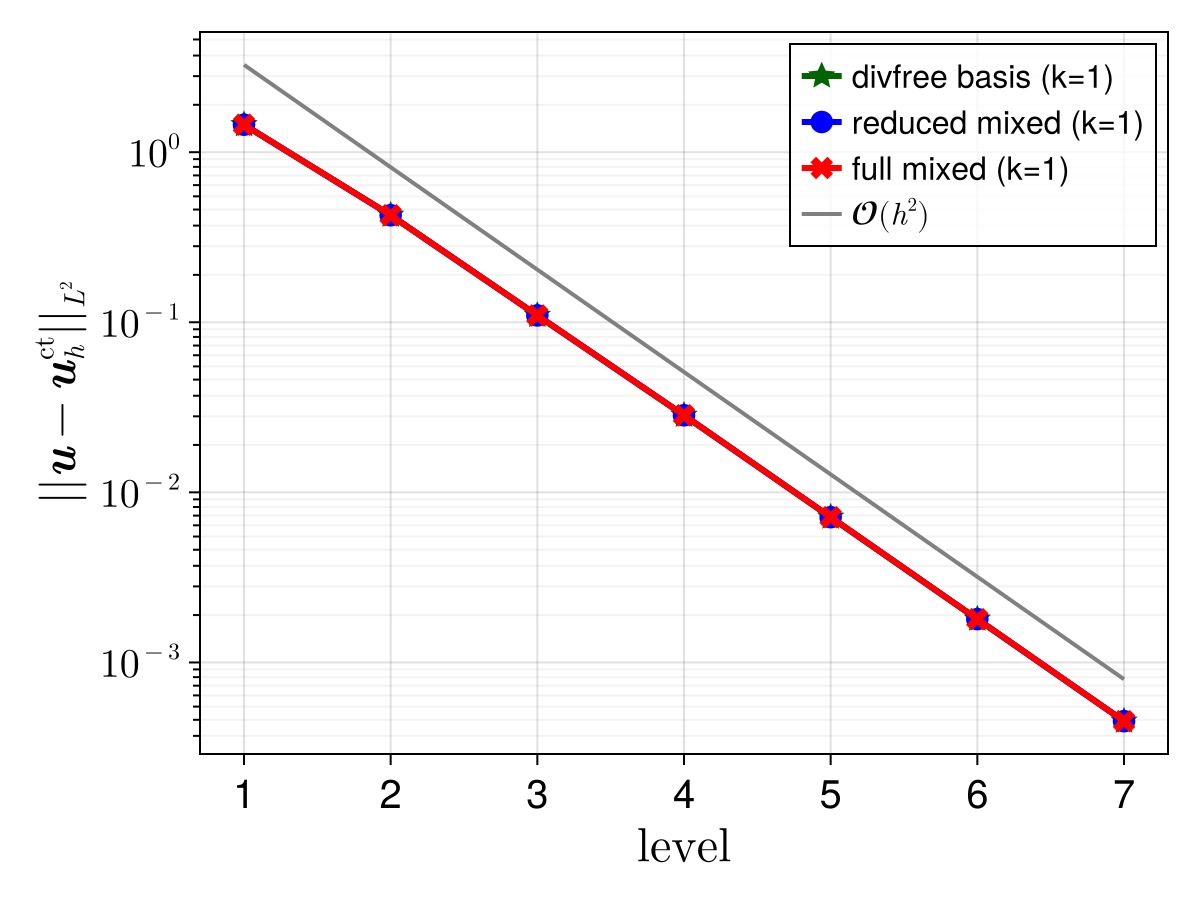}
    \includegraphics[width=0.3\textwidth]{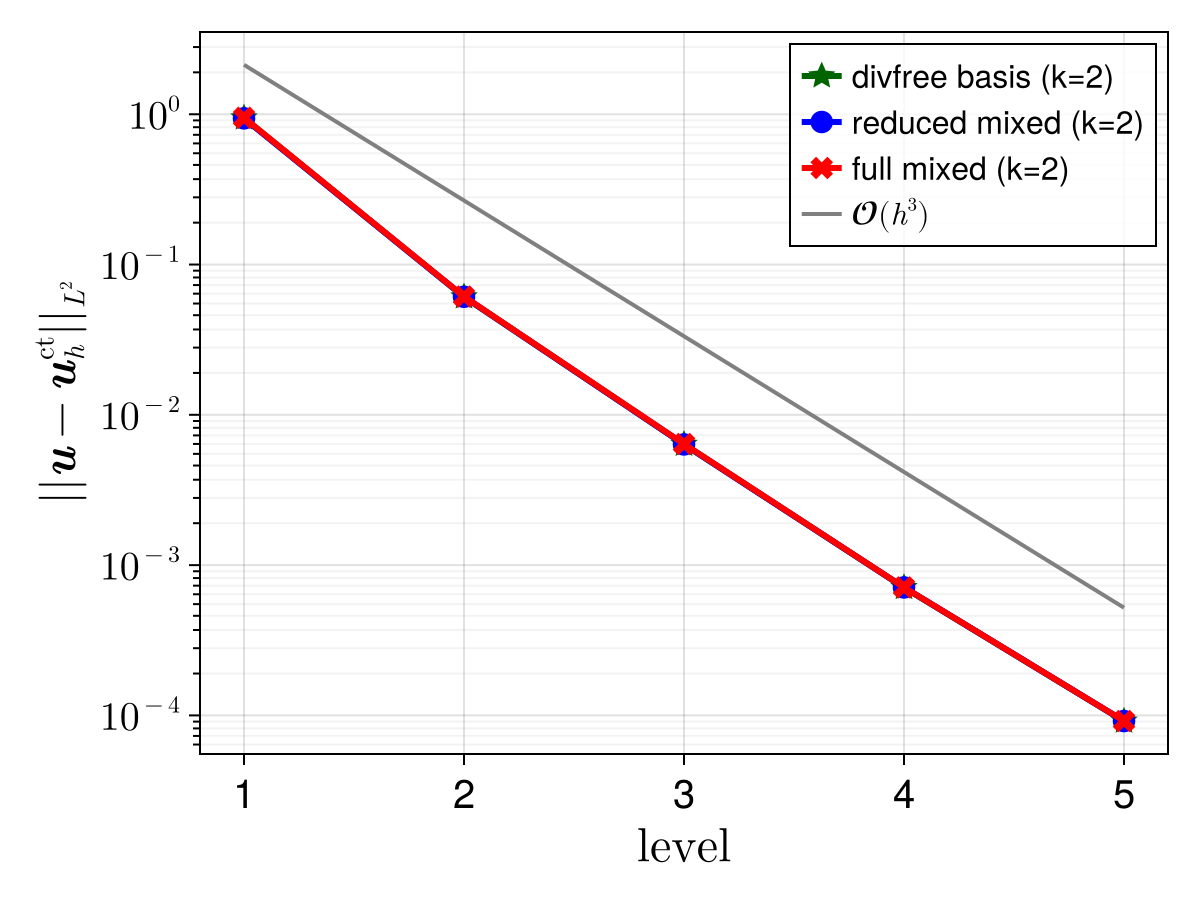}
    \includegraphics[width=0.3\textwidth]{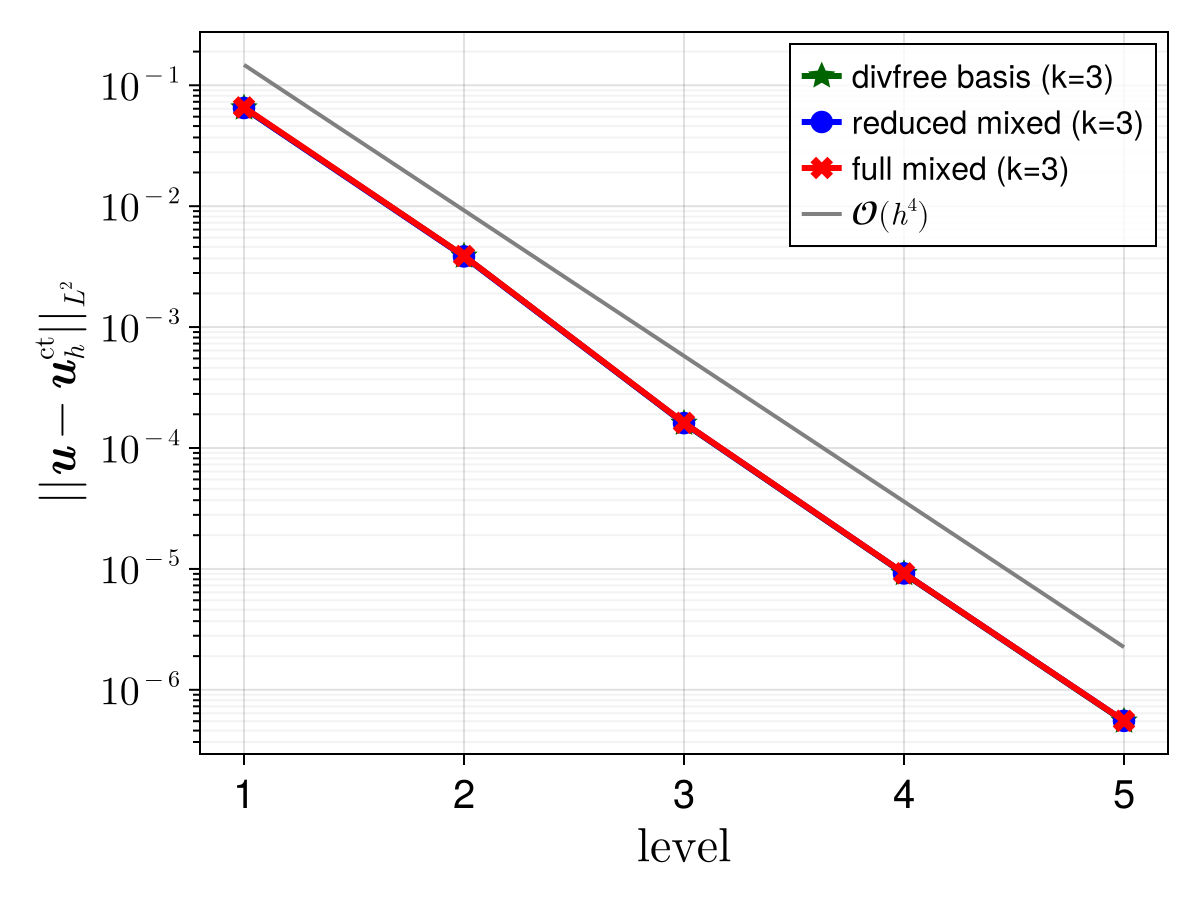}}
\centerline{
    \includegraphics[width=0.3\textwidth]{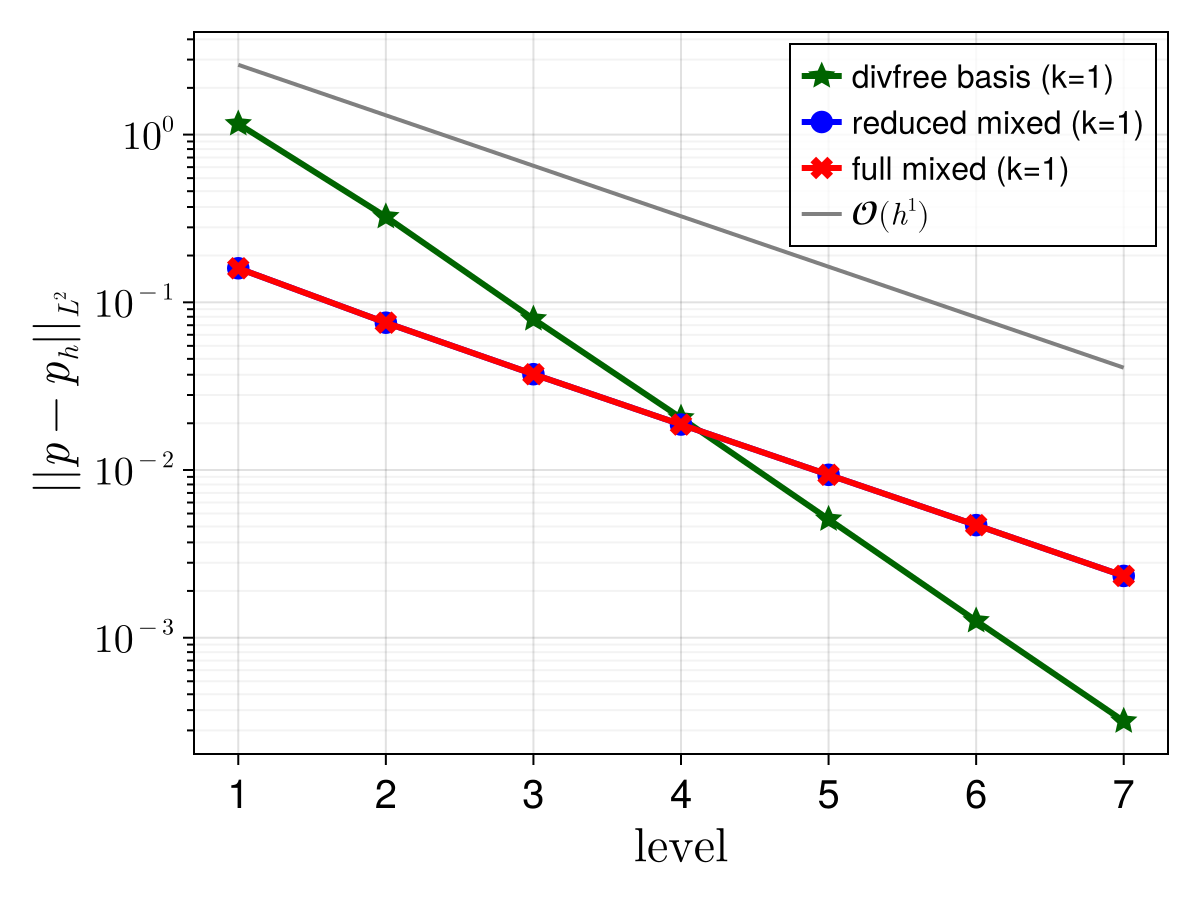}
    \includegraphics[width=0.3\textwidth]{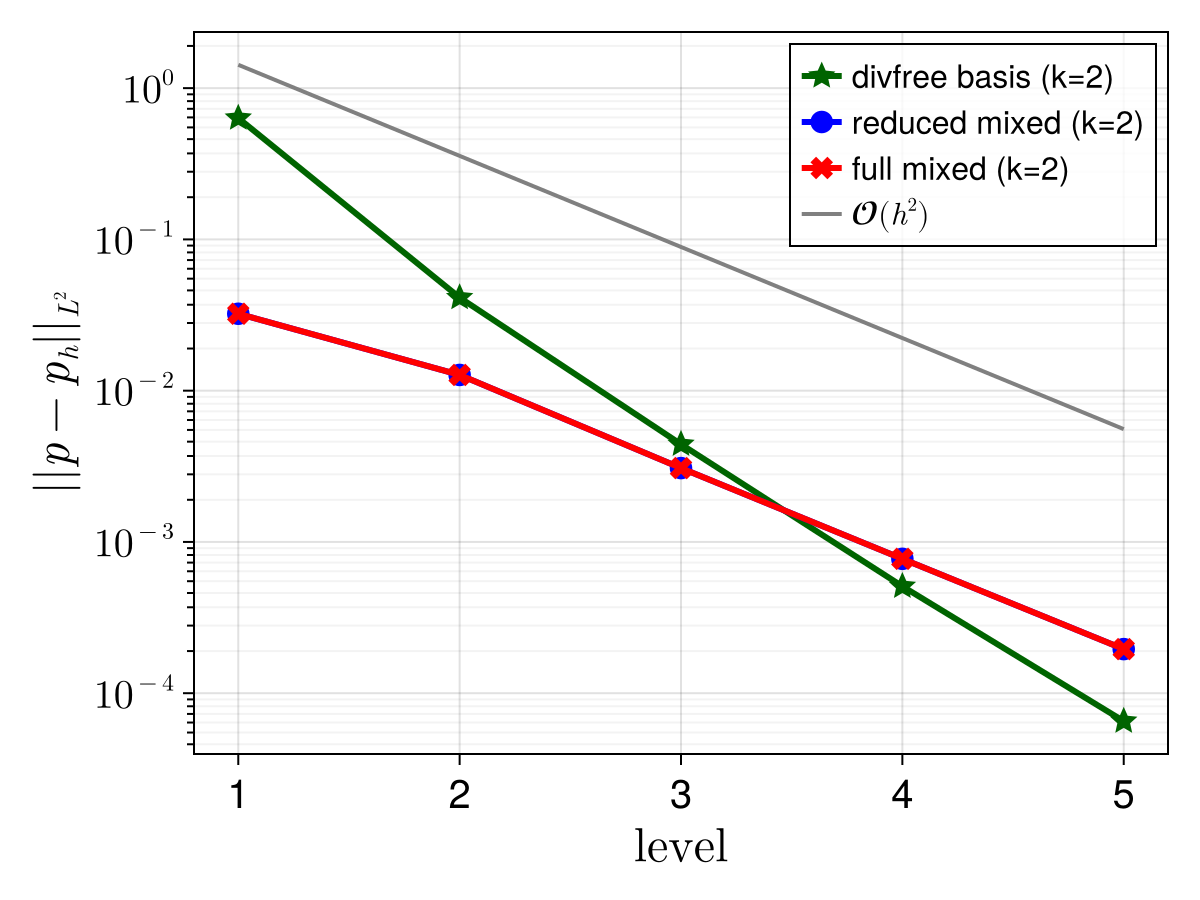}
    \includegraphics[width=0.3\textwidth]{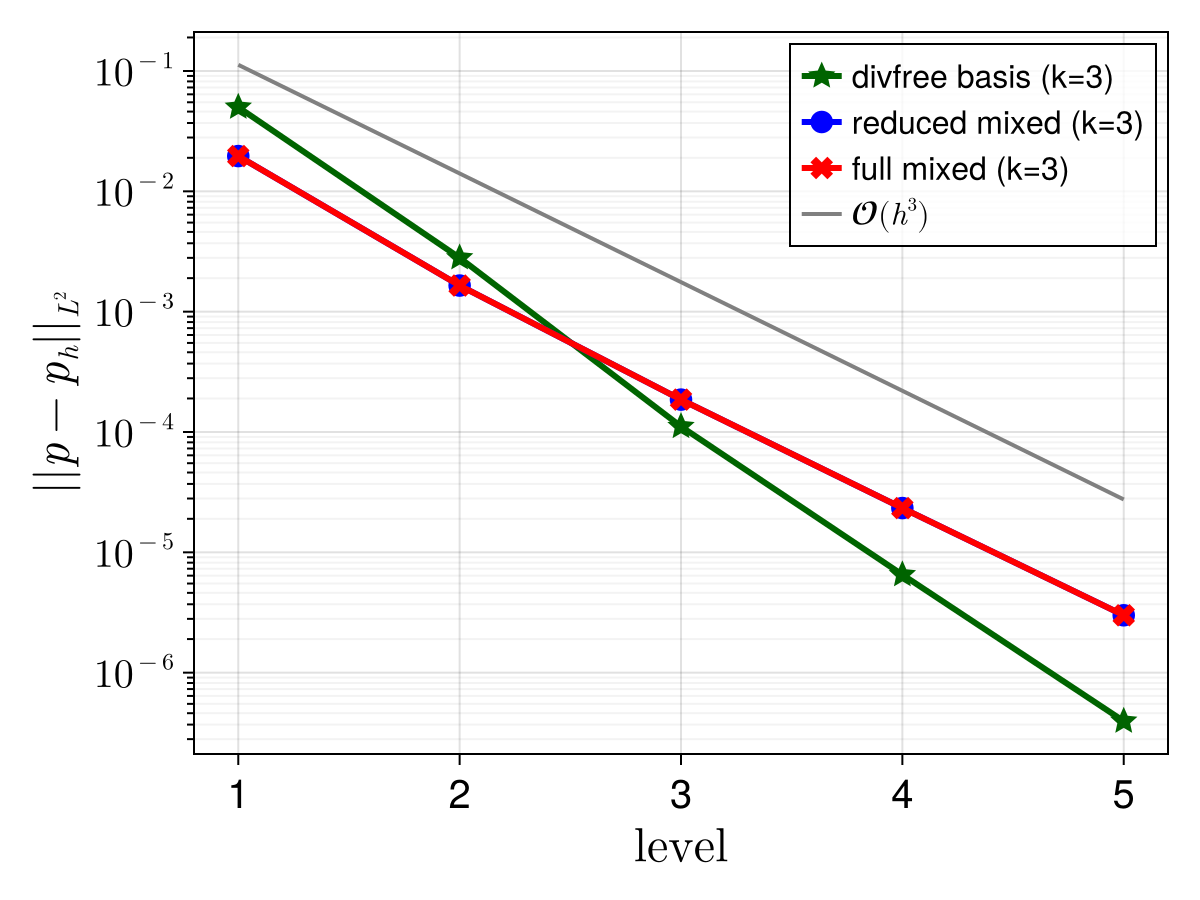}}
    \caption{Example~\ref{ex:2d} with $\nu = 10^{-6}$. $L^2(\Omega)$ velocity error (top) and pressure error (bottom) for the polynomial degree $k=2$ (left), $k=3$ (center), and $k=4$ (right), 
    $\delta=-1$ (symmetric case).} \label{fig:2d_conv_sym_small_nu}
\end{figure}

The convergence of the studied methods in the  large viscosity case $\nu=1$
is illustrated in Figures~\ref{fig:2d_conv_nonsym}
and~\ref{fig:2d_conv_sym} for the skew-symmetric and the symmetric case, respectively.
One can observe that all methods show the expected orders of convergence. The velocity
errors for all the methods are virtually the same.
The proposed pressure reconstruction for the 
decoupled methods leads often to results of superior accuracy compared to those obtained with the 
mixed methods. With respect to the velocity, similar observations can be made in the 
small viscosity case $\nu=10^{-6}$, compare Figures~\ref{fig:2d_conv_nonsym_small_nu} 
and~\ref{fig:2d_conv_sym_small_nu}. 
The velocity errors are of the same order of magnitude as for the large viscosity case. The latter observation 
supports the pressure-robustness of the methods.
Interestingly, the pressure reconstruction used in 
the decoupled methods leads on finer grids consistently to noticeably more accurate 
discrete pressures than those computed with the mixed methods. 

\begin{figure}[t!]
\centerline{
    \includegraphics[width=0.3\textwidth]{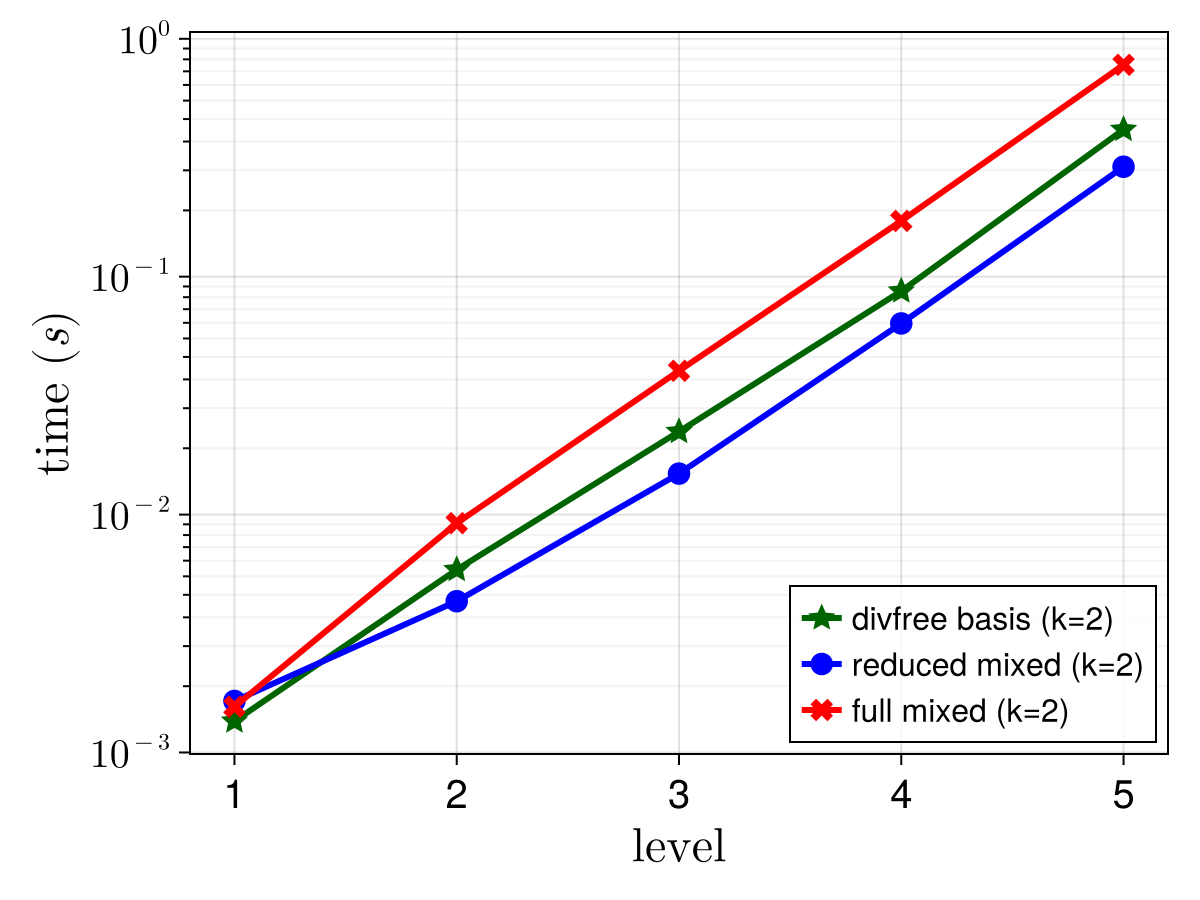}
    \includegraphics[width=0.3\textwidth]{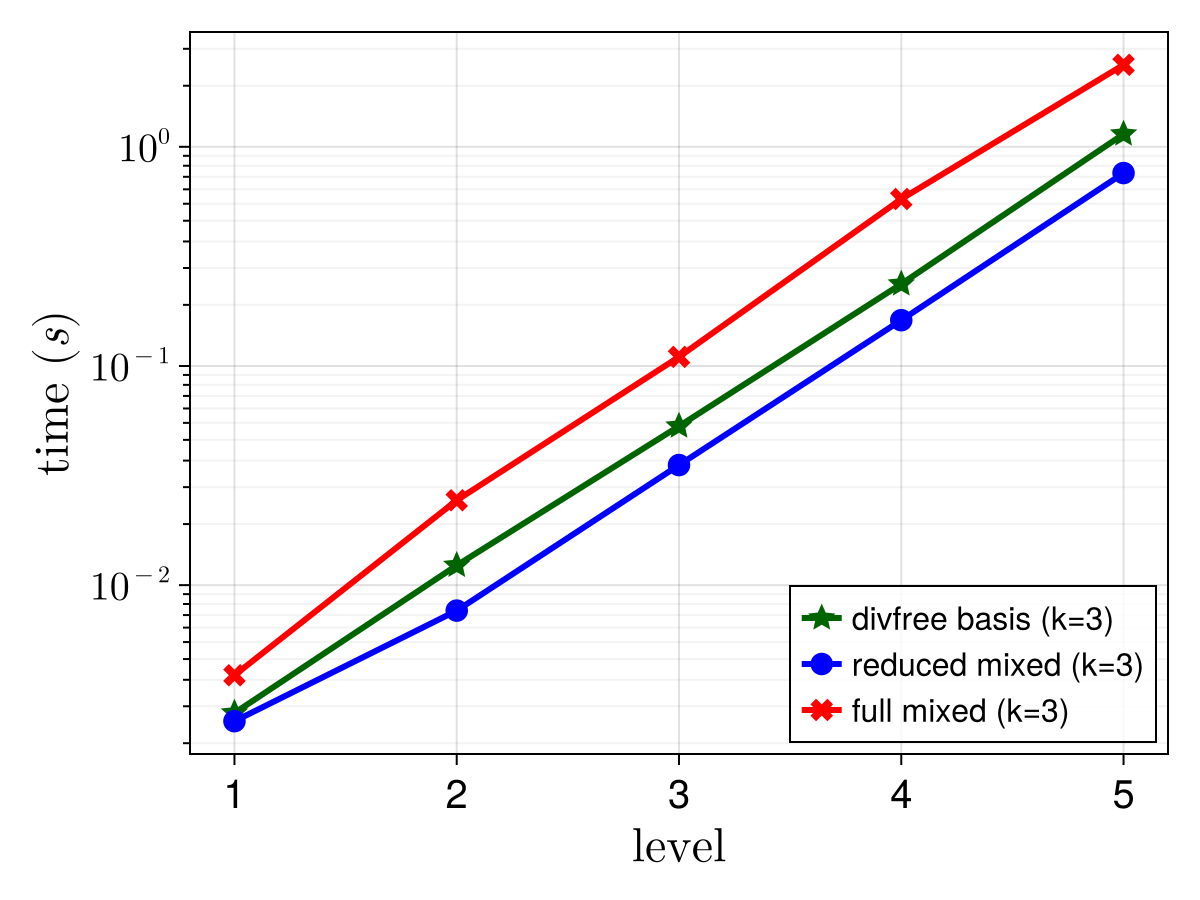}
    \includegraphics[width=0.3\textwidth]{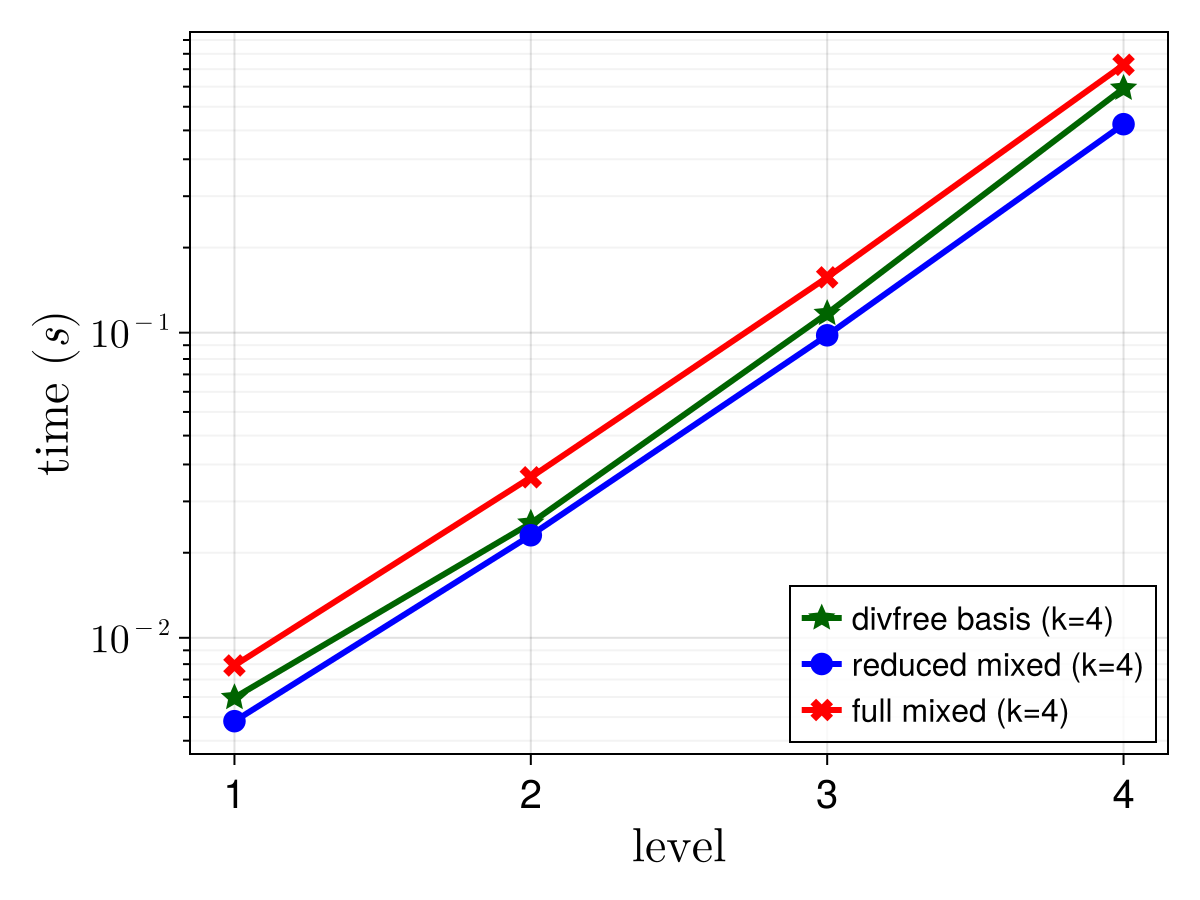}}
    \caption{Example~\ref{ex:2d} with $\nu = 1$. Solver times with Pardiso for polynomial degree $k=2$ (left), $k=3$ (center), and $k=4$ (right), $\delta = 1$ (skew-symmetric case).}
    \label{fig:2d_solvertimes_nonsym}
\end{figure}

\begin{figure}[t!]
\centerline{
    \includegraphics[width=0.3\textwidth]{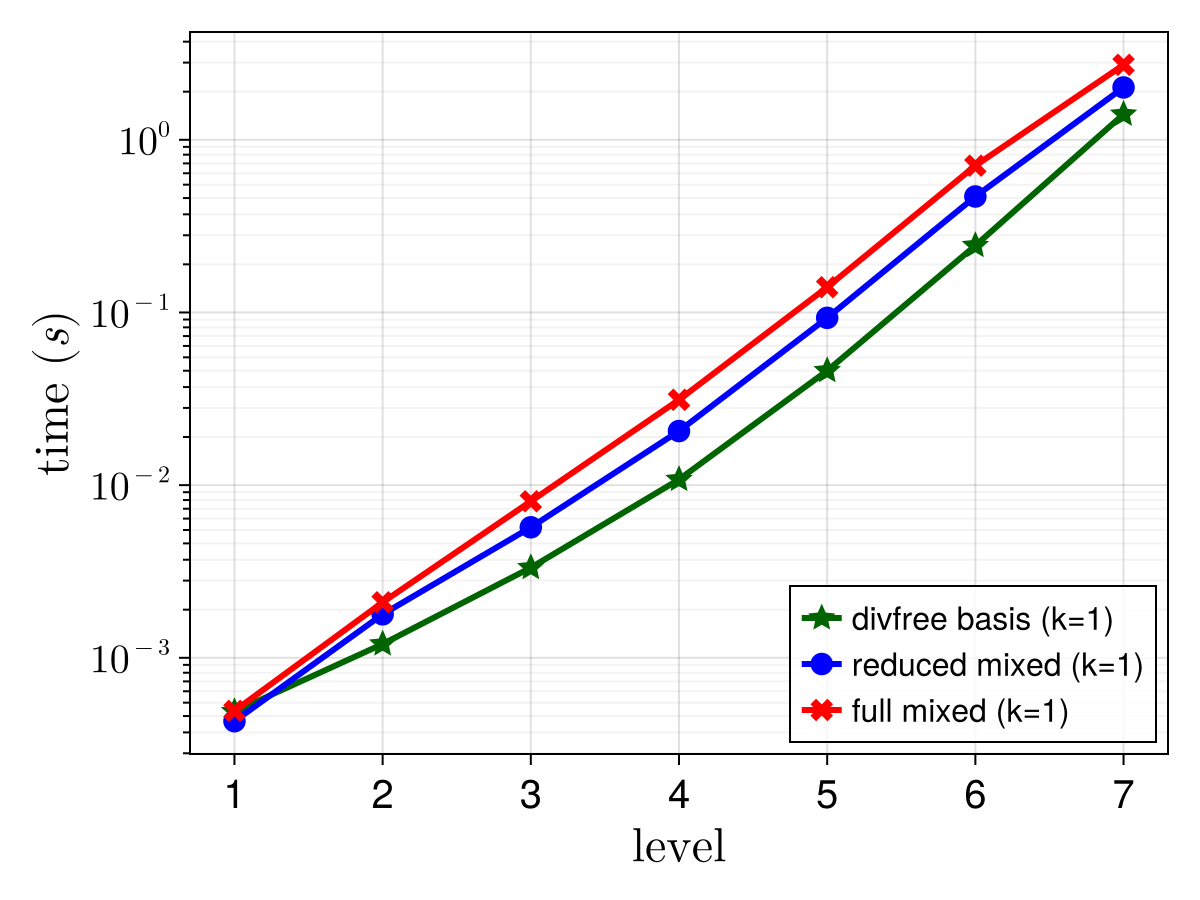}
    \includegraphics[width=0.3\textwidth]{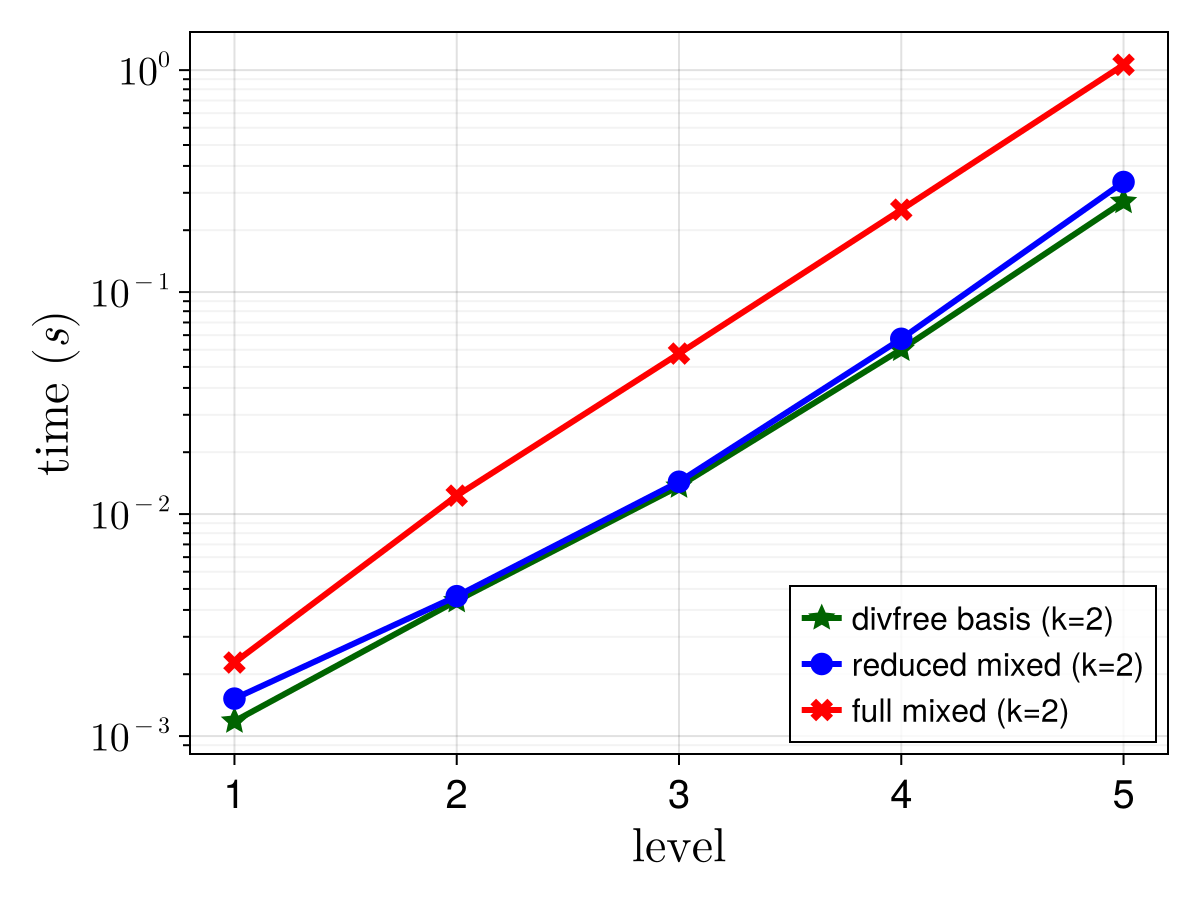}
    \includegraphics[width=0.3\textwidth]{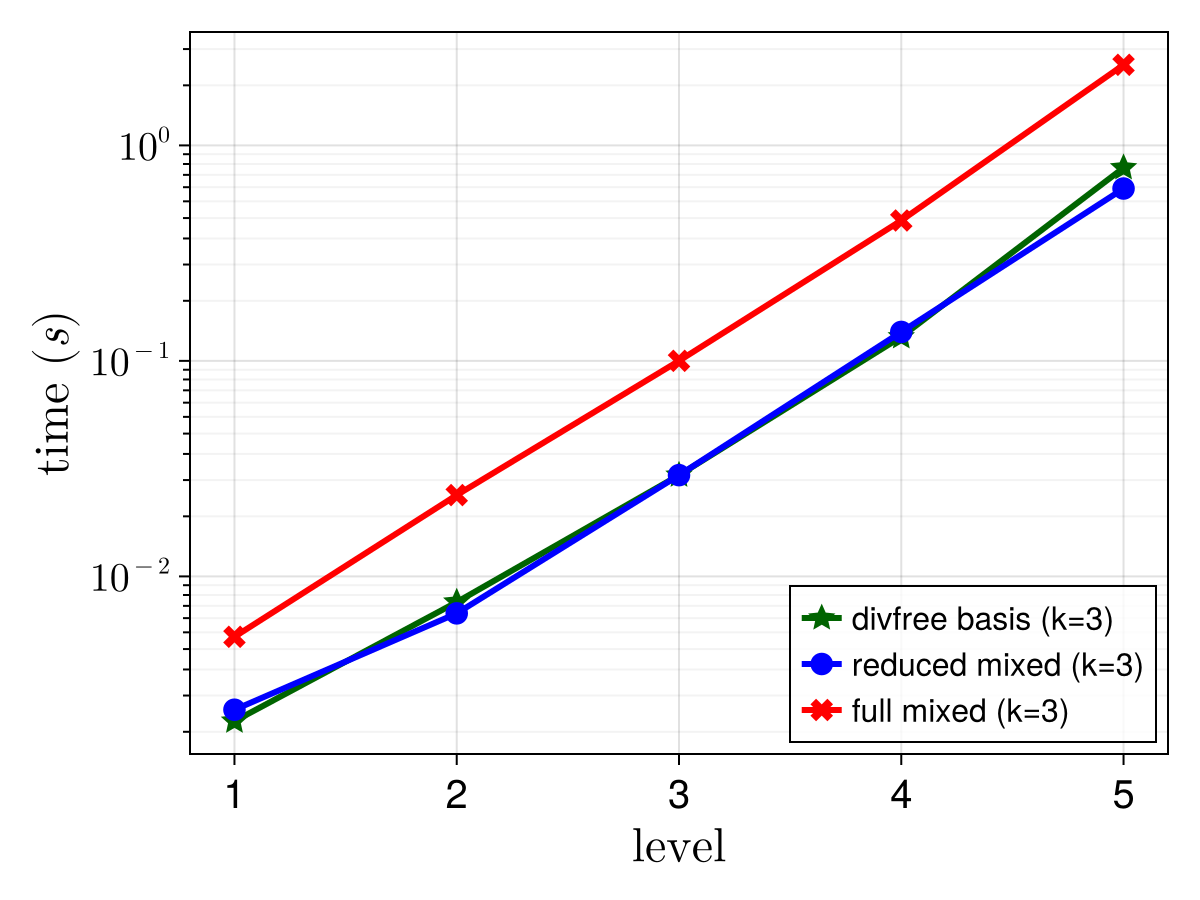}}
    \caption{Example~\ref{ex:2d} with $\nu = 1$. Solver times with Pardiso for polynomial degree $k=1$ (left), $k=2$ (center), and $k=3$ (right), $\delta = -1$ (symmetric case).}\label{fig:2d_solvertimes_sym}
\end{figure}

\begin{figure}[t!]
\centerline{
    \includegraphics[width=0.3\textwidth]{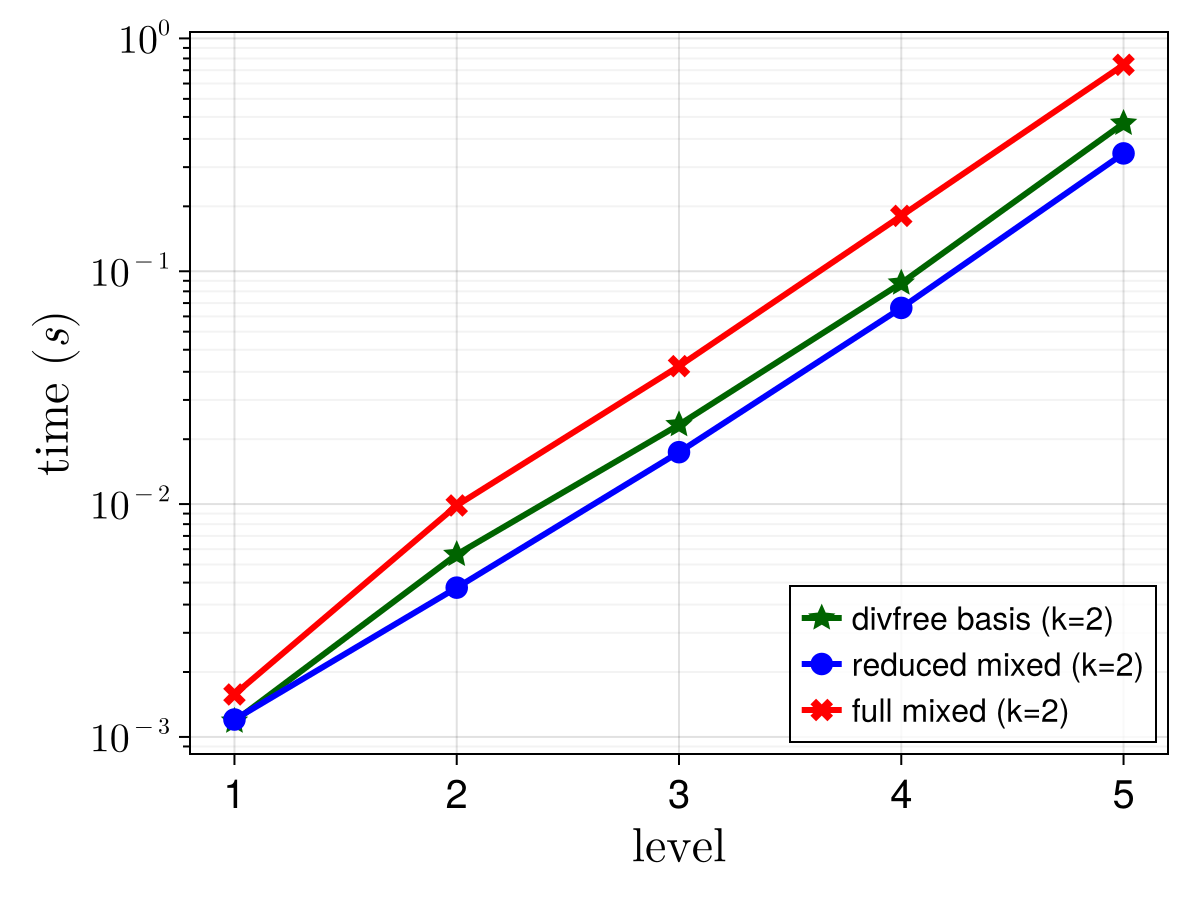}
    \includegraphics[width=0.3\textwidth]{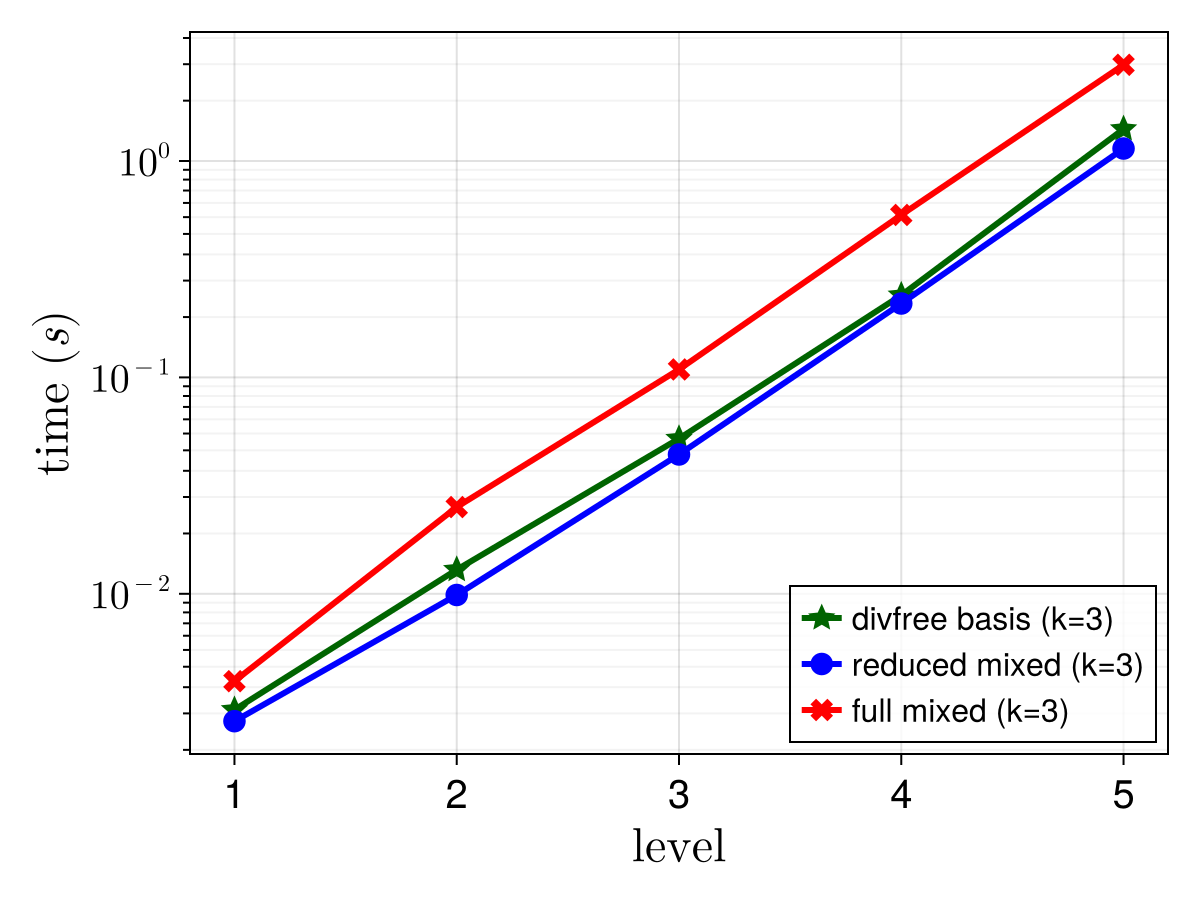}
    \includegraphics[width=0.3\textwidth]{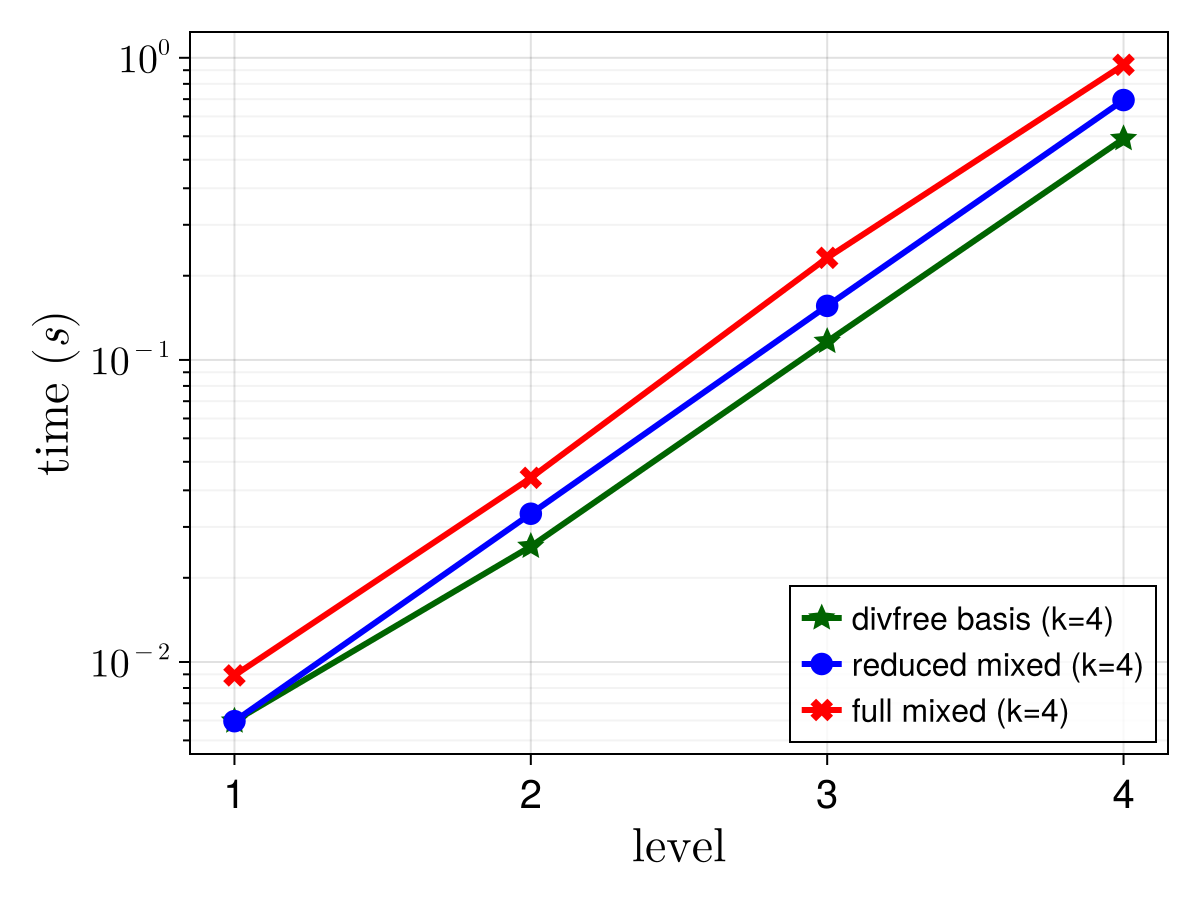}}
    \caption{Example~\ref{ex:2d} with $\nu = 10^{-6}$. Solver times with Pardiso for polynomial degree $k=2$ (left), $k=3$ (center), and $k=4$ (right), $\delta = 1$ (skew-symmetric case).}
    \label{fig:2d_solvertimes_nonsym_small_nu}
\end{figure}

\begin{figure}[t!]
\centerline{
    \includegraphics[width=0.3\textwidth]{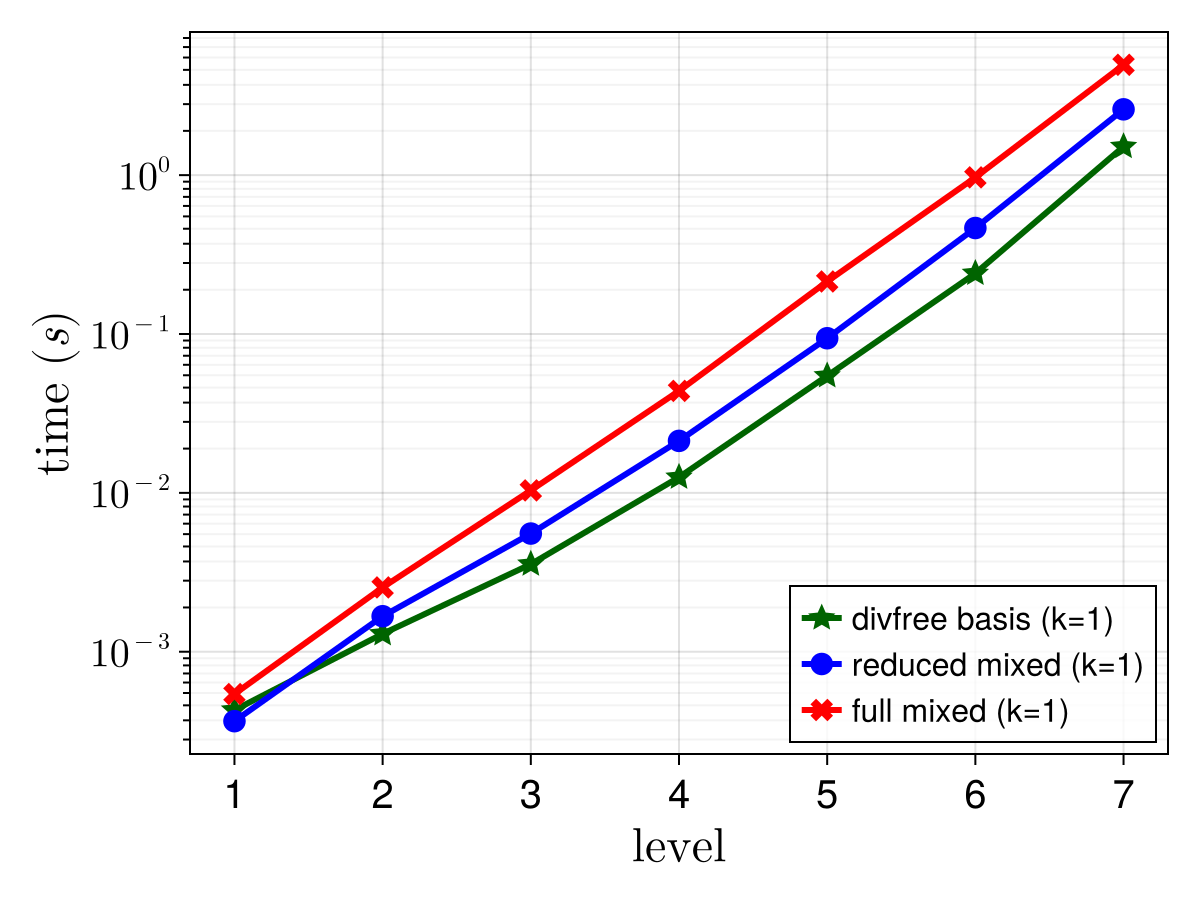}
    \includegraphics[width=0.3\textwidth]{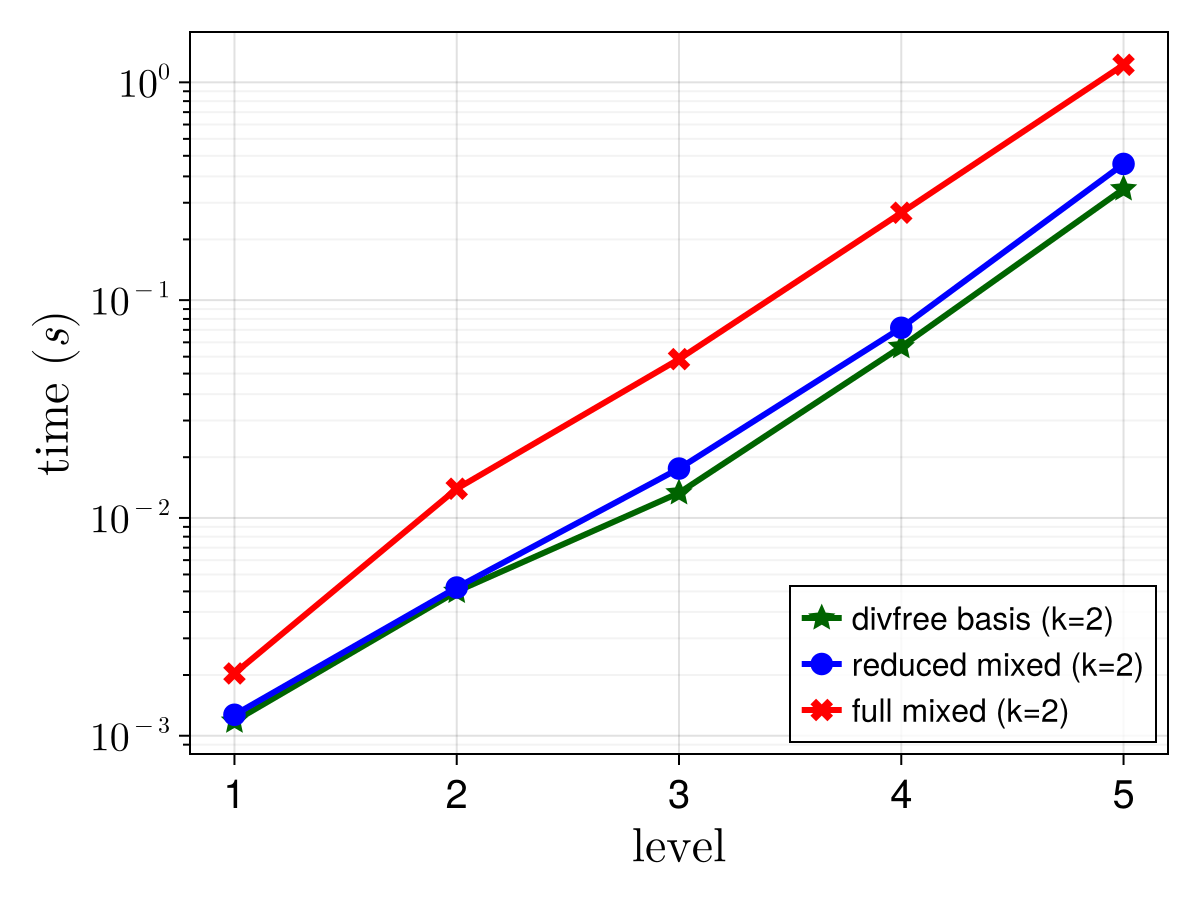}
    \includegraphics[width=0.3\textwidth]{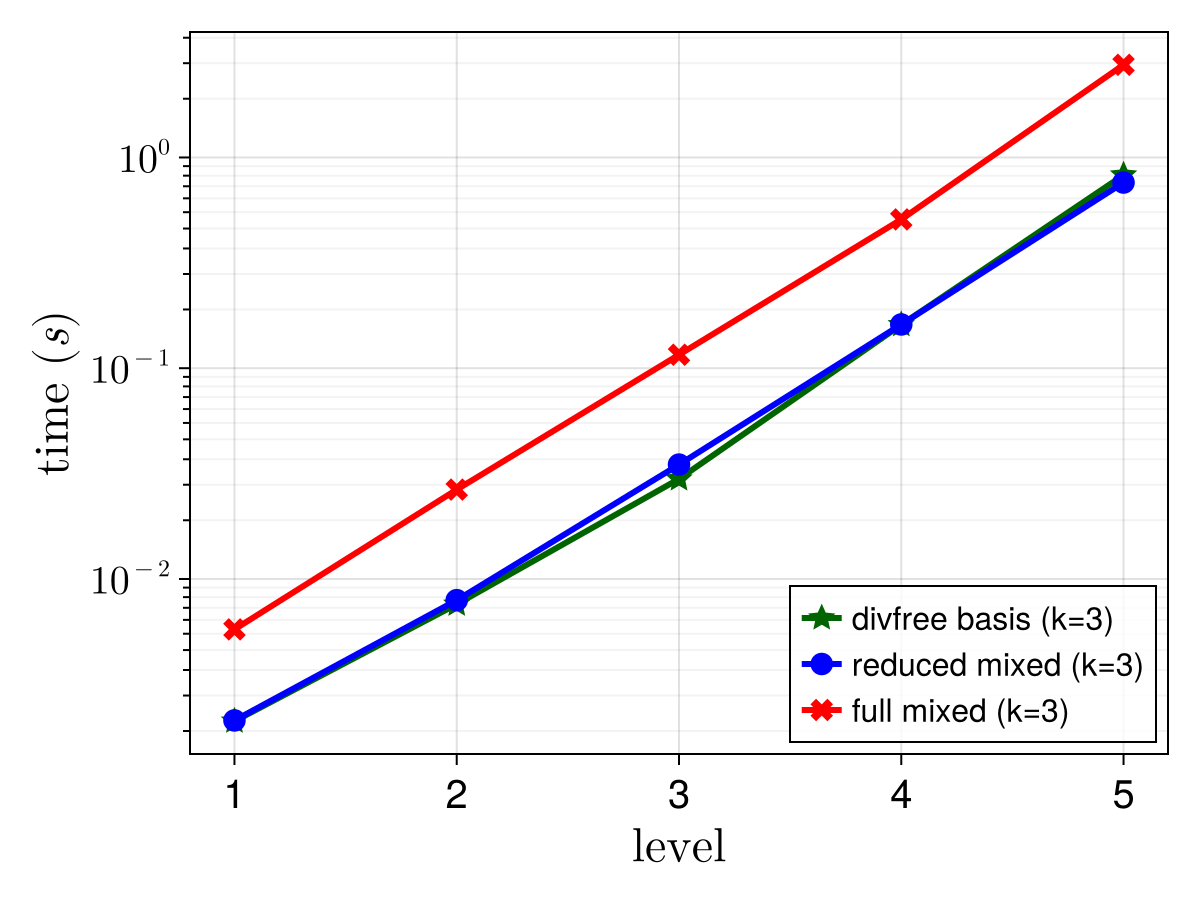}}
    \caption{Example~\ref{ex:2d} with $\nu = 10^{-6}$. Solver times with Pardiso for polynomial degree $k=1$ (left), $k=2$ (center), and $k=3$ (right), $\delta = -1$ (symmetric case).}\label{fig:2d_solvertimes_sym_small_nu}
\end{figure}

Figures~\ref{fig:2d_solvertimes_nonsym}--\ref{fig:2d_solvertimes_sym_small_nu} provide information about the efficiency of the 
studied methods in terms of computing times, using the direct sparse solver. 
As expected by construction, the reduced mixed methods are always more efficient than the full mixed methods.
Also the decoupled methods with divergence-free velocity basis are always faster than the full mixed methods. 
In some situations, the simulations with the decoupled methods were even somewhat faster than with the reduced
mixed methods, sometimes both approaches are of the same efficiency, and in the other cases the reduced mixed 
methods were a little bit faster. Altogether, the presented results show the good efficiency 
of both, the decoupled and the reduced mixed methods. We observe that, in the symmetric cases, which are of particular 
interest in practice, the decoupled methods are often a bit faster than the reduced mixed methods.

\begin{figure}[t!]
\centerline{
    \includegraphics[width=0.3\textwidth]{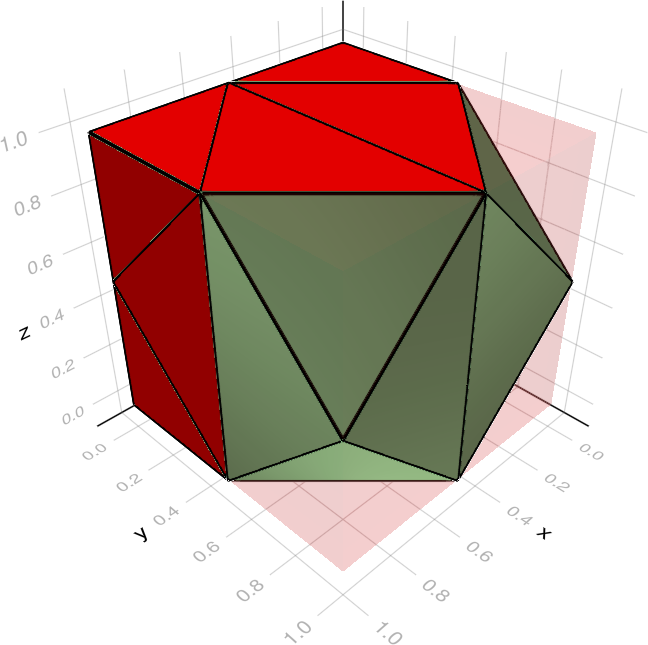}\hspace*{1em}
    \includegraphics[width=0.3\textwidth]{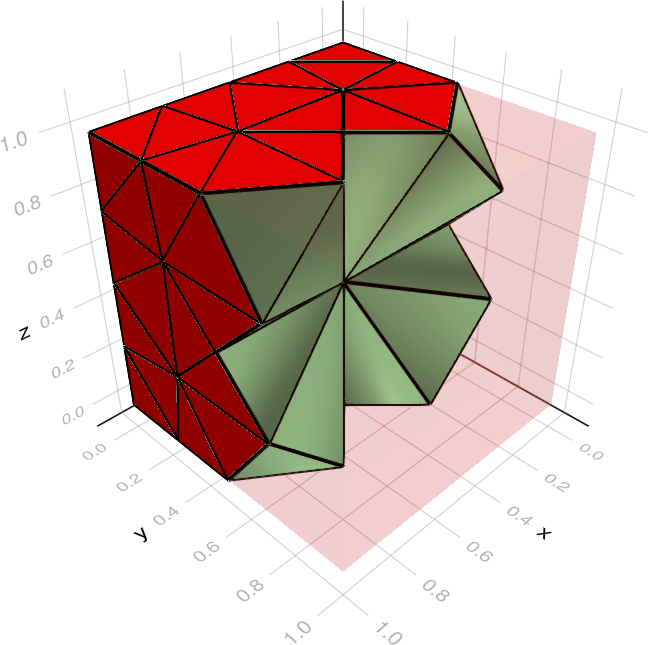}\hspace*{1em}
    \includegraphics[width=0.3\textwidth]{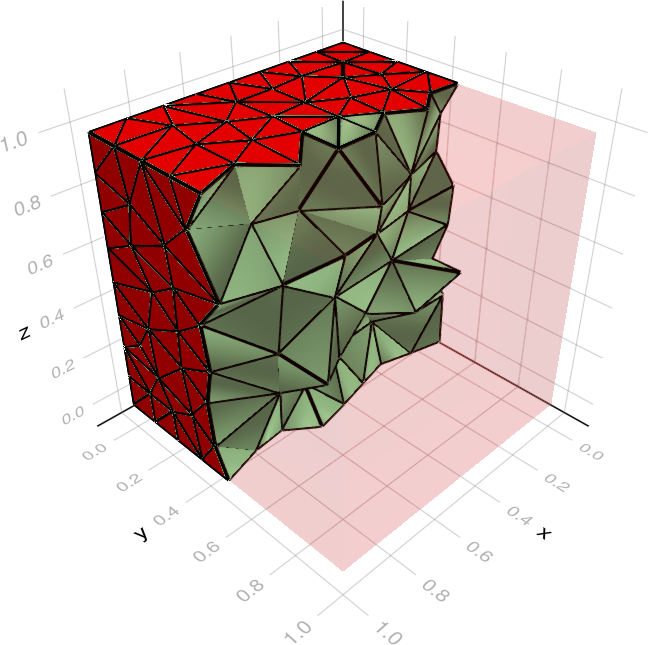}}
    \caption{Example~\ref{ex:3d}. Slices of the three coarsest computational grids.} \label{fig:3d_grids}
\end{figure}

\subsection{A three-dimensional  example}\label{ex:3d}

This example considers a Stokes problem in $\Omega = (0,1)^3$ with $\nu=1$ and the prescribed solution 
\[
\vecb{u}(x,y,z)  = \mathrm{curl} \left(\sin(y),\cos(z),\sin(x) \right)^T, \quad
    p(x,y,z)     =  \frac{\cos(4 \pi x) - \cos(4 \pi y)}4,
\]
which determines the right-hand side $\vecb{f}$ and the Dirichlet boundary condition on $\partial\Omega$.

\begin{table}[t!]
    \caption{\label{tab:3d_dof_nnz_nosym}Example~\ref{ex:3d}.  Number of degrees of freedom (ndofs) and number of non-zero sparse matrix entries (nnz)
    of the full system for the decoupled method with divergence-free basis (dfb) and the reduced mixed method (red)
    for different degrees $k$ and $\delta = 1$ (skew-symmetric case).}
    \footnotesize
    \begin{center}
    \input{tables/3d/table_order=2_levels=4.txt}

    \input{tables/3d/table_order=3_levels=4.txt}

    \end{center}
\end{table}

\begin{table}[t!]
    \caption{\label{tab:3d_dof_nnz_sym}Example~\ref{ex:3d}. Number of degrees of  freedom (ndofs) and number of non-zero sparse matrix entries (nnz)
    of the full system for the decoupled method with divergence-free basis (dfb) and the reduced mixed method (red) 
    for different degrees $k$ and $\delta = -1$ (symmetric case).}
    \footnotesize
    \begin{center}
    \input{tables/3d/table_order=1_levels=5_symmetric.txt}

    \input{tables/3d/table_order=2_levels=4_symmetric.txt}

    \input{tables/3d/table_order=3_levels=4_symmetric.txt}

    \end{center}
\end{table}

Numerical results are presented only for the reduced mixed methods and the decoupled methods with 
divergence-free velocity basis. 
The simulations considered a sequence of unstructured grids with decreasing mesh width, 
see Figure~\ref{fig:3d_grids} for the three coarsest grids, and polynomials
up to degree $k=3$.
To impose the non-homogeneous Dirichlet boundary condition in the decoupled methods, the method that restricts 
the Darcy-type problem to a subdomain close to the boundary, see the end of Section~\ref{sec:nonhomo_BC_3d}, 
was applied.
The number of degrees of freedom and non-zero matrix entries is given in Tables~\ref{tab:3d_dof_nnz_nosym}
and~\ref{tab:3d_dof_nnz_sym}. It can be observed that usually both, the numbers of degrees of freedom and
the numbers of non-zero entries, are smaller for the reduced methods, except $k=2$.
For $k=2$, the situation with respect to the non-zero matrix entries is vice versa.

\begin{figure}[t!]
\centerline{
    \includegraphics[width=0.3\textwidth]{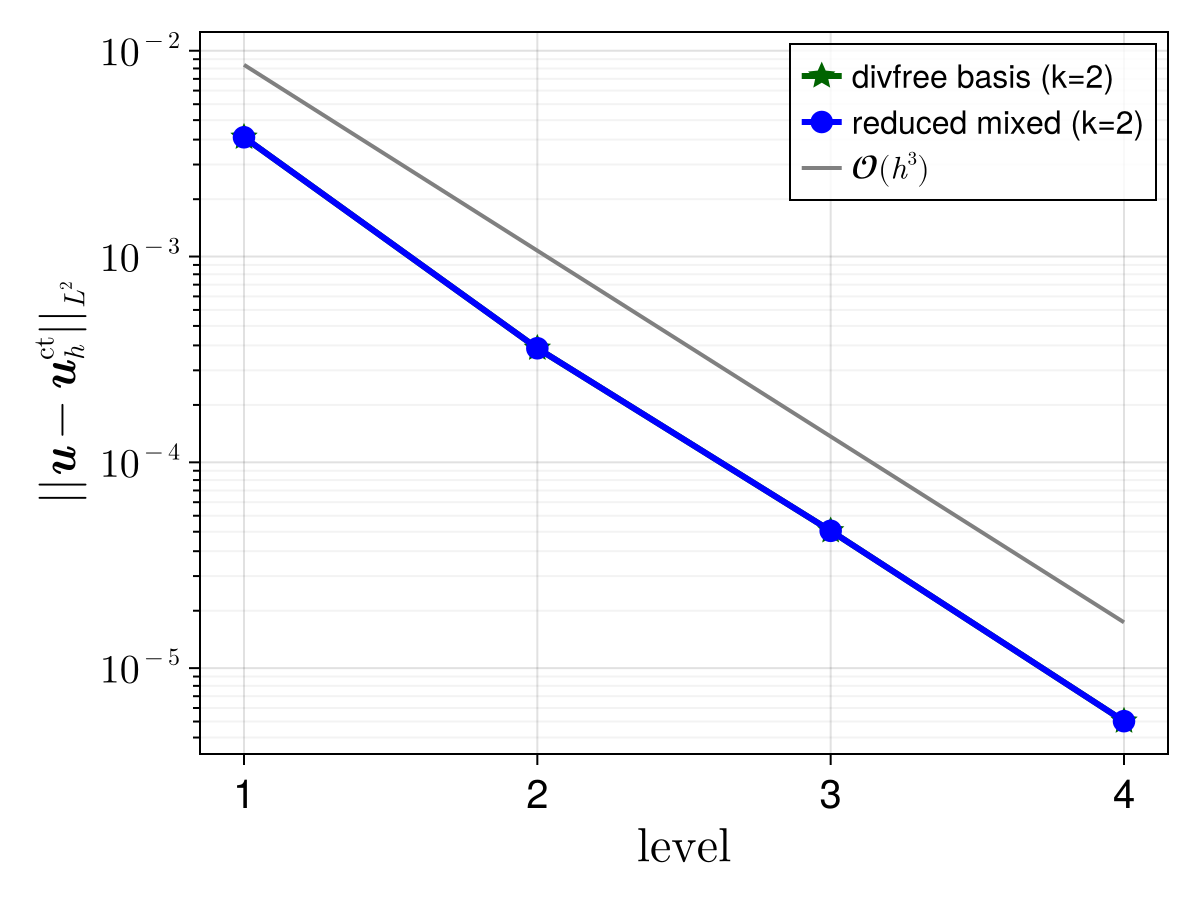}
    \includegraphics[width=0.3\textwidth]{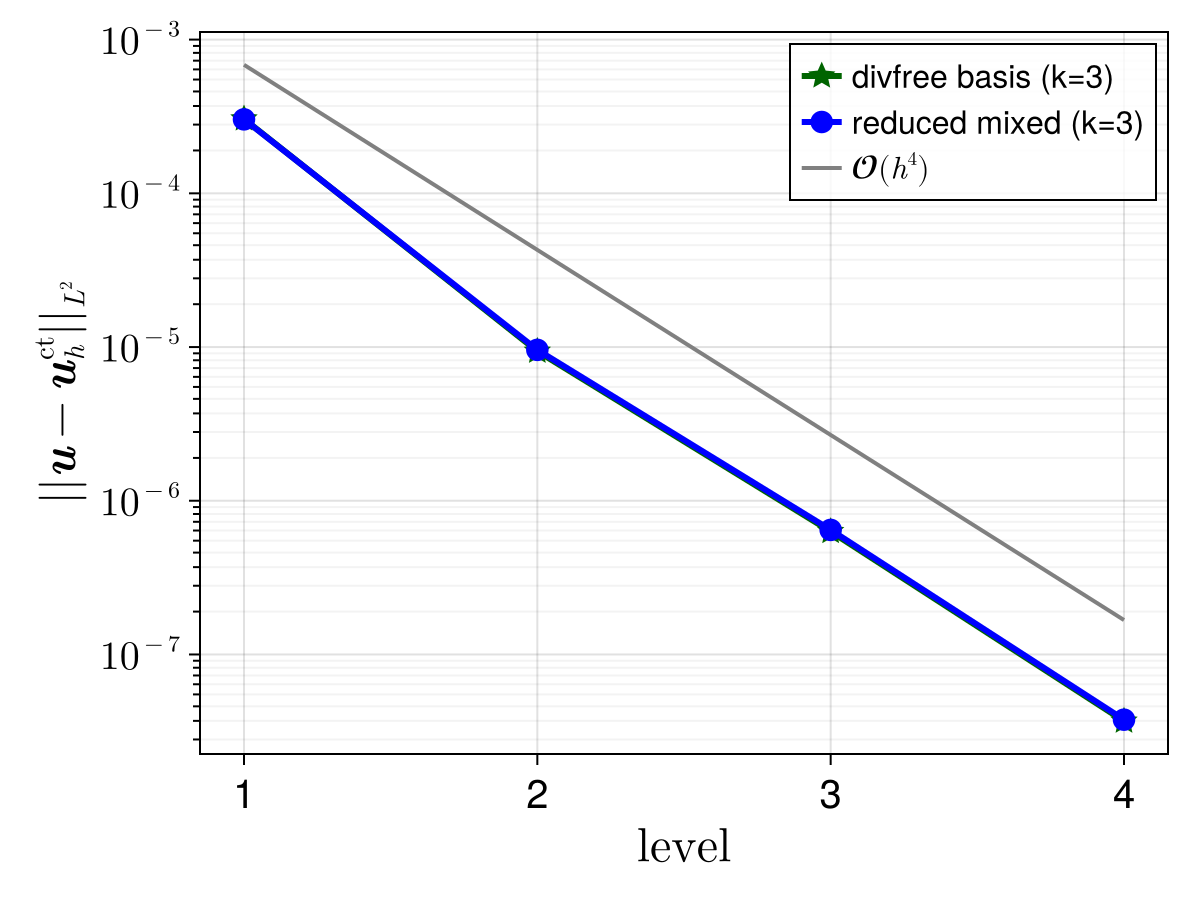}}
\centerline{
    \includegraphics[width=0.3\textwidth]{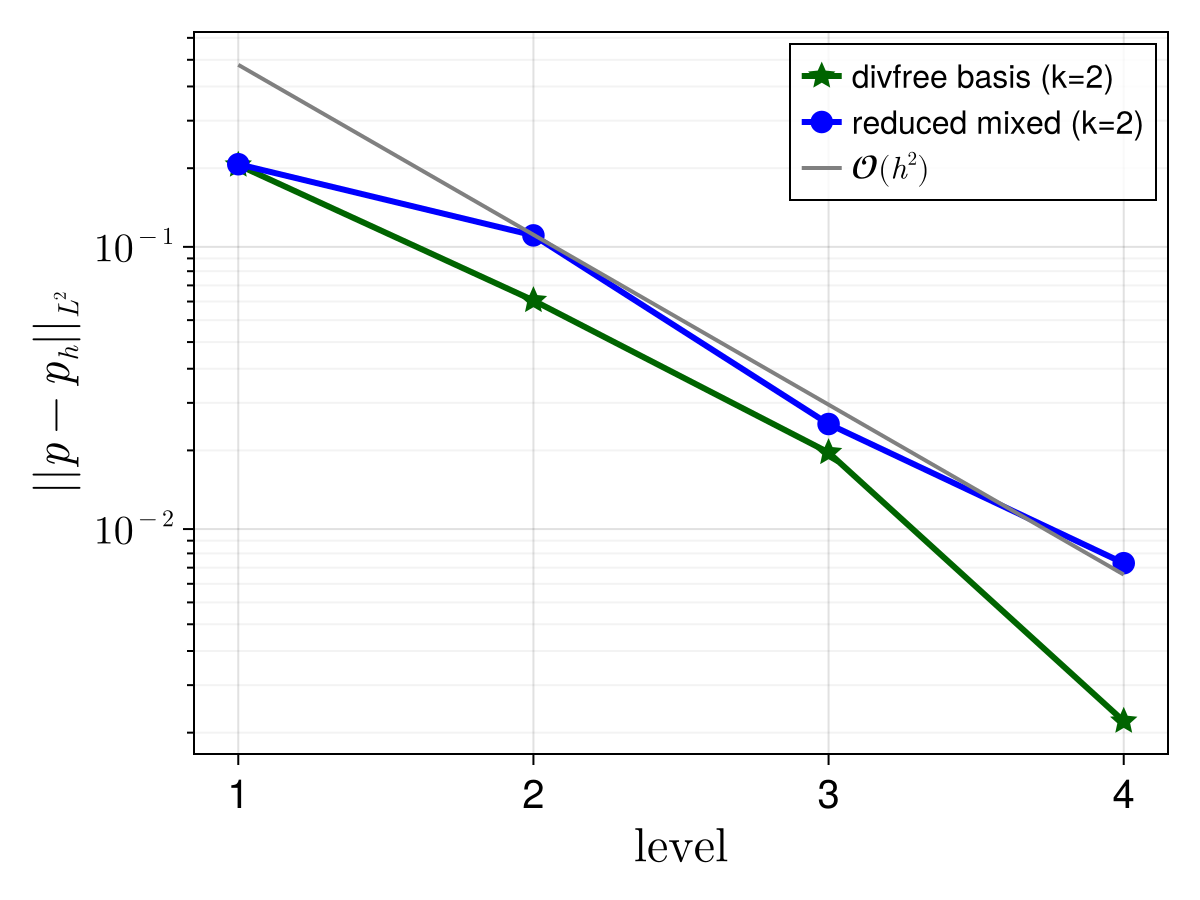}
    \includegraphics[width=0.3\textwidth]{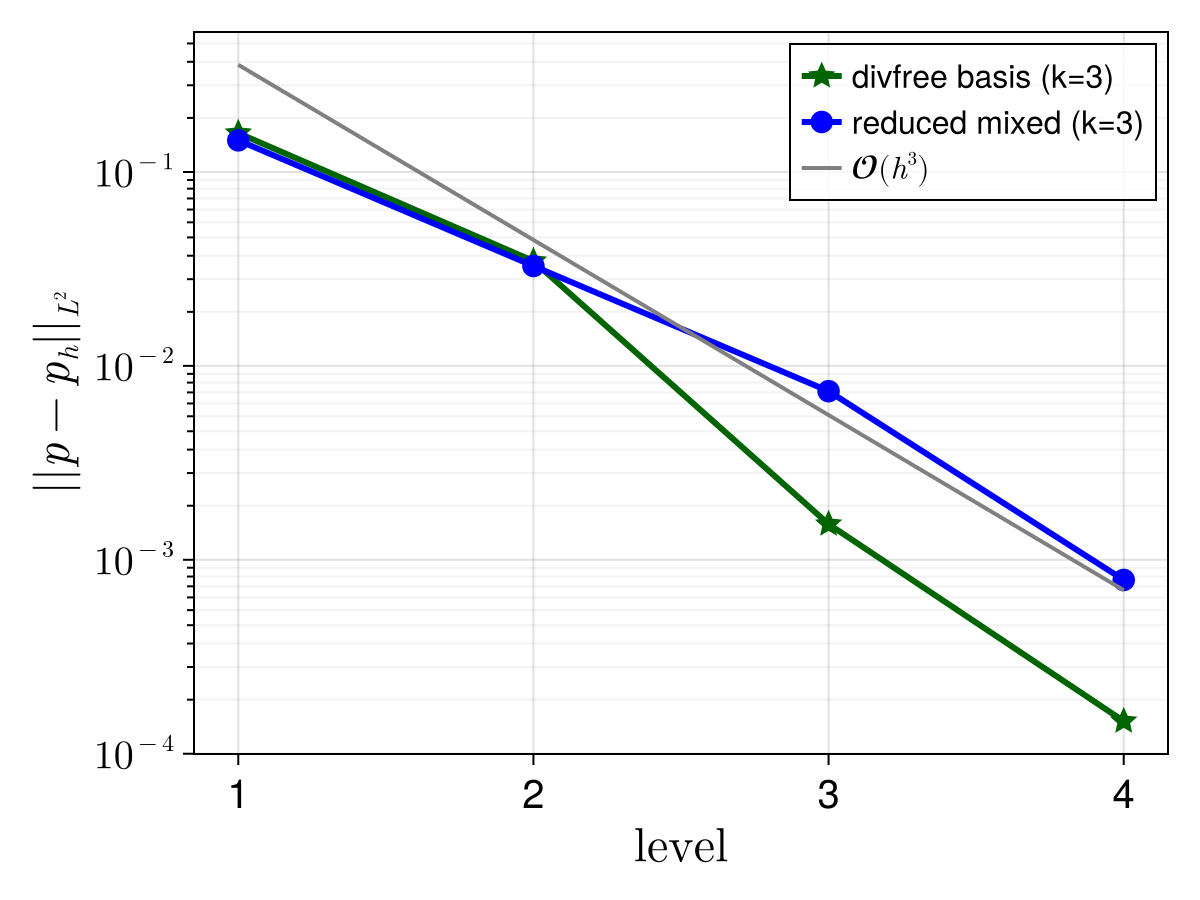}}
    \caption{Example~\ref{ex:3d}. $L^2(\Omega)$ velocity error (top) and pressure error (bottom) for the polynomial degree $k=2$ (left) and $k=3$ (right), $\delta = 1$ (skew-symmetric case).} \label{fig:3d_conv_nonsym}
\end{figure}

\begin{figure}[t!]
\centerline{
    \includegraphics[width=0.3\textwidth]{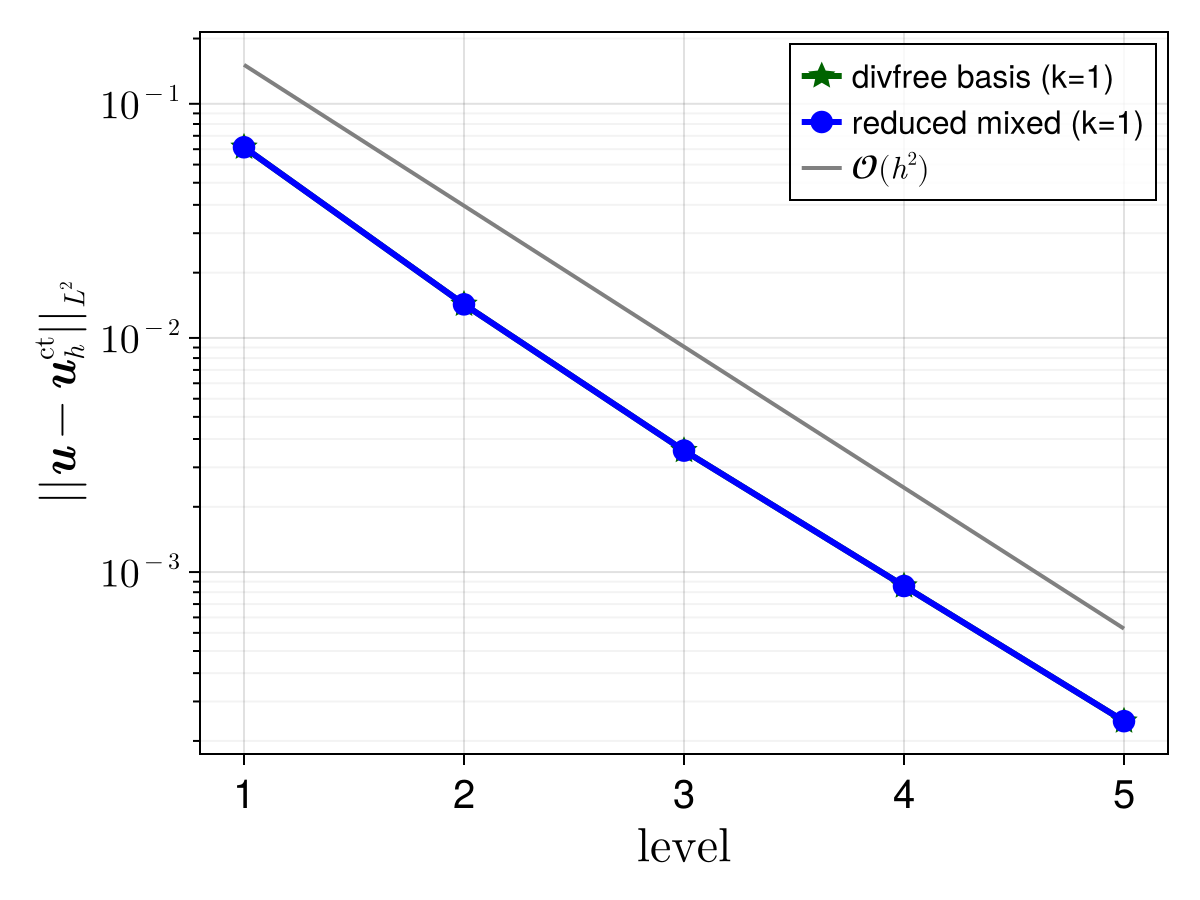}
    \includegraphics[width=0.3\textwidth]{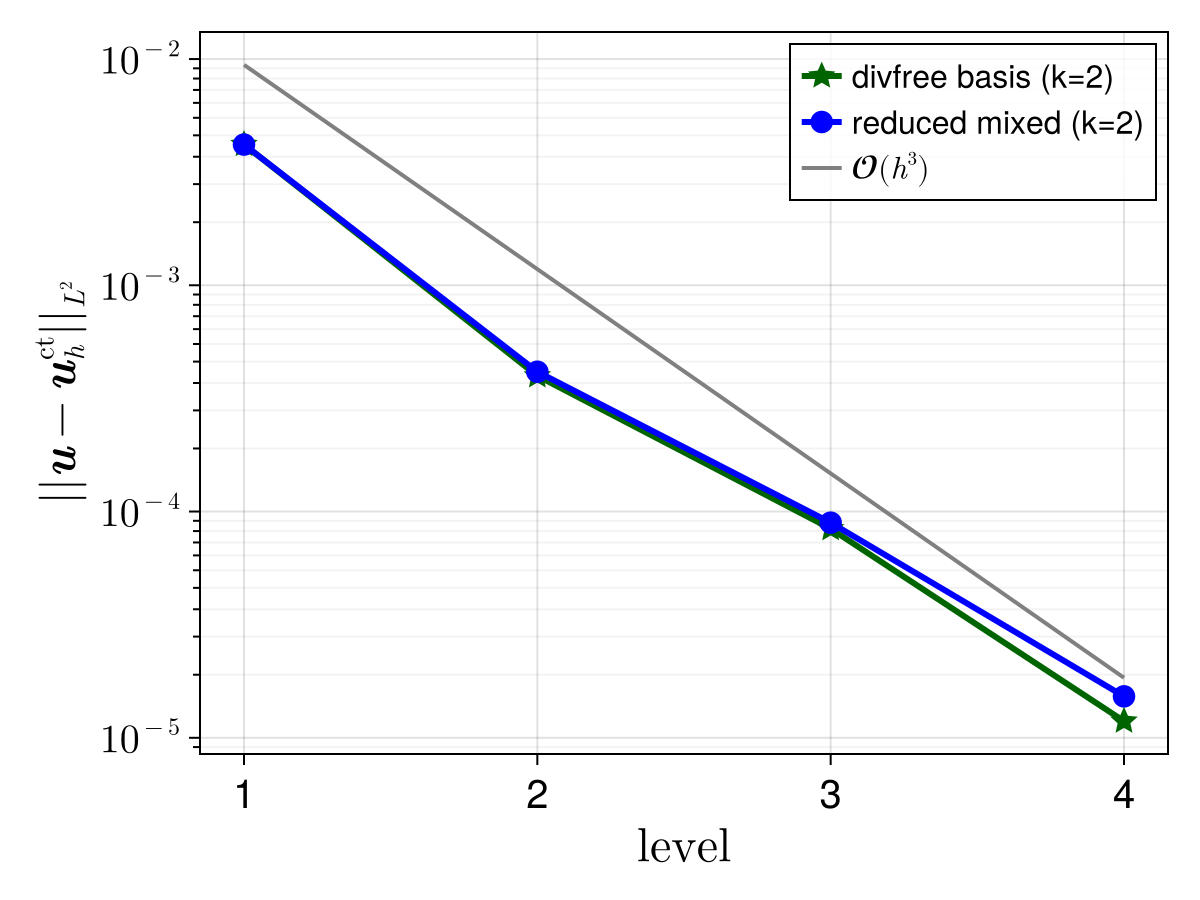}
    \includegraphics[width=0.3\textwidth]{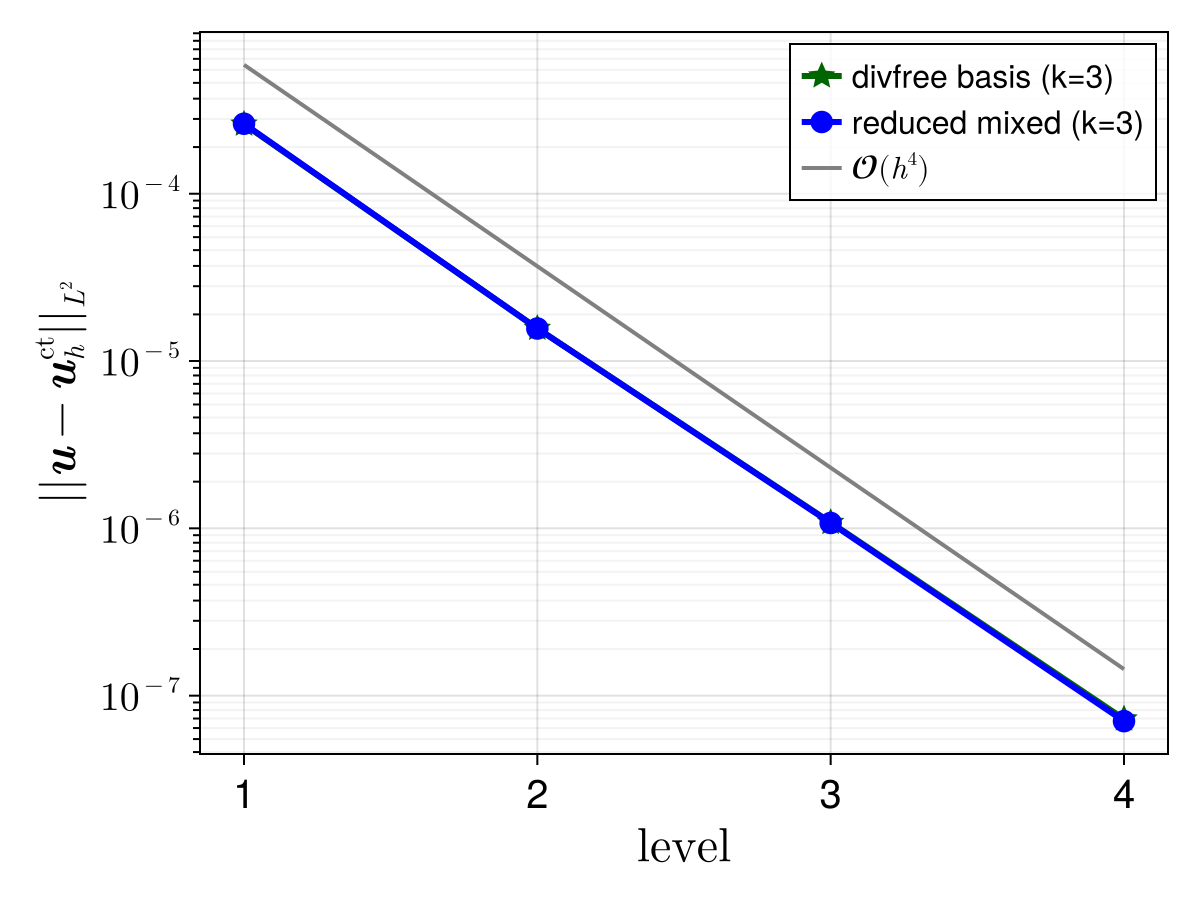}}
\centerline{    
    \includegraphics[width=0.3\textwidth]{pics/3d/L2error_velocity_order=1_solver=PardisoMKL_mtype=11_levels=5_symmetric_alpha=1,0.png}
    \includegraphics[width=0.3\textwidth]{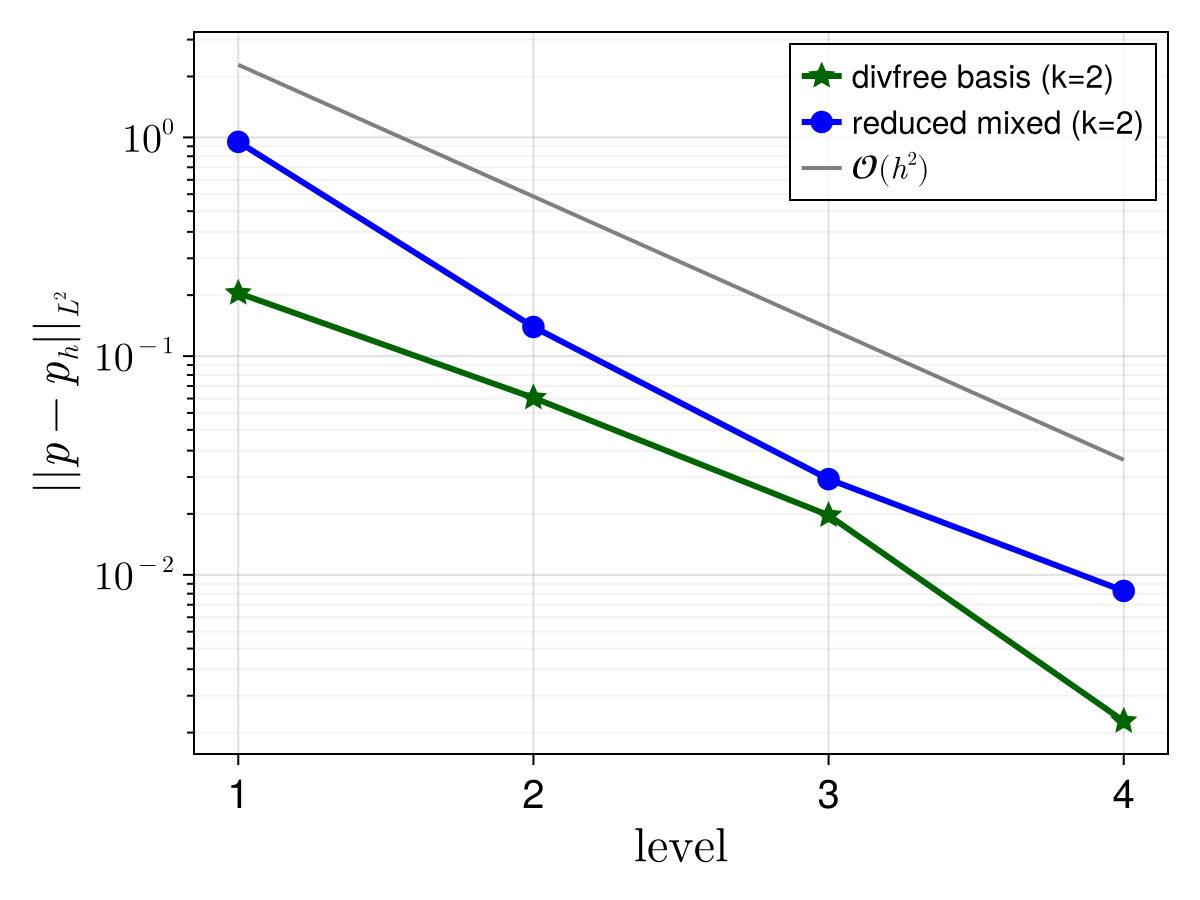}
    \includegraphics[width=0.3\textwidth]{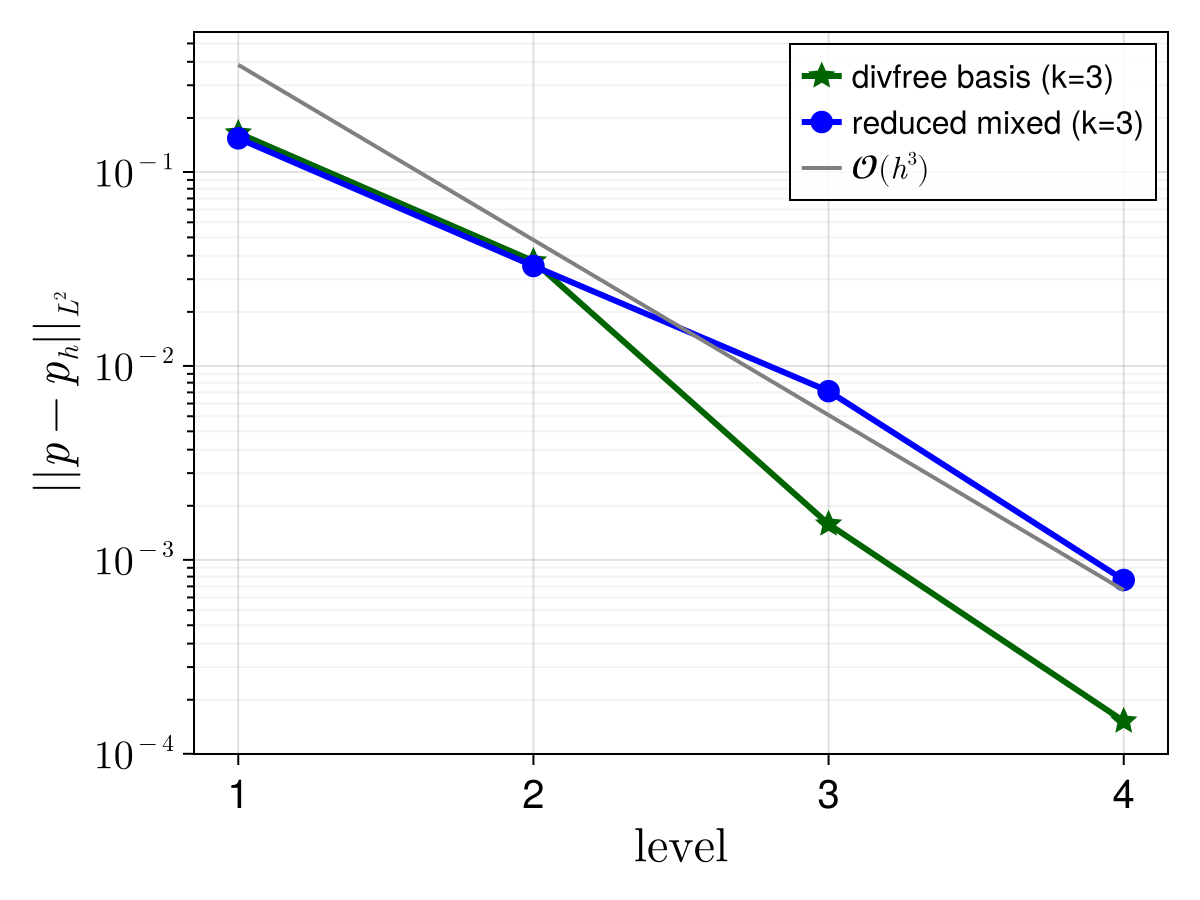}}
    \caption{Example~\ref{ex:3d}. $L^2(\Omega)$ velocity error (top) and pressure error (bottom) for the polynomial degree $k=1$ (left), $k=2$ (center), and $k=3$ (right), $\delta = -1$ (symmetric case).} \label{fig:3d_conv_sym}
\end{figure}

The behavior of velocity and pressure errors is depicted in Figures~\ref{fig:3d_conv_nonsym} and~\ref{fig:3d_conv_sym}. Like in the two-dimensional example, 
the velocity errors computed with both methods are virtually the same, 
while the pressure approximations for the decoupled methods show a better accuracy than those computed with the reduced mixed methods.

\begin{figure}[t!]
\centerline{
    \includegraphics[width=0.3\textwidth]{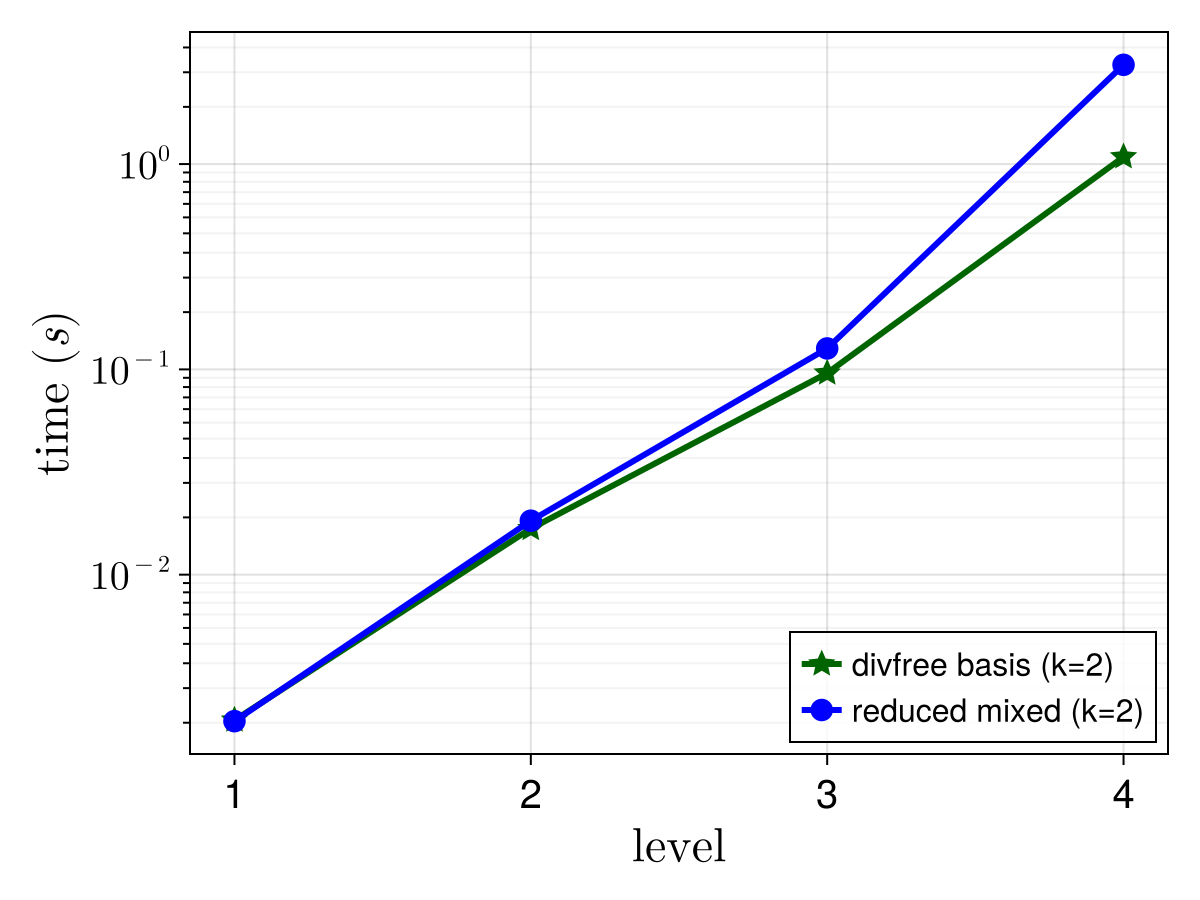}
    \includegraphics[width=0.3\textwidth]{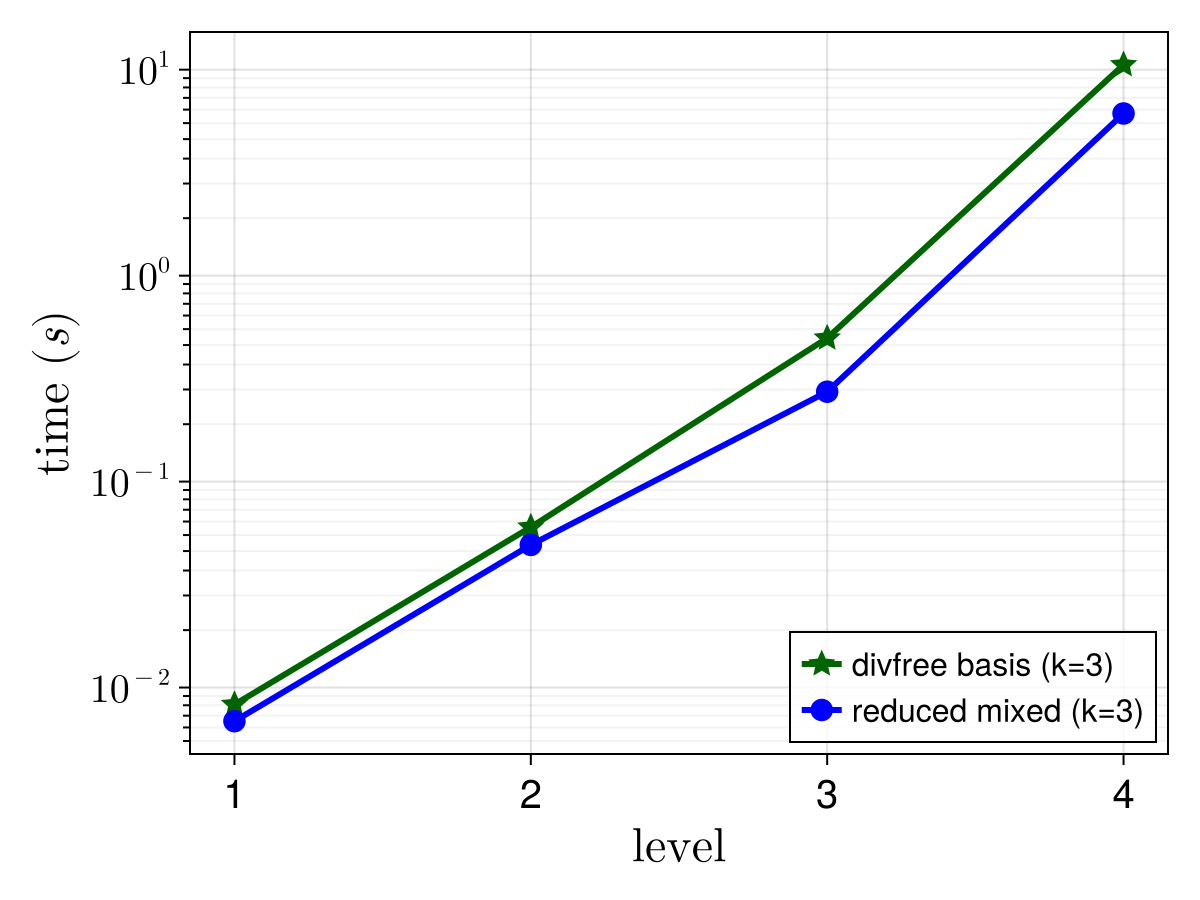}}
    \caption{Example~\ref{ex:3d}. Solver times with Pardiso for polynomial degree $k=2$ (left) and $k=3$ (right), $\delta = 1$ (skew-symmetric case).}
    \label{fig:3d_solvertimes_nonsym}
\end{figure} 
 
\begin{figure}[t!]
\centerline{
    \includegraphics[width=0.3\textwidth]{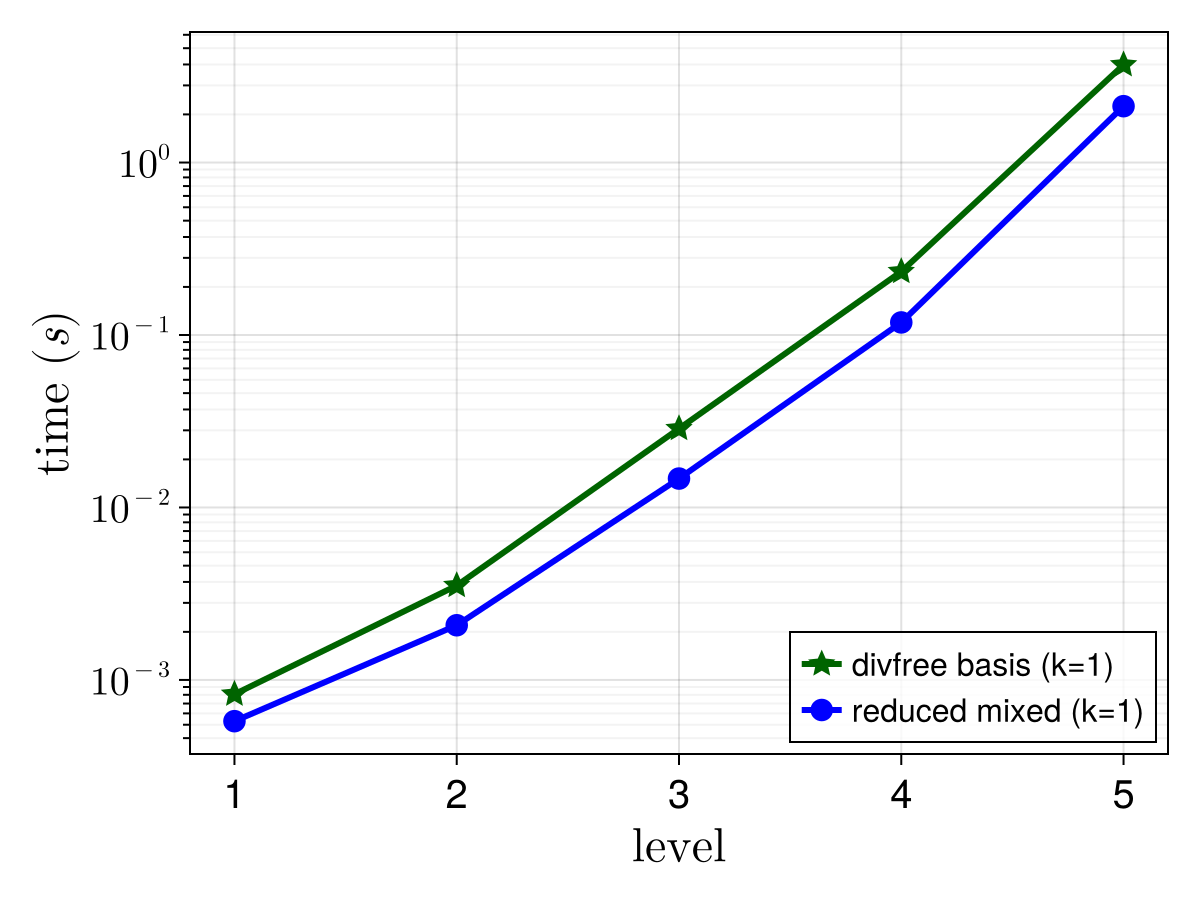}
    \includegraphics[width=0.3\textwidth]{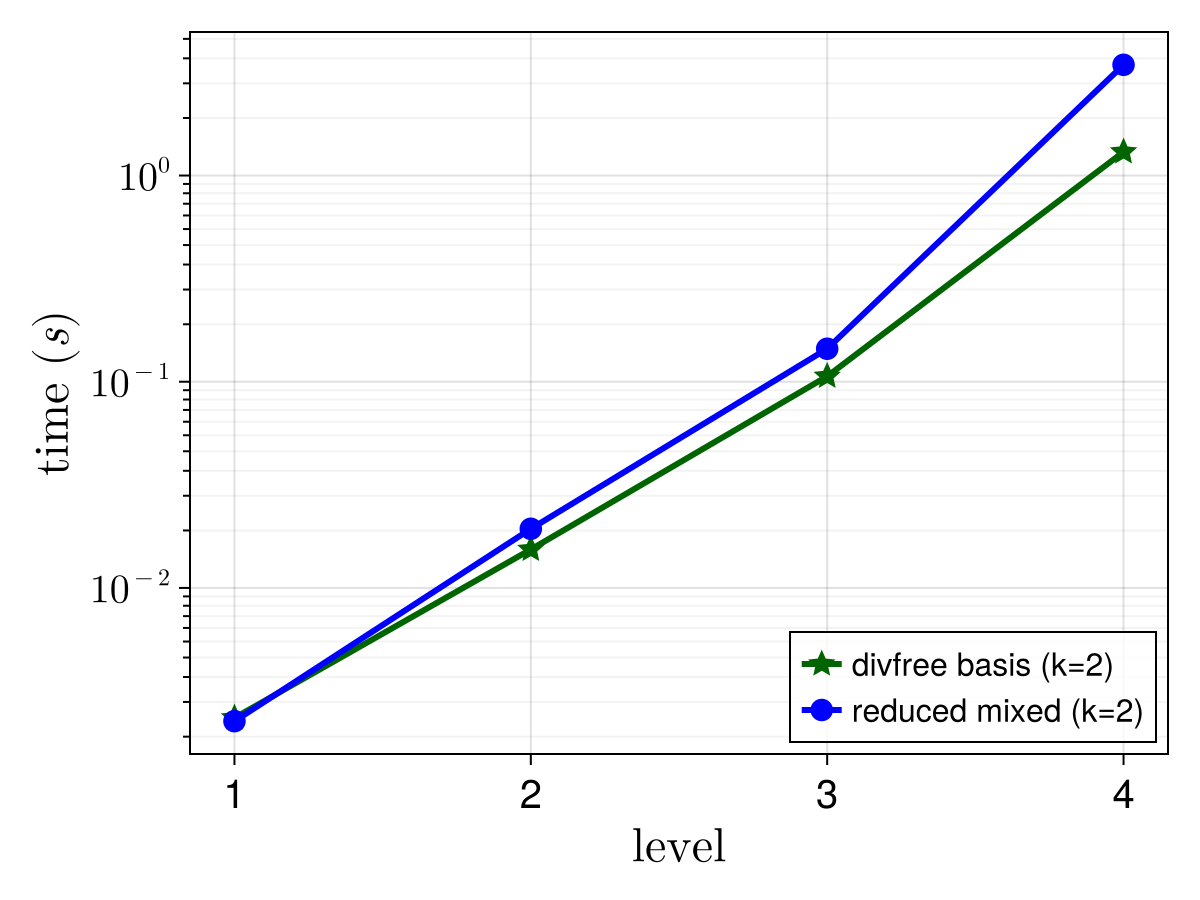}
    \includegraphics[width=0.3\textwidth]{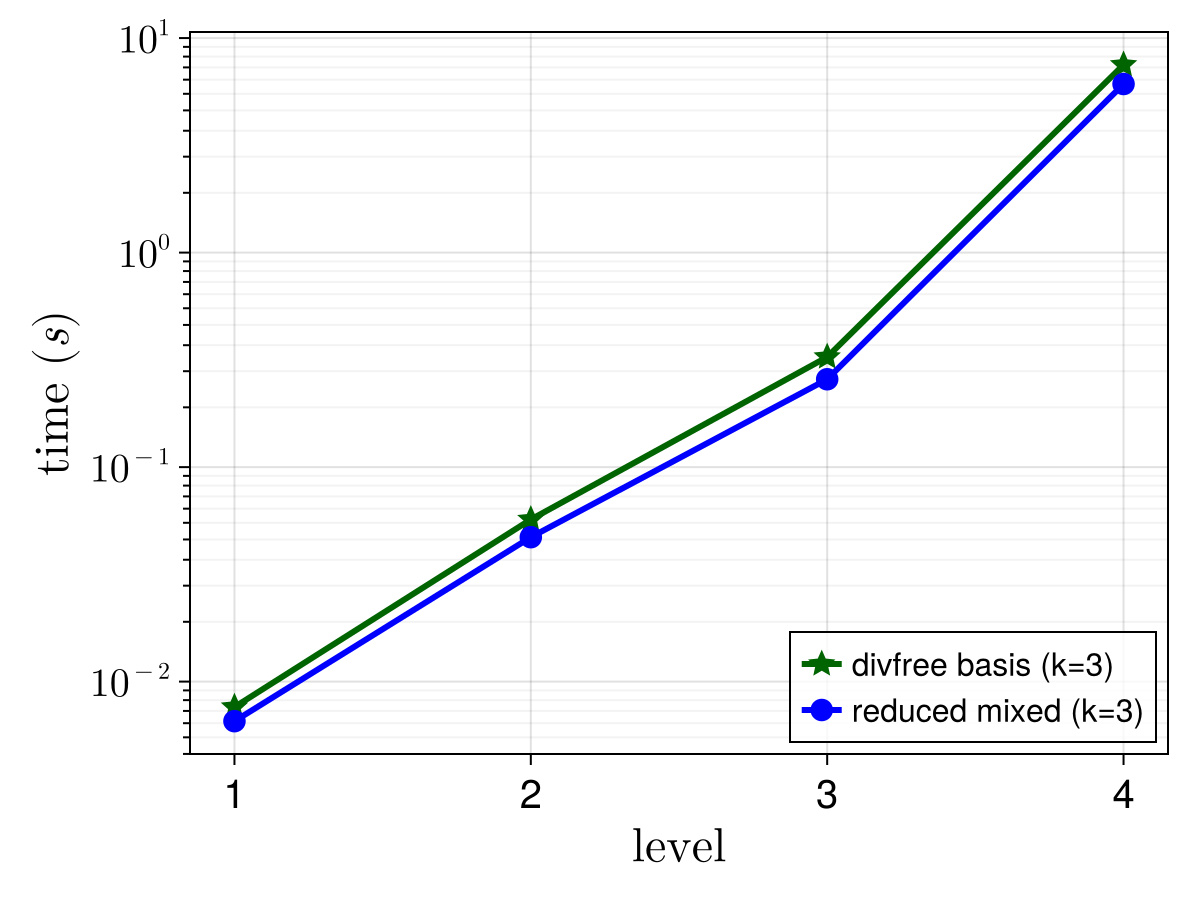}}
    \caption{Example~\ref{ex:3d}. Solver times with Pardiso for polynomial degree $k=1$ (left), $k=2$ (center), and $k=3$ (right), $\delta = -1$ (symmetric case).}\label{fig:3d_solvertimes_sym}
\end{figure} 

Finally, computing times are provided in Figures~\ref{fig:3d_solvertimes_nonsym} and~\ref{fig:3d_solvertimes_sym}.
In most cases, the reduced mixed methods are somewhat faster than the decoupled methods. However, for $k=2$, 
which is a very common choice in practice, the decoupled methods are more efficient.

%% file: tables/2d/table_order=2_levels=5.txt
\begin{tabular}{ccccccc}
\toprule
level & ndofs dfb (k=2) & ndofs red (k=2) & ndofs full (k=2) & nnz dfb (k=2) & nnz red (k=2) & nnz full (k=2)\\
\midrule
1 & 196 & 204 & 340 & 4152 & 3954 & 4588\\
2 & 961 & 1040 & 1816 & 25761 & 22990 & 30008\\
3 & 3829 & 4194 & 7426 & 106975 & 95380 & 126734\\
4 & 14744 & 16265 & 29037 & 418918 & 375183 & 500821\\
5 & 57805 & 63997 & 114721 & 1656779 & 1486953 & 1990197\\
\bottomrule
\end{tabular}

%% file: tables/2d/table_order=3_levels=5.txt
\begin{tabular}{ccccccc}
\toprule
level & ndofs dfb (k=3) & ndofs red (k=3) & ndofs full (k=3) & nnz dfb (k=3) & nnz red (k=3) & nnz full (k=3)\\
\midrule
1 & 382 & 390 & 662 & 12588 & 12182 & 18574\\
2 & 1965 & 2044 & 3596 & 74355 & 68232 & 108096\\
3 & 7945 & 8310 & 14774 & 309561 & 282480 & 450990\\
4 & 30858 & 32379 & 57923 & 1218462 & 1112093 & 1781891\\
5 & 121505 & 127697 & 229145 & 4831089 & 4408413 & 7076425\\
\bottomrule
\end{tabular}

%% file: tables/2d/table_order=4_levels=4.txt
\begin{tabular}{ccccccc}
\toprule
level & ndofs dfb (k=4) & ndofs red (k=4) & ndofs full (k=4) & nnz dfb (k=4) & nnz red (k=4) & nnz full (k=4)\\
\midrule
1 & 636 & 644 & 1086 & 29958 & 28314 & 43886\\
2 & 3357 & 3436 & 5958 & 171319 & 159246 & 253858\\
3 & 13677 & 14042 & 24546 & 711305 & 660364 & 1058620\\
4 & 53358 & 54879 & 96388 & 2801492 & 2602693 & 4183050\\
\bottomrule
\end{tabular}

%% file: tables/2d/table_order=1_levels=7_symmetric.txt
\begin{tabular}{ccccccc}
\toprule
level & ndofs dfb (k=1) & ndofs red (k=1) & ndofs full (k=1) & nnz dfb (k=1) & nnz red (k=1) & nnz full (k=1)\\
\midrule
1 & 78 & 86 & 145 & 980 & 676 & 853\\
2 & 345 & 424 & 732 & 6201 & 4404 & 5328\\
3 & 1329 & 1694 & 2944 & 25847 & 18566 & 22316\\
4 & 5016 & 6537 & 11401 & 101328 & 73393 & 87985\\
5 & 19467 & 25659 & 44828 & 400971 & 291609 & 349116\\
6 & 77094 & 102213 & 178727 & 1603802 & 1168725 & 1398267\\
7 & 306741 & 407826 & 713404 & 6411081 & 4676514 & 5593248\\
\bottomrule
\end{tabular}

%% file: tables/2d/table_order=2_levels=5_symmetric.txt
\begin{tabular}{ccccccc}
\toprule
level & ndofs dfb (k=2) & ndofs red (k=2) & ndofs full (k=2) & nnz dfb (k=2) & nnz red (k=2) & nnz full (k=2)\\
\midrule
1 & 196 & 204 & 340 & 4564 & 3994 & 4656\\
2 & 961 & 1040 & 1816 & 26597 & 23008 & 30396\\
3 & 3829 & 4194 & 7426 & 109871 & 95382 & 128350\\
4 & 14744 & 16265 & 29037 & 430134 & 375227 & 507207\\
5 & 57805 & 63997 & 114721 & 1700653 & 1486965 & 2015559\\
\bottomrule
\end{tabular}

%% file: tables/2d/table_order=3_levels=5_symmetric.txt
\begin{tabular}{ccccccc}
\toprule
level & ndofs dfb (k=3) & ndofs red (k=3) & ndofs full (k=3) & nnz dfb (k=3) & nnz red (k=3) & nnz full (k=3)\\
\midrule
1 & 382 & 390 & 662 & 13058 & 12182 & 18778\\
2 & 1965 & 2044 & 3596 & 75469 & 68232 & 109260\\
3 & 7945 & 8310 & 14774 & 313051 & 282480 & 455838\\
4 & 30858 & 32379 & 57923 & 1230714 & 1112093 & 1801049\\
5 & 121505 & 127697 & 229145 & 4875581 & 4408413 & 7152511\\
\bottomrule
\end{tabular}

%% file: tables/2d/table_order=4_levels=4_symmetric.txt
\begin{tabular}{ccccccc}
\toprule
level & ndofs dfb (k=4) & ndofs red (k=4) & ndofs full (k=4) & nnz dfb (k=4) & nnz red (k=4) & nnz full (k=4)\\
\midrule
1 & 636 & 644 & 1086 & 30100 & 28314 & 44294\\
2 & 3357 & 3436 & 5958 & 171949 & 159246 & 256186\\
3 & 13677 & 14042 & 24546 & 713821 & 660364 & 1068316\\
4 & 53358 & 54879 & 96388 & 2811244 & 2602693 & 4221366\\
\bottomrule
\end{tabular}

%% file: tables/3d/table_order=2_levels=4.txt
\begin{tabular}{ccccc}
\toprule
level & ndofs dfb (k=2) & ndofs red (k=2) & nnz dfb (k=2) & nnz red (k=2)\\
\midrule
1 & 277 & 259 & 15134 & 18933\\
2 & 1312 & 1239 & 94027 & 135001\\
3 & 8647 & 8339 & 756249 & 1190877\\
4 & 52994 & 51805 & 5028568 & 8308893\\
\bottomrule
\end{tabular}

%% file: tables/3d/table_order=3_levels=4.txt
\begin{tabular}{ccccc}
\toprule
level & ndofs dfb (k=3) & ndofs red (k=3) & nnz dfb (k=3) & nnz red (k=3)\\
\midrule
1 & 640 & 622 & 64883 & 58728\\
2 & 3262 & 3189 & 395725 & 350243\\
3 & 23161 & 22853 & 3260511 & 2882731\\
4 & 147080 & 145891 & 22085754 & 19540909\\
\bottomrule
\end{tabular}

%% file: tables/3d/table_order=1_levels=5_symmetric.txt
\begin{tabular}{ccccc}
\toprule
level & ndofs dfb (k=1) & ndofs red (k=1) & nnz dfb (k=1) & nnz red (k=1)\\
\midrule
1 & 100 & 82 & 2469 & 784\\
2 & 442 & 369 & 14833 & 5195\\
3 & 2737 & 2429 & 124101 & 46959\\
4 & 16334 & 15145 & 824112 & 327725\\
5 & 109681 & 104831 & 5886385 & 2412889\\
\bottomrule
\end{tabular}

%% file: tables/3d/table_order=2_levels=4_symmetric.txt
\begin{tabular}{ccccc}
\toprule
level & ndofs dfb (k=2) & ndofs red (k=2) & nnz dfb (k=2) & nnz red (k=2)\\
\midrule
1 & 277 & 259 & 15254 & 19067\\
2 & 1312 & 1239 & 94053 & 135001\\
3 & 8647 & 8339 & 756255 & 1190877\\
4 & 52994 & 51805 & 5028578 & 8308893\\
\bottomrule
\end{tabular}

%% file: tables/3d/table_order=3_levels=4_symmetric.txt
\begin{tabular}{ccccc}
\toprule
level & ndofs dfb (k=3) & ndofs red (k=3) & nnz dfb (k=3) & nnz red (k=3)\\
\midrule
1 & 640 & 622 & 65431 & 58728\\
2 & 3262 & 3189 & 398321 & 350243\\
3 & 23161 & 22853 & 3280405 & 2882731\\
4 & 147080 & 145891 & 22211122 & 19540909\\
\bottomrule
\end{tabular}

%% file: summary.tex
\section{Summary and outlook}

This paper proposed finite element methods for the incompressible Stokes equations
with a basis of the velocity space consisting of divergence-free functions. In this way, 
it is possible to compute the velocity approximation without needing to solve a coupled
velocity-pressure system, possibly of saddle point type. Constructions of the basis 
for arbitrary polynomial degree in two and three space dimensions for simplicial meshes 
were provided. Other algorithmic issues, like the incorporation of non-homogeneous 
Dirichlet boundary conditions and the computation of a finite element pressure approximation
were addressed and corresponding algorithms were proposed. The finite element error analysis
for the velocity from \cite{JLMR2022} carries over to the proposed method and an estimate for the 
gradient of the pressure error was derived. Numerical studies 
showed the expected orders of convergence. Using a sparse direct solver, the proposed methods 
proved to be consistently faster than the full mixed methods from \cite{JLMR2022} and in some situations
even more efficient than the reduced mixed methods derived in \cite{JLMR2022}.

The extension of the proposed methods to the incompressible Navier--Stokes equations can be 
performed along the lines of \cite{AJLM24}, i.e., special discrete forms of the 
nonlinear convective term have to be  considered. A further aspect of interest is the efficiency 
of the methods, in particular in comparison with the reduced mixed methods from \cite{JLMR2022}, 
if appropriate iterative solvers for the arising linear systems of equations are applied.